\documentclass[envcountsame]{svmult}

\usepackage{mathptmx}       
\usepackage{helvet}         
\usepackage{courier}        
\usepackage{type1cm}        
\usepackage{amsmath,amsfonts,amssymb}
\usepackage{makeidx}         
\usepackage{graphicx}        
\usepackage{multicol}        
\usepackage{mathrsfs}

\usepackage{eucal}
\usepackage{epic,eepic}

\makeindex

\usepackage{sdpgf}[2009/12/02]

\usepackage{url}

\DeclareMathOperator{\linSpan}{LinSpan}

\newcommand*{\mc}[1]{\mathcal{#1}}
\newcommand*{\meta}[1]{\ensuremath{\langle\text{\textit{#1}}\rangle}}
\newcommand*{\mf}[1]{\mathfrak{#1}}

\newcommand*{\parenthetic}[1]{\/\textup{(#1)}}
\newcommand*{\pin}[1]{
   \mathchoice{%
      \mathrel{\mathrm{in}}%
   }{%
      \mathrel{\mathrm{in}}%
   }{%
      \mathop{\mathrm{in}}%
   }{%
      \mathop{\mathrm{in}}%
   }#1%
}

\makeatletter
\newcommand*{\setOf}[3][\@gobble]{%
   \left\{ \, #2 \,\,\vrule\relax#1.\,\, #3 \, \right\}%
}
\makeatother

\DeclareMathSymbol{\txthyphen}{\mathalpha}{operators}{`\-}

\makeatletter
\newcommand{\xfmapsto}[2][1fill]{%
  \sbox\z@{$\scriptstyle #2$}%
  \skip@=\wd\z@ \relax
  \ifdim 0.4em<\skip@
    \addtolength\skip@{0.4em plus #1}%
  \else
    \setlength\skip@{0.8em plus #1}%
  \fi
  \divide\skip@ by 2%
  \;
  \mathord- \mkern-6mu%
  \cleaders\hbox{$\mkern-2mu \mathord- \mkern-2mu$}\hskip\skip@
  \dimen@=1ex%
  \advance\dimen@ \dp\z@
  \raise\dimen@ \hb@xt@\z@{\hss\unhbox\z@\hss}%
  \hb@xt@\z@{\hss$-$\hss}%
  \mkern-6mu
  \cleaders\hbox{$\mkern-2mu \mathord- \mkern-2mu$}\hskip\skip@
  \mkern-6mu \mathord\rightarrow \;
}
\makeatother

\newcommand{\tone}{{
\begin{picture}(2,10)(0,3)
\drawline(1,0)(1,10)
\put(1,10){\circle*{2}}
\end{picture}
}}

\newcommand{\ttwo}{{
\begin{picture}(16,14)(0,5)
\drawline(8,0)(8,6)(0,14)
\drawline(8,6)(16,14)
\put(0,14){\circle*{2}}
\put(16,14){\circle*{2}}
\put(8,6){\circle*{2}}
\end{picture}
}}

\newcommand{\tthreeone}{{
\begin{picture}(16,14)(0,5)
\drawline(8,0)(8,6)(0,14)
\drawline(8,6)(16,14)
\drawline(12,10)(8,14)
\put(0,14){\circle*{2}}
\put(16,14){\circle*{2}}
\put(8,14){\circle*{2}}
\put(8,6){\circle*{2}}
\put(12,10){\circle*{2}}
\end{picture}
}}

\newcommand{\tthreetwo}{{
\begin{picture}(16,14)(0,5)
\drawline(8,0)(8,6)(0,14)
\drawline(4,10)(8,14)
\drawline(8,6)(16,14)
\put(0,14){\circle*{2}}
\put(16,14){\circle*{2}}
\put(8,14){\circle*{2}}
\put(8,6){\circle*{2}}
\put(4,10){\circle*{2}}
\end{picture}
}}

\newcommand{\tfourone}{{
\begin{picture}(18,15)
\drawline(9,-4)(9,2)(0,11)
\drawline(9,2)(18,11)
\drawline(12,5)(6,11)
\drawline(15,8)(12,11)
\put(0,11){\circle*{2}}
\put(18,11){\circle*{2}}
\put(6,11){\circle*{2}}
\put(12,11){\circle*{2}}
\put(9,2){\circle*{2}}
\put(12,5){\circle*{2}}
\put(15,8){\circle*{2}}
\end{picture}
}}

\newcommand{\tfourtwo}{{
\begin{picture}(18,15)
\drawline(9,-4)(9,2)(0,11)
\drawline(9,2)(18,11)
\drawline(12,5)(6,11)
\drawline(9,8)(12,11)
\put(0,11){\circle*{2}}
\put(18,11){\circle*{2}}
\put(6,11){\circle*{2}}
\put(12,11){\circle*{2}}
\put(9,2){\circle*{2}}
\put(12,5){\circle*{2}}
\put(9,8){\circle*{2}}
\end{picture}
}}

\newcommand{\tfourthree}{{
\begin{picture}(18,15)
\drawline(9,-4)(9,2)(0,11)
\drawline(3,8)(6,11)
\drawline(9,2)(18,11)
\drawline(15,8)(12,11)
\put(0,11){\circle*{2}}
\put(18,11){\circle*{2}}
\put(6,11){\circle*{2}}
\put(12,11){\circle*{2}}
\put(9,2){\circle*{2}}
\put(3,8){\circle*{2}}
\put(15,8){\circle*{2}}
\end{picture}
}}

\newcommand{\tfourfour}{{
\begin{picture}(18,15)
\drawline(9,-4)(9,2)(0,11)
\drawline(6,5)(12,11)
\drawline(9,8)(6,11)
\drawline(9,2)(18,11)
\put(0,11){\circle*{2}}
\put(18,11){\circle*{2}}
\put(6,11){\circle*{2}}
\put(12,11){\circle*{2}}
\put(9,2){\circle*{2}}
\put(6,5){\circle*{2}}
\put(9,8){\circle*{2}}
\end{picture}
}}

\newcommand{\tfourfive}{{
\begin{picture}(18,15)
\drawline(9,-4)(9,2)(0,11)
\drawline(3,8)(6,11)
\drawline(6,5)(12,11)
\drawline(9,2)(18,11)
\put(0,11){\circle*{2}}
\put(18,11){\circle*{2}}
\put(6,11){\circle*{2}}
\put(12,11){\circle*{2}}
\put(9,2){\circle*{2}}
\put(6,5){\circle*{2}}
\put(3,8){\circle*{2}}
\end{picture}
}}

\newcommand{\psiwedgephi}{{
\begin{picture}(20,18)(-2,5)
\drawline(8,0)(8,6)(0,14)
\drawline(8,6)(16,14)
\put(8,6){\circle*{2}}
\put(-1,18){\makebox(0,0){$\psi$}}
\put(19,18){\makebox(0,0){$\varphi$}}
\end{picture}}}

\DeclareMathVersion{normal2}

\begin{document}

\mathversion{normal2}

\title*{Universal Algebra Applied
  to Hom-Associative Algebras, and More}
  \titlerunning{Universal Algebra Applied
  to Hom-Associative Algebras, and More}

\author{Lars Hellstr\"om \and Abdenacer Makhlouf \and  Sergei D. Silvestrov}
\institute{
  Lars Hellstr\"om \at
    Division of Applied Mathematics,
    The School of Education, Culture and Communication,\\
    M{\"a}lardalen University, Box 883, 721 23 V{\"a}ster{\aa}s,
    Sweden,
    \email{Lars.Hellstrom@residenset.net}
  \and
  Abdenacer Makhlouf \at
    University of Haute Alsace, LMIA,
    4 rue des Fr\`eres Lumi\`ere 68093 Mulhouse, France
    \email{Abdenacer.Makhlouf@uha.fr}
  \and
  Sergei D. Silvestrov \at
    Division of Applied Mathematics,
    The School of Education, Culture and Communication,\\
    M{\"a}lardalen University, Box 883, 721 23 V{\"a}ster{\aa}s,
    Sweden,
    \email{sergei.silvestrov@mdh.se}%
}

\maketitle

\abstract*{
  The purpose of this paper is to discuss the universal algebra 
  theory of hom-algebras. This kind of algebra involves a linear map 
  which twists the usual identities. We focus on hom-associative 
  algebras and hom-Lie algebras for which we review the main results. 
  We discuss the envelopment problem, operads, and the Diamond Lemma; 
  the usual tools have to be adapted to this new situation. Moreover 
  we study Hilbert series for the hom-associative operad and free 
  algebra, and describe them up to total degree equal $8$ and $9$ 
  respectively.}
\abstract{
  The purpose of this paper is to discuss the universal algebra 
  theory of hom-algebras. This kind of algebra involves a linear map 
  which twists the usual identities. We focus on hom-associative 
  algebras and hom-Lie algebras for which we review the main results. 
  We discuss the envelopment problem, operads, and the Diamond Lemma; 
  the usual tools have to be adapted to this new situation. Moreover 
  we study Hilbert series for the hom-associative operad and free 
  algebra, and describe them up to total degree equal $8$ and $9$ 
  respectively.}
%
\section*{Introduction}

Abstract algebra is a subject that may be investigated on many
different levels of maturity. At the most elementary level that still
meets the standards of mathematical rigor, the investigator simply
postulates some set of axioms (usually in the form of a definition)
and then goes on to derive random consequences of these axioms,
hopefully topping it off with examples to illustrate the range of
possible outcomes for the results that are stated (as there have been
some spectacular instances of mathematical theories that died due to
having no nontrivial examples where they were applicable). This level
of investigation may produce a nicely whole theory of something, but
in the hands of an immature investigator it runs a significant risk
of ending up as a random collection of facts that don't combine to
anything greater than themselves; the whole of a good theory should
be greater than the sum of its parts.

One way of reaching a higher level can be to investigate matters
using the techniques of universal algebra, since these combine
looking at concrete examples with the generality of investigating the
generic case. Another way is to employ the language of category
theory to investigate matters on a level that is even more abstract.
Indeed, category theory has become so fashionable that modern
presentations of universal algebra may treat it as a mere application
of the categorical formalism. This has the advantage of allowing
definitions of for example free algebras to be given that do not
presuppose a specific construction machinery, but on the other hand
it runs the risk of losing itself in the heavens of abstraction,
because the difficulties have been postponed rather than taken care
of; doing any nontrivial example may bring them all back with a
vengeance. Therefore we were glad to see how Yau
in~\cite{Yau:EnvLieAlg} would proceed from an abstract categorical
definition to concrete constructions of many free algebras of
relevance to hom-associative and hom-Lie algebras---glad, but also a
bit curious as to why the constructions were not more systematic.

For better or worse, there is probably a simple reason for someone
doing \emph{ad hoc} constructions rather than the standard systematic
ones here: even though the systematic constructions are well known
within the Formal languages, Logic, and Discrete mathematics
communities, they are \emph{not} so within the Algebra community. 
Therefore one aim for us in writing this paper has been to bring to 
the attention of the Algebra community this veritable treasure-trove 
of methods and techniques that universal algebra and formal languages 
have to offer. Another aim was of course to find out more about 
hom-algebras, as what as come so far is only the beginning of the 
exploration of these.


The first motivation to study nonassociative hom-algebras comes from
quasi-deformations of Lie algebras of vector fields, in particular
$q$-deformations of Witt and Virasoro algebras
\cite{AizawaSaito,ChaiElinPop,ChaiKuLukPopPresn,ChaiIsKuLuk,ChaiPopPres,
CurtrZachos1,DaskaloyannisGendefVir, Kassel1,LiuKeQin,Hu}. The
deformed algebras arising in connection with $\sigma$-derivation are
no longer Lie algebras. It was observed in the pioneering works
 that in these
examples a twisted Jacobi identity holds. Motivated by these
examples and their generalisation on the one hand, and the desire to
be able to treat within the same framework such well-known
generalisations of Lie algebras as the color and  Lie superalgebras
on the other hand, quasi-Lie algebras and subclasses of
quasi-hom-Lie algebras and hom-Lie algebras were introduced by
Hartwig, Larsson and Silvestrov in \cite{HLS,LS1,LS2,LS3}.

The hom-associative algebras play the role of associative algebras in
the hom-Lie setting. They were introduced by Makhlouf and Silvestrov
in \cite{MS}. Usual functors between the categories of Lie algebras
and associative algebras were extended to hom-setting, see
\cite{Yau:EnvLieAlg} for the construction of the enveloping algebra of
a hom-Lie algebra. Likewise, many classical structures as alternative,
Jordan, Malcev, graded algebras and $n$-ary algebras of Lie and
associative type, were considered in this framework, see
\cite{AEM,AmmarMakhlouf2009,AMS2008,BenayadiMakhlouf,Canepl2009,
  Mak:Almeria,HomHopf,MakSil-def,HomAlgHomCoalg,Makhlouf-Yau,Sheng,
  Yau:YangBaxter,Yau:comodule,Yau:YangBaxter2,Yau:ClassicYangBaxter,
  Yau:MalsevAlternJordan}.
Notice that Hom-algebras over a PROP were defined and studied
in~\cite{Yau:Prop} and deformations of hom-type of the Associative
operad from the point of view of the confluence property discussed
in~\cite{Iyudu-Makhlouf}.

The main feature of all these algebras is that classical identities
are twisted by a homomorphism. Pictorially, drawing the
multiplication $\mathsf{m}$ as a circle and the linear map $\alpha$ as a
square, \index{hom-associativity} hom-associativity may be written as
\begin{equation}
  \left[ \begin{sdpgf}{0}{0}{120}{-196}{0.2pt}
    \m 60 -191 \L 0 41 \S \m 30 -46 \L 0 41 \S \m 79 -51 \C -12 12 -7
    18 0 16 \S \m 101 -51 \C 14 14 -25 17 0 15 \S \m 49 -123 \C -12 12
    -7 17 0 16 \S \m 71 -123 \C 12 12 7 17 0 16 \S \ov 44 -150 32 32
    \S \re 14 -78 32 32 \S \ov 74 -78 32 32 \S
  \end{sdpgf} \right]
  \equiv
  \left[ \begin{sdpgf}{0}{0}{120}{-196}{0.2pt}
    \m 60 -191 \L 0 41 \S \m 19 -51 \C -14 14 25 17 0 15 \S \m 41 -51
    \C 12 12 7 18 0 16 \S \m 90 -46 \L 0 41 \S \m 49 -123 \C -12 12 -7
    17 0 16 \S \m 71 -123 \C 12 12 7 17 0 16 \S \ov 44 -150 32 32 \S
    \ov 14 -78 32 32 \S \re 74 -78 32 32 \S
  \end{sdpgf} \right]
\end{equation}

In this paper, we summarize the basics of hom-algebras in the first
section. We emphasize on hom-associative and hom-Lie algebras. We show
first the paradigmatic example of $q$-deformation of
$\mathfrak{sl}_2$ using $\sigma$-derivations, leading to an
interesting example of hom-Lie algebra.  We provide the general method
and some other procedures to construct examples of hom-associative or
hom-Lie algebras. We describe the free hom-nonassociative algebra
constructed by Yau. It leads to free hom-associative algebra and to the
enveloping algebra of a hom-Lie algebra.  In
Section~\ref{Sec:ClassicalUA} we recall the basic concepts in
universal algebra as signature $\Omega$, $\Omega$-algebra, formal
terms, normal form, rewriting system, and quotient algebra. We
emphasize on hom-associative algebras and discuss the envelopment
problem. Section~\ref{Sec:OperadUA} is devoted to operadic
approach. We discuss this concept and universal algebra for operads.
We provide a diamond lemma for operads and discuss ambiguities for
symmetric operads. Then we focus on hom-associative algebras operad
for which attempt to resolve the ambiguities.  Likewise we study
congruence modulo hom-associativity and Hilbert series in this case.
Moreover we study Hilbert series for the hom-associative operad 
and compute several dozen terms of it exactly using techniques from 
formal languages (notably regular tree languages).

\section{Hom-algebras: definitions, constructions and examples}
\label{Sec:Hom-algebras}

We summarize in this section the basics about hom-associative
algebras and hom-Lie algebras.

The hom-associative identity \( \alpha(x) \cdot (y \cdot\nobreak z) =
(x \cdot\nobreak y) \cdot \alpha(z)\) is a generalisation of the
ordinary associative identity \( x \cdot (y \cdot\nobreak z) =
(x \cdot\nobreak y) \cdot z\). Study of it could be motivated simply
by the creed that ``one should always generalise'', and in
Subsection~\ref{Ssec:Hilbert} we will briefly consider the view that
hom-associativity (in a rather abstract setting) can be considered as
homogenisation of ordinary associativity, but historically the
hom-associative identity was first suggested by an application; the
line of thought went from $\sigma$-derivations, then to hom-\emph{Lie}
algebras, before finally touching upon hom-associative algebras. We
sketch the $\sigma$-derivation development in the first subsection
below, but the rest of the text does not depend on the material
presented there, so the reader who prefers to skip to
Subsection~\ref{Ssec:Hom-ass-examples} now should have no problem
doing so.

\subsection{$q$-Deformations and $\sigma$-derivations}
\label{Ssec:sigma-derivations}

Let $A$ be an associative $\mathbb{K}$-algebra with unity $1$.
Let $\sigma$ be an endomorphism on $A$. By a  \index{twisted derivation} \emph{twisted derivation}
or \index{$\sigma$-derivation} \emph{$\sigma$-derivation} on $A$, we mean a $\mathbb{K}$-linear map
\(\Delta\colon A \longrightarrow A\) such that a $\sigma$-twisted
product rule (Leibniz rule) holds:
\begin{equation}\label{eq:sigmaLeibniz}
  \Delta(ab) = \Delta(a) b + \sigma(a) \Delta(b)
  \text{.}
\end{equation}
The ordinary derivative $(\partial\,a)(t) = a'(t)$ on the polynomial
ring \(A = \mathbb{K}[t]\) is a $\sigma$-derivation for \(\sigma = \mathrm{id}\). If
on a superalgebra \(A = A_0 \oplus A_1\) one defines \(\sigma(a) = a\)
for \(a \in A_0\) but \(\sigma(a) = -a\) for \(a \in A_1\), then
\eqref{eq:sigmaLeibniz} precisely captures the parity adjustments of
the product rule that derivations in such settings tend to exhibit,
and it does so in a manner that unifies the even and odd cases.
Returning to the the polynomial ring \(A = \mathbb{K}[t]\), the
$\sigma$-derivation concept offers a unified framework for various
derivation-\emph{like} operators, perhaps most famously the Jackson
$q$-derivation operator \((D_q a)(t) = \frac{1}{(q-1)t} \bigl( a(qt)
-\nobreak a(t) \bigr)\) for some \(q \in \mathbb{K}\), that has the ordinary
derivative as the \(q \to 1\) limit and the product rule \(D_q(ab)(t)
= D_q(a)(t) \, b(t) + a(qt) \, D_q(b)(t)\); this is thus a
$\sigma$-derivation for \(\sigma(a)(t) = a(qt)\), which acts on the
standard basis for $\mathbb{K}[t]$ as \(\sigma(t^n) = q^n t^n\).
(See~\cite{HS} and references therein.)

The big algebraic insight about derivations is that they form Lie
algebras, from which one can go on to universal enveloping algebras
and exploit the connections to formal groups and Lie groups. What
about twisted derivations, then? A quick calculation will reveal that
they do not form a Lie algebra in the usual way, but there can still
be a Lie-algebra-like structure on them.

We let $\mathfrak{D}_\sigma(A)$ denote the set of $\sigma$-derivations
on $A$. As with vector fields in differential geometry, one may
define the product of some \(a \in A\) and \(\Delta \in
\mathfrak{D}_\sigma(A)\) to be the \(a \cdot \Delta \in
\mathfrak{D}_\sigma(A)\) defined by \((a \cdot\nobreak \Delta)(b) = a
\, \Delta(b)\) for all \(b \in A\); hence $\mathfrak{D}_\sigma(A)$ can
be regarded as a left $A$-module.
The \emph{annihilator} $\mathrm{Ann}(\Delta)$ of some \(\Delta \in
\mathfrak{D}(A)\) is the set of all \(a \in A\) such that \(a \cdot
\Delta = 0\). By~\cite[Th.~4]{HLS}, if $A$ is a commutative unique
factorisation domain then $\mathfrak{D}_\sigma(A)$ is as a left
$A$-module free and of rank one, which lets us use the following
construction to exhibit a Lie-algebra-like structure on
$\mathfrak{D}_\sigma(A)$.

\begin{theorem}[{\cite[Th.~5]{HLS}}] \label{theorem:GenWitt}
  Let $A$ be a commutative associative $\mathbb{K}$-algebra with unit $1$ and
  let \(\sigma\colon A\longrightarrow A\) be an algebra homomorphism other than
  the identity map. Fix some \(\Delta \in \mathfrak{D}_\sigma(A)\)
  such that \(\sigma\bigl( \mathrm{Ann}(\Delta) \bigr) \subseteq
  \mathrm{Ann}(\Delta)\). Define a binary operation
  $[\cdot,\cdot]_\sigma$ on the left $A$-module $A \cdot \Delta$ by
  \begin{equation} \label{eq:GenWittProdDef}
    [ a\cdot \Delta, b\cdot \Delta]_\sigma =
    (\sigma(a)\cdot \Delta) \circ (b\cdot \Delta)
    - (\sigma(b)\cdot \Delta) \circ (a\cdot \Delta)
    \quad\text{for all \(a,b \in A\),}
  \end{equation}
  where $\circ$ denotes composition of functions. This operation is
  well-defined and satisfies the two identities
  \begin{align}\label{eq:GenWittProdFormula}
    [ a\cdot \Delta, b\cdot \Delta]_\sigma ={}&
      (\sigma(a) \Delta(b) - \sigma(b) \Delta(a))
      \cdot \Delta
      \text{,}\\
    [ b\cdot \Delta, a\cdot \Delta]_\sigma ={}&
      - [a\cdot \Delta, b\cdot \Delta]_\sigma
      \label{eq:GenWittSkew}
  \end{align}
  for all \(a,b \in A\). If there in addition is some \(\delta \in
  A\) such that
  \begin{equation}
    \Delta\bigl( \sigma(a) \bigr) =
    \delta \sigma\bigl( \Delta(a) \bigr)
    \qquad\text{for all \(a\in A\),}
      \label{eq:GenWittCond2}
  \end{equation}
  then $[\cdot,\cdot]_\sigma$ satisfies the deformed six-term Jacobi
  identity
  \begin{equation} \label{eq:GenWittJacobi}
    \circlearrowleft_{a,b,c} \, \bigl(
      [\sigma(a)\cdot \Delta,
        [b\cdot \Delta, c\cdot \Delta]_\sigma
      ]_\sigma
      + \delta\cdot[ a\cdot \Delta,
        [b\cdot \Delta, c\cdot \Delta]_\sigma
      ]_\sigma
    \bigr) = 0
  \end{equation}
  for all \(a,b,c \in A\).
\end{theorem}

The algebra $A\cdot \Delta$ in the theorem is then a
quasi-hom-Lie algebra with, in the notation of~\cite{LS1},
\(\alpha(a \cdot\nobreak \Delta) = \sigma(a) \cdot \Delta\),
\(\beta(a \cdot\nobreak \Delta) = (\delta a) \cdot \Delta\), and
$\omega = -\mathrm{id}_{A\cdot\Delta}$. For \(\delta \in \mathbb{K}\), as is the case
with \(\Delta = D_q\), \eqref{eq:GenWittJacobi} further simplifies to
the deformed three-term Jacobi identity \eqref{Eq:HomJacobi} of a
hom-Lie algebra.

As example of how the method in Theorem~\ref{theorem:GenWitt} ties in
with more basic deformation approaches, we review the results
in~\cite{LS2,LS3} concerned with this quasi-deformation scheme when
applied to the simple Lie algebra $\mathfrak{sl}_2(\mathbb{K})$.
Recall that the Lie algebra $\mathfrak{sl}_2(\mathbb{K})$ can be realized as a
vector space generated by elements $H$, $E$ and
$F$ with the bilinear bracket product defined by
the relations
\begin{align}\label{eq:sl2}
  [H,E] ={}& 2E \text{,}&
  [H,F] ={}& -2F \text{,} &
  [E,F] ={}& H \text{.}
\end{align}
A basic starting point is the following representation of $\mathfrak{sl}_2(\mathbb{K})$
in terms of first order differential operators acting on some vector
space of functions in a variable $t$:
\begin{align*} 
  E \mapsto{}& \partial \text{,}&
  H \mapsto{}& -2t\partial \text{,}&
  F \mapsto{}& -t^2\partial \text{.}
\end{align*}
To quasi-deform $\mathfrak{sl}_2(\mathbb{K})$ means that we firstly replace $\partial$ by
some twisted derivation $\Delta$ in this representation. At our
disposal as deformation parameters are now $A$ (the ``algebra of
functions'') and the endomorphism $\sigma$. After computing the bracket
on $A \cdot \Delta$ by Theorem~\ref{theorem:GenWitt} the relations in
the quasi-Lie deformation are obtained by pullback.

Let $A$ be a commutative, associative $\mathbb{K}$-algebra with unity $1$,
let $t$ be an element of $A$, and let $\sigma$ denote a $\mathbb{K}$-algebra
endomorphism on $A$. As above, $\mathfrak{D}_\sigma(A)$ denotes the
linear space of $\sigma$-derivations on $A$. Choose an element
$\Delta$ of $\mathfrak{D}_\sigma(A)$ and consider the
$\mathbb{K}$-subspace $A \cdot \Delta$ of elements on the form
$a\cdot \Delta$ for $a\in A$.
The elements \(e := \Delta\), \(h := -2t \cdot \Delta\), and
\(f := -t^2 \cdot \Delta\) span a $\mathbb{K}$-linear subspace
\[
  \mathcal{S} :=
  \linSpan_\mathbb{K} \{ \Delta, -2t \cdot \Delta, -t^2 \cdot \Delta\} =
  \linSpan_\mathbb{K}\{e,h,f\}
\]
of $A \cdot \Delta$. We restrict the multiplication
\eqref{eq:GenWittProdFormula} to $\mathcal{S}$ without, at this point,
assuming closure. Now, \(\Delta(t^2) = \Delta(t \cdot\nobreak t) =
\sigma(t) \Delta(t) + \Delta(t) t =
\bigl( \sigma(t) +\nobreak t \bigr) \Delta(t) \).
Under the natural (see~\cite{LS2}) assumptions $\sigma(1) = 1$ and
$\Delta(1) = 0$, \eqref{eq:GenWittProdFormula} leads to
\begin{subequations}
\begin{align}
  [h,f] ={}& 2\sigma(t) t \Delta(t) \cdot \Delta
    \text{,} \label{eq:Shfsimp}\\
  [h,e] ={}& 2 \Delta(t) \cdot \Delta
    \text{,} \label{eq:Shesimp}\\
  [e,f] ={}& -\bigl( \sigma(t) + t \bigr) \Delta(t) \cdot \Delta
    \text{,} \label{eq:Sefsimp}
\end{align}
\end{subequations}
hence as long as $\sigma$ and $\Delta$, similarly to their untwisted
counterparts, yield that the degrees of $t$ in the expressions on the right hand side remain among
those present in the generating set for the $\mathbb{K}$-linear subspace $\mathcal{S}$, it
follows that $\mathcal{S}$ indeed is closed under this bracket.


In the particular case that \(\sigma(t) = qt\) for some \(q \in \mathbb{K}\)
and \(\Delta = D_q\), we obtain a family of hom-Lie algebras deforming
$\mathfrak{sl}_2$, defined with respect to the basis $\{e,f,h\}$ by
the brackets and the linear map $\alpha$ as follows:
\begin{subequations} \label{Eq:q-def-sl2}
\begin{align}
  [h, f] &= -2q f \text{,} &
    \alpha(f) &= q^2 f\text{,} \\
  [h, e] &= 2 e\text{,} &
    \alpha(e) &= q e \text{,} \\
  [e, f] &= \tfrac{1}{2}(1+q) h \text{,}&
    \alpha(h) &= q h \text{.}
\end{align}
\end{subequations}
This is a hom-Lie algebra for all \(q \in \mathbb{K}\) but not a Lie algebra
unless $q=1$, in which case we recover the classical $\mathfrak{sl}_2$.

\subsection{Hom-algebras: Lie and associative}
\label{Ssec:Hom-ass-examples}

An ordinary Lie or associative algebra may informally be described as
an underlying linear space (often assumed to be a vector space, but we
will typically allow it to be a more general module) on which is
defined some bilinear map $m$ called the \emph{multiplication} (or in
the Lie case sometimes the \emph{bracket}). Depending on what
identities this multiplication satisfies, the algebra is classified
as being associative, commutative, anticommutative, Lie, etc.
A \emph{hom-algebra} may similarly be described as an underlying
linear space on which is defined two maps $m$ and $\alpha$. The
multiplication $m$ is again required to be bilinear, whereas $\alpha$
is merely a linear map from the underlying set to itself. The `hom-'
prefix is historically because $\alpha$ in many examples turn out to
be a \emph{homomorphism} with respect to some operation (not
necessarily the $m$ of the hom-algebra, even though that is certainly
not uncommon), but the modern understanding is that $\alpha$ may be
any linear map.


Practically, the point of incorporating some extra map $\alpha$ in
the definition of an algebra is that this can be used to ``twist'' or
``deform'' the identities defining a variety of algebras, and thus
offer greater opportunities for capturing within an abstract axiomatic
framework the many concrete ``twisted'' or ``deformed'' algebras that
have emerged in recent decades. It was shown in~\cite{HLS} that
hom-Lie algebras are closely related to discrete
and deformed vector fields and differential calculus and that some
$q$-deformations of the Witt and the Virasoro algebras have the
structure of a hom-Lie algebra. The paradigmatic example (given above)
is the $\mathfrak{sl}_2$ Lie algebra which deforms to a new nontrivial
hom-Lie algebra by means of $\sigma$-derivations.
Hom-associative algebras are likewise a
generalisation of a usual associative algebras. A common recipe for
producing the hom-analogue of a classical identity is to insert
$\alpha$ applications wherever some variable is not acted upon by $m$
as many times as the others.

\begin{definition} \label{Def:hom-algebra}
  Let $\mathcal{R}$ be some associative and commutative unital ring.
  Formally, \index{hom-algebra} an \emph{$\mathcal{R}$-hom-algebra} $\mathcal{A}$ is a
  triplet $(A, m, \alpha)$, where $A$ is an $\mathcal{R}$-module,
  \(m\colon A \times A \longrightarrow A\) is a bilinear map, and
  \(\alpha\colon A \longrightarrow A\) is a linear map. As usual, the algebra
  $\mathcal{A}$ and its carrier set $A$ are notationally identified
  whenever there is no risk of confusion.

  The \emph{hom-associative identity} for $\mathcal{A}$ is the
  formula
  \begin{equation} \label{Eq:hom-associativity}
    m\bigl( \alpha(x), m(y,z) \bigr) =
    m\bigl( m(x,y), \alpha(z) \bigr)
    \quad\text{for all \(x,y,z \in \mathcal{A}\).}
  \end{equation}
  A hom-algebra which satisfies the hom-associative identity is said
  to be \index{hom-associative algebra} a \emph{hom-associative} algebra. Similarly, \index{hom-Jacobi identity}
  the \emph{hom-Jacobi identity} for $\mathcal{A}$ is the formula
  \begin{equation} \label{Eq:HomJacobi}
    m\bigl( \alpha(x), m(y,z) \bigr) +
    m\bigl( \alpha(y), m(z,x) \bigr) +
    m\bigl( \alpha(z), m(x,y) \bigr) = 0
    \quad\text{for all \(x,y,z \in \mathcal{A}\).}
  \end{equation}
  For a hom-algebra $\mathcal{A}$ to be \index{hom-Lie algebra} a \emph{hom-Lie algebra},
  it must satisfy the hom-Jacobi identity and the ordinary
  anticommutativity (skew-symmetry) identity
  \begin{equation} \label{Eq:anticommutative}
    m(x,x) = 0
    \quad\text{for all \(x \in \mathcal{A}\).}
  \end{equation}
  A hom-algebra $\mathcal{A}$ is said to be \index{multiplicative hom-algebra} \emph{multiplicative} if
  $\alpha$ is an endomorphism of the algebra $(A,m)$, i.e., if
  \begin{equation}
    m\bigl( \alpha(x), \alpha(y) \bigr) =
    \alpha\bigl( m(x,y) \bigr)
    \quad\text{for all \(x,y \in \mathcal{A}\).}
  \end{equation}

  Now let \(\mathcal{A} = (A,m,\alpha)\) and
  \(\mathcal{A}' = (A',m',\alpha')\) be two hom-algebras.
  A \emph{morphism} \(f\colon \mathcal{A} \longrightarrow \mathcal{A}'\) of
  hom-algebras is a linear map \(f\colon A \longrightarrow A'\) such that
  \begin{align}
    m\bigl( f(x), f(y) \bigr) ={}& f\bigl( m(x,y) \bigr)
     &&\text{for all \(x,y \in A\),}
     \label{Eq1:WeakMorphism}\\
   \alpha\bigl( f(x) \bigr) ={}& f\bigl( \alpha(x) \bigr)
     &&\text{for all \(x \in A\).}
  \end{align}
  A linear map \(f\colon A \longrightarrow A'\) that merely satisfies the
  first condition \eqref{Eq1:WeakMorphism} is called a
  \emph{weak morphism} of hom-algebras.
\end{definition}

The concept of weak morphism is somewhat typical of the classical
algebra attitude towards hom-algebras: the multiplication $m$ is
taken as part of the core structure, whereas the map $\alpha$ is seen
more as an add-on. In both universal algebra and the categorical
setting, it is instead natural to view $m$ and $\alpha$ as equally
important for the hom-algebra concept, even though it is of course
also possible to treat weak morphisms (for example with the help
of a suitable forgetful functor) within these settings, should weak
morphisms turn out to be of interest for the problems at hand.
Yau~\cite{Yau:EnvLieAlg} goes one step in the opposite direction and
considers hom-algebras as being hom-modules with a multiplication;
this makes $\alpha$ part of the core structure whereas $m$ is the
add-on.

As usual, the squaring form \eqref{Eq:anticommutative} of the
anticommutative identity implies the more traditional
\begin{equation} \label{Eq2:anticommutative}
    m(x,y) = - m(y,x)
    \quad\text{for all \(x,y \in \mathcal{A}\)}
\end{equation}
in any hom-algebra $\mathcal{A}$. The two are equivalent in an
algebra over a field of characteristic $\neq 2$, but
\eqref{Eq2:anticommutative} implies nothing about $m(x,x)$ in an
algebra over a field of characteristic equal to $2$, and for
hom-algebras over other rings more intermediate outcomes are
possible.

An example of a hom-Lie algebra was given in the previous subsection.
A similar example of a hom-associative algebra would be:

\begin{example}\label{example1ass}
  Let $\{e_1,e_2,e_3\}$ be a basis of a $3$-dimensional linear space
  $A$ over some field $\mathbb{K}$. Let \(a,b \in \mathbb{K}\) be arbitrary
  parameters. The following equalities
  \begin{gather*}
    \begin{aligned}
       m(e_1,e_1) &= a\,e_1 \text{,} &
       m(e_2,e_2) &= a\,e_2 \text{,} \\
       m(e_1,e_2) &= m(e_2,e_1) = a\,e_2 \text{,}&
       m(e_2,e_3) &= b\,e_3 \text{,}\\
       m(e_1,e_3) &= m(e_3,x_1) = b\,e_3 \text{,}&
       m(e_3,e_2) &= m(e_3,e_3) = 0 \text{,}
    \end{aligned}
    \\
    \alpha(e_1) = a\,e_1 \text{,} \quad
    \alpha(e_2) = a\,e_2 \text{,} \quad
    \alpha(e_3) = b\,e_3 \text{,}
  \end{gather*}
  define the multiplication $m$ and linear map $\alpha$ on a
  hom-associative algebra on $\mathbb{K}^3$. This algebra is not
  associative when $a\neq b$ and $b\neq 0$, since
  \( m\bigl( m(e_1,e_1), e_3 \bigr) - m\bigl( e_1, m(e_1,e_3) \bigr)
  = (a -\nobreak b) b e_3\).
\end{example}

\begin{example}[Polynomial hom-associative algebra \cite{Yau:homology}]
  Consider the polynomial algebra $A=\mathbb{K} [x_1,\cdots x_n]$ in $n$
  variables. Let $\alpha$ be an algebra endomorphism of $A$ which is
  uniquely determined by the $n$ polynomials $\alpha(x_i) =
  \sum \lambda_{i;r_1,\dotsc,r_n} x^{r_1}_1 \cdots x^{r_n}_n$ for
  \(1\leqslant i \leqslant n\). Define $m$ by
  \begin{equation}
    m(f,g) =
    f(\alpha(x_1), \dotsc, \alpha (x_n))
    g(\alpha(x_1), \dotsc, \alpha (x_n))
  \end{equation}
  for $f,g$ in $A$.
  Then $(A, m, \alpha)$ is a hom-associative algebra.
  (This example is a special case of Corollary~\ref{cor1:twist}.)
\end{example}

\begin{example}[\cite{Yau:comodule}]
  Let $(A,m,\alpha)$ be a hom-associative $\mathcal{R}$-algebra.
  Denote by $\mathrm{M}_n(A)$ the $\mathcal{R}$-module of
  $n\times n$ matrices with entries in $A$. Then
  $(\mathrm{M}_n(A),m',\alpha')$ is also a hom-associative algebra,
  in which \(\alpha'\colon \mathrm{M}_n(A) \longrightarrow \mathrm{M}_n(A)\) is
  the map that applies $\alpha$ to each matrix element and the
  multiplication $m'$ is the ordinary matrix multiplication over
  $(A,m)$.
\end{example}

The following result states that hom-associative algebra yields
another hom-associative algebra when its multiplication and twisting
map are twisted by a morphism. The following results work as well for
hom-Lie algebras and more generally $G$-hom-associative algebras.
These constructions introduced in~\cite{Yau:homology} and generalized
in~\cite{Yau:MalsevAlternJordan} were extended to many other algebraic
structures.

\begin{theorem}\label{thm:firsttwist}
  Let \(\mathcal{A} = (A, m, \alpha)\) be a hom-algebra and
  \(\beta \colon \mathcal{A} \longrightarrow \mathcal{A}\) be a weak morphism.
  Then \(\mathcal{A}_\beta = (A, m_\beta, \alpha_\beta)\) where
  \(m_\beta = \beta \circ m\) and \(\alpha_\beta = \beta \circ \alpha\)
  is also a hom-algebra. Furthermore:
  \begin{enumerate}
    \item
      If $\mathcal{A}$ is hom-associative then $\mathcal{A}_\beta$ is
      hom-associative.
    \item
      If $\mathcal{A}$ is hom-Lie then $\mathcal{A}_\beta$ is hom-Lie.
    \item
      If $\mathcal{A}$ is multiplicative and $\beta$ is a morphism then
      $\mathcal{A}_\beta$ is multiplicative.
  \end{enumerate}
\end{theorem}
\begin{proof}
  For the hom-associative and hom-Jacobi identities, it suffices to
  consider what a typical term in these identities looks like. We have
  \begin{multline*}
    m_\beta\bigl( \alpha_\beta(x), m_\beta(y,z) \bigr) =
    (\beta \circ m)\bigl( (\beta \circ \alpha)(x),
      (\beta \circ m)(y,z) \bigr)
      = \\ =
    \beta\Bigl( (m \circ \beta\otimes\beta)\bigl(
      \alpha(x), m(y,z) \bigr) \Bigr) =
    \beta\Bigl( (\beta \circ m)\bigl( \alpha(x), m(y,z) \bigr) \Bigr)
      = \\ =
    (\beta \circ \beta)\Bigl( m\bigl( \alpha(x), m(y,z) \bigr) \Bigr)
  \end{multline*}
  Hence either side of the hom-associative and hom-Jacobi respectively
  identities for $\mathcal{A}_\beta$ comes out as $\beta^{\circ 2}$ of
  the corresponding side of the corresponding identity for
  $\mathcal{A}$, and thus these identities for $\mathcal{A}_\beta$
  follow directly from their $\mathcal{A}$ counterparts. The
  anticommutativity identity similarly follows from its counterpart,
  as does the multiplicative identity via
  \begin{multline*}
    m_\beta\bigl( \alpha_\beta(x), \alpha_\beta(y) \bigr) =
    \beta\Bigl( m\bigl( \beta\bigl( \alpha(x),
      \beta\bigl( \alpha(y) \bigr) \bigr) \Bigr) =
    \beta^{\circ 2}\Bigl(
      m\bigl( \alpha(x), \alpha(y) \bigr) \Bigr) = \\ =
    \beta^{\circ 2}\Bigl( \alpha\bigl( m(x,y) \bigr) \Bigr) =
    \beta\Bigl( \alpha\bigl( (\beta \circ m)(x,y) \bigr) \Bigr) =
    \alpha_\beta\bigl( m_\beta(x,y) \bigr)
  \end{multline*}
  for all \(x,y \in A\).
\end{proof}

The \(\alpha = \mathrm{id}\) special case of Theorem~\ref{thm:firsttwist}
yields.

\begin{corollary} \label{cor1:twist}
  Let $(A,m)$ be an associative algebra and $\beta\colon A \longrightarrow A$ be
  an algebra endomorphism. Then $A_\beta = (A, m_\beta, \beta)$ where
  \(m_\beta = \beta \circ m\) is a multiplicative hom-associative
  algebra.
\end{corollary}

That result also has the following partial converse.

\begin{corollary}[\cite{Gohr}]
  Let \(\mathcal{A} = (A,m,\alpha)\) be a multiplicative hom-algebra
  in which $\alpha$ is invertible. Then \(\mathcal{A}' =
  (A, \alpha^{-1} \circ\nobreak m, \mathrm{id})\) is a hom-algebra. In
  particular, any multiplicative hom-associative or hom-Lie algebra
  where $\alpha$ is invertible may be regarded as an ordinary
  associative or Lie respectively algebra, albeit with an awkwardly
  defined operation.
\end{corollary}
\begin{proof}
  Take \(\beta = \alpha^{-1}\) in Theorem~\ref{thm:firsttwist}.
\end{proof}

An application of that corollary is the identity
\[
  m\bigl( x_0, m(x_1,x_2) \bigr) =
  m\Bigl( m\bigr( \alpha^{-1}(x_0), x_1\bigr), \alpha(x_2) \Bigr)
\]
which hold in multiplicative hom-associative algebras with invertible
$\alpha$, and generalises to change the ``tilt'' of longer products.
The idea is to rewrite the product in terms of the corresponding
associative multiplication \(\tilde{m} = \alpha^{-1} \circ m\), with
respect to which $\alpha$ and $\alpha^{-1}$ are also algebra
homomorphisms, and apply the ordinary associative law to change the
``tilt'' of the product before converting the result back to the
hom-associative product $m$.

Since many (hom-)Lie algebras of practical interest are
finite-dimensional, and injectivity implies invertibility for
linear operators on a finite-dimensional space, one might expect
hom-Lie algebras to be particularly prone to fall under the domain
of that corollary, but the important condition that should not be
forgotten is that of the algebra being multiplicative. For example
the $q$-deformed $\mathfrak{sl}_2$ of \eqref{Eq:q-def-sl2} is easily
seen to not be multiplicative for general $q$.

\medskip

An identity that may seem conspicuously missing from
Definition~\ref{Def:hom-algebra} is that of the unit; although they
do not make sense in Lie algebras due to contradicting
anticommutativity, units are certainly a standard feature of
associative algebras, so why has there been no mention of
hom-associative unital algebras? The reason is that they, by the
following theorem, constitute a subclass of that of hom-associative
algebras which is even more restricted than that of the
multiplicative hom-associative algebras.
Unitality of hom-associative algebras were discussed first
in~\cite{Gohr}.

\begin{theorem}
  Let $\mathcal{A}$ be a hom-associative algebra. If there is some
  \(e \in \mathcal{A}\) such that
  \begin{equation}
    m(e,x) = x = m(x,e) \qquad\text{for all \(x \in \mathcal{A}\)}
  \end{equation}
  then
  \begin{equation} \label{Eq:centroid}
    m\bigl( \alpha(x), y \bigr) =
    m\bigl( x, \alpha(y) \bigr) =
    \alpha\bigl( m(x,y) \bigr)
    \qquad\text{for all \(x,y \in \mathcal{A}\).}
  \end{equation}
\end{theorem}
\begin{proof}
  For the first equality,
  \[
    m\bigl( \alpha(x), y \bigr) =
    m\bigl( \alpha(x), m(e,y) \bigr) =
    m\bigl( m(x,e), \alpha(y) \bigr) =
    m\bigl( x, \alpha(y) \bigr)
  \]
  by hom-associativity. For the second equality,
  \begin{multline*}
    m\bigl( x, \alpha(y) \bigr) =
    m\bigl( m(e,x), \alpha(y) \bigr) =
    m\bigl( \alpha(e), m(x,y) \bigr) = \\ =
    m\Bigl( e, \alpha\bigl( m(x,y) \bigr) \Bigr) =
    \alpha\bigl( m(x,y) \bigr)
  \end{multline*}
  by hom-associativity and the first equality.
\end{proof}

An identity such as \eqref{Eq:centroid} has profound effects on
the structure of a hom-associative algebra. Basically, it means
applications of $\alpha$ are not located in any particular position
in a product, but can move around unhindered. At the same time, even
a single $\alpha$ somewhere will act as a powerful lubricant that
lets the hom-associative identity shuffle around parentheses as
easily as the ordinary associative identity. In particular, any
product of $n$ algebra elements $x_1,\dotsc,x_n$ where at least one
is in the image of $\alpha$ will effectively be an associative
product; probably not the wanted outcome if one's aim is to create
new structures through deformations of old ones.

On the other hand, $\alpha$ satisfying \eqref{Eq:centroid} obviously
have some rather special properties. One may for any algebra
\(\mathcal{A} = (A,m)\) define the \emph{centroid}
$\mathrm{Cent}(\mathcal{A})$ of $\mathcal{A}$ as the set of all linear
self-maps \(\alpha \colon A \longrightarrow A\) satisfying the condition
\(\alpha\bigl( m(x, y) \bigr) = m\bigl( \alpha(x), y\bigr) =
m \bigl( x, \alpha (y) \bigr)\) for all $x,y\in A$.  Notice that if
\(\alpha \in \mathrm{Cent}(\mathcal{A})\), then we have
\(m\bigl( \alpha^p(x), \alpha^q(y) \bigr) =
(\alpha^{p+q} \circ\nobreak m)(x,y)\) for all \(p,q \geqslant 0\).
The construction of hom-algebras using elements of the centroid was
initiated in~\cite{BenayadiMakhlouf} for Lie algebras. We have

\begin{proposition}
  Let $(\mathcal{A},m)$ be an associative algebra and \(\alpha \in
  \mathrm{Cent}(\mathcal{A})\). Set for $x,y\in\mathcal{A}$
  \begin{align*}
    m_1(x,y) ={}& m\bigl( \alpha(x), y \bigr) \text{,}\\
    m_2(x,y) ={}& m\bigl( \alpha(x), \alpha(y) \bigr) \text{.}
  \end{align*}
  Then $(\mathcal{A},m_1,\alpha)$ and $(\mathcal{A},m_2,\alpha)$
  are hom-associative algebras.
\end{proposition}

Indeed we have
\begin{multline*}
  m_1\bigl( \alpha(x), m_1(y,z) \bigr) =
  m\Bigl( \alpha^2(x), m\bigl( \alpha(y), z \bigr) \Bigr) =
  \alpha\Bigl( m\bigl( \alpha(x), m( \alpha(y), z) \bigr) \Bigr)
  =  \\ =
  m\Bigl( \alpha (x), \alpha\bigl( m\bigl( \alpha (y), z
    \bigr) \bigr) \Bigr) =
  m\Bigl( \alpha (x), m\bigl( \alpha (y), \alpha(z) \bigr) \Bigr)
  \text{.}
\end{multline*}

\begin{remark}
The definition of unitality which fits with Corollary~\ref{cor1:twist}
was introduced in~\cite{Gohr} and then used for hom-bialgebra and
hom-Hopf algebras in~\cite{Canepl2009}.

  Let $(\mathcal{A},m,\alpha)$ be a hom-associative algebra.
  It is said to be unital if there is some \(e \in \mathcal{A}\)
  such that
  \begin{equation} \label{HomAssUnit}
    m(e,x) = \alpha(x) = m(x,e)
    \qquad\text{for all \(x \in \mathcal{A}\).}
  \end{equation}
  
Therefore, similarly to Corollary~\ref{cor1:twist}, a unital
associative algebra gives rise a unital hom-associative algebra.
\end{remark}

\subsection{Admissible and enveloping hom-algebras}
\label{Ssec:Yau-constructions}

Two concepts that are of key importance in the theory of ordinary Lie
algebras are those of Lie-admissible and enveloping algebras. In the
setting of hom-algebras, these concepts are defined as follows, with
the classical non-hom concepts arising in the special case \(\alpha =
\mathrm{id}\).

\begin{definition}
  Let a hom-algebra \(\mathcal{A} = (A, m, \alpha)\) be given.
  Define \(b(x,y) = m(x,y) - m(y,x)\) to be the commutator (bracket)
  corresponding to $m$, and let $\mathcal{A}^-$ be the hom-algebra
  $(A,b,\alpha)$. The algebra $\mathcal{A}$ is said to be \index{hom-Lie-admissible hom-algebra}
  \emph{hom-Lie-admissible} if the hom-algebra $\mathcal{A}^-$ is
  hom-Lie.

  Now let $\mathcal{L}$ be a hom-Lie algebra.
  $\mathcal{A}$ is said to be \index{enveloping algebra for hom-Lie algebra} an \emph{enveloping algebra} for
  $\mathcal{L}$ if $\mathcal{L}$ is isomorphic to some hom-subalgebra
  \(\mathcal{B} = \bigl( B, b_{\vert B \times B}, \alpha_{\vert B})\)
  of $\mathcal{A}^-$ such that $B$ generates $\mathcal{A}$.
\end{definition}

It was shown in~\cite[Prop.~1.6]{MS} that any hom-associative algebra
is hom-Lie-admissible. On one hand, this becomes another method of
constructing new hom-Lie algebras, but it is more interesting when
wielded to the opposite end of studying a given hom-Lie algebra
through a corresponding enveloping algebra. To explain why this is
so, we will briefly review the classical theory of ordinary Lie and
associative algebras.

On a Lie group, the exponential map \(v \mapsto \exp(v)\) allows
transitioning from tangent vectors to non-infinitesimal shifts;
$\exp(tv)$ is the point where you end up if travelling from the
identity point at velocity $v$ for time $t$. Under the interpretation
that identifies vectors with invariant vector fields, and vector
fields with derivations on the ring of scalar-valued functions
(``scalar fields'', in the physicist terminology), the exponential
map may in fact be defined via the elementary power series formula
\(\exp(v) = \sum_{n=0}^\infty \frac{v^n}{n!}\) (where multiplication
of vectors is composition of differential operators) and in the Lie
group $(\mathbb{R},+)$ this turns out to be Taylor's formula:
$\exp\left( t\frac{d}{dx} \right)$ is the shift operator mapping an
analytic function $f$ to the shifted variant \(x \mapsto
f(x +\nobreak t)\). When doing the same in a more general Lie group,
one must however be careful
to note that vector fields need not commute, and that already the
degree $2$ term of for example $\exp(u +\nobreak v)$ contains $uv$
and $vu$ terms that need not be equal. The role of the Lie algebra is
precisely to keep track of the extent to which vector fields do not
commute, so the proper place to do algebra with vector fields to the
aim of studying the exponential map must be in an enveloping algebra
of the Lie algebra of invariant vector fields on the underlying Lie
group.

Conversely, one may start with a Lie algebra $\mathfrak{g}$ and ask
oneself what the corresponding Lie group would be like, by studying
formal series in the basic vector fields, while keeping in mind that
these should satisfy the commutation relations encoded into
$\mathfrak{g}$; this leads to the concept of \emph{formal groups}.
An important step towards it is the construction of the
\emph{(associative) universal enveloping algebra}
$\mathrm{U}(\mathfrak{g})$, which starts with the free associative
algebra generated by $\mathfrak{g}$ as a module and imposes upon it
the relations that
\begin{equation}
  xy-yx = [x,y]
  \qquad\text{for all \(x,y \in \mathfrak{g}\),}
\end{equation}
where on the left hand side we have multiplication in
$\mathrm{U}(\mathfrak{g})$ but on the right hand side the bracket
operation of the Lie algebra $\mathfrak{g}$. More technically, the
free associative algebra in question can be constructed as the tensor
algebra \(\mathrm{T}(\mathfrak{g}) = \bigoplus_{n=0}^\infty
\mathfrak{g}^{\otimes n}\) where the product of \(x_1 \otimes \dotsb
\otimes x_m \in \mathfrak{g}^{\otimes m}\) and \(y_1 \otimes \dotsb
\otimes y_n \in \mathfrak{g}^{\otimes n}\) is \(x_1 \otimes \dotsb
\otimes x_m \otimes y_1 \otimes \dotsb \otimes y_n \in
\mathfrak{g}^{\otimes (m+n)}\). Imposing the commutation relations
can then be done by taking the quotient by the two-sided ideal
$J(\mathfrak{g})$ in $\mathrm{T}(\mathfrak{g})$ that is generated by
all \(xy - yx - [x,y]\) for \(x,y \in \mathfrak{g}\), i.e.,
\[
  \mathrm{U}(\mathfrak{g}) :=
  \mathrm{T}(\mathfrak{g}) \big/ J(\mathfrak{g}) =
  \mathrm{T}(\mathfrak{g}) \bigg/ \Big<
    \setOf[\big]{ xy - yx - [x,y] }{ x,y \in \mathfrak{g} }
  \Bigr> \text{.}
\]

With this in mind, it is only natural to generalise this construction
to the hom-case, and in~\cite{Yau:EnvLieAlg} Yau does so. Since he
in the non-associative case cannot take advantage of familiar
concepts such as the tensor algebra, this construction will however
involve a few steps more than one might be used to from the non-hom
setting.
Notably, Yau begins with setting up the free hom-algebra
$F_{\mathrm{HNAs}}(\mathfrak{g})$: neither hom-associativity
nor ordinary associativity is inherent. Then he goes on to impose
hom-associativity by taking a quotient, which results in the free
hom-associative algebra $F_{\mathrm{HAs}}(\mathfrak{g})$;
this is what corresponds to the tensor algebra
$\mathrm{T}(\mathfrak{g})$. Another quotient imposes also the
commutation relations, to finally yield the universal enveloping
hom-associative algebra $U_{\mathrm{HLie}}(\mathfrak{g})$.

When reading through the technical details of these constructions,
which we shall quote below for the reader's convenience, they may
seem a daring plunge forward into very general algebra, that
harnesses advanced combinatorial objects to achieve a clear picture
of the algebra. It may be that they are that, but our main point in
the next section is that they are also an entirely straightforward
application of the basic methods of universal algebra, so there is in
fact very little that was novel in these constructions. The reader
who has grasped the material in Section~\ref{Sec:ClassicalUA} will be
able to recreate something equivalent to the following 
(modulo some minor optimisations) from scratch.

For $n \geqslant 1$, let $T_n$ denote the set of isomorphism classes
of plane\footnote{
  Yau, like many other algebraists, actually uses the term `planar'
  rather than `plane', but this practice is simply wrong as the two
  words refer to slightly different graph-theoretical properties:
  a graph is \emph{planar} if it can be embedded in a genus $0$
  surface, but \emph{plane} if it is given with such an embedding.
  To speak of a
  `planar tree' is a tautology, because trees by definition contain
  no cycles, will therefore have no subdivided $K_5$ or $K_{3,3}$ as
  subgraph, and thus by Kuratowski's Theorem be planar. What is of
  utmost importance here is rather that the trees are given with a
  (combinatorial) \emph{embedding} into the plane, since that specifies
  a local cyclic order on edges incident with a vertex, which is what
  the isomorphisms spoken of are required to preserve. As rooted trees,
  the two elements of $T_3$ are isomorphic, but as plane rooted trees
  they are not.
%
} binary trees with $n$ leaves and one root.  The first $T_n$
are depicted below.
\begin{multline*}
  T_1 = \left\lbrace \tone \, \right\rbrace, \quad
  T_2 = \left\lbrace \ttwo \right\rbrace, \quad
  T_3 = \left\lbrace \tthreeone,\, \tthreetwo \right\rbrace,
  \\
  T_4 = \left\lbrace
    \tfourone,\, \tfourtwo,\, \tfourthree,\, \tfourfour,\, \tfourfive
  \right\rbrace
  \text{.}
\end{multline*}
Each dot represents either a leaf, which is always depicted at the
top, or an internal vertex.  An element in $T_n$ will be called an
\emph{$n$-tree}.  The set of nodes ($=$ leaves and internal vertices)
in a tree $\psi$ is denoted by $N(\psi)$.  The node of an $n$-tree
$\psi$ that is connected to the root (the lowest point in the
$n$-tree) will be denoted by $v_{\mathrm{low}}$.  In other words,
$v_{\mathrm{low}}$ is the lowest internal vertex in $\psi$ if
$n \geq 2$ and is the only leaf if $n = 1$.

Given an $n$-tree $\psi$ and an $m$-tree $\varphi$, their
\emph{grafting} $\psi \vee \varphi \in T_{n+m}$
is the tree obtained by placing $\psi$ on the left and $\varphi$ on
the right and joining their roots to form the new lowest internal
vertex, which is connected to the new root.  Pictorially, we have
$$
  \psi \vee \varphi \,=\, \psiwedgephi
  \text{.}
$$
Note that grafting is a nonassociative operation.  As we will discuss
below, the operation of grafting is for generating the multiplication
$m$ of a free nonassociative algebra.

To handle hom-algebras, we need to introduce weights on plane trees.
A \emph{weighted $n$-tree} is a pair \(\tau = (\psi, w)\), in which
$\psi \in T_n$ is an $n$-tree and
$w$ is a function from the set of internal vertices of $\psi$
to the set $\mathbb{N}$ of non-negative integers.
If $v$ is an internal vertex of $\psi$, then we call $w(v)$ the
weight of $v$. The $n$-tree $\psi$ is called the underlying $n$-tree of
$\tau$, and $w$ is called the weight function of $\tau$. The set of all
weighted $n$-trees is denoted $T^{\mathrm{wt}}_n$.
Since the $1$-tree has no internal vertex, we have that $T_1 =
T^{\mathrm{wt}}_1$. Likewise, the grafting of two weighted trees is
defined as above by connecting them to a new root for which the weight
is $0$. There is also an operation to change the weight; for \(\tau =
(\psi,w)\), we define \(\tau[r] = (\psi,w')\) where
\(w'(v_{\mathrm{low}}) = w(v_{\mathrm{low}}) + r\) and \(w'(v)=w(v)\)
for all internal vertices \(v \neq v_{\mathrm{low}}\).

Now let an $\mathcal{R}$-module $A$ and a linear map \(\alpha\colon A
\longrightarrow A\) be given. As a set,
$$
  F_{\mathrm{HNAs}}(A) =
  \bigoplus_{n\geq 1} \bigoplus_{\tau\in T_n^{\mathrm{wt}}}
    A^{\otimes n}
  \text{.}
$$
We write $A_\tau^{\otimes n}$ for the component in this direct
sum that corresponds to the values $n$ and $\tau$ of these summation
indices. There is a canonical isomorphism \(A_\tau^{\otimes n} \cong
A^{\otimes n}\). For any \(n \geqslant 1\), \(\tau \in
T_n^{\mathrm{wt}}\), and \(x_1,\dotsc,x_n \in A\), we write
$(x_1 \otimes\nobreak \dotsb \otimes\nobreak x_n)_\tau$ for the
element of $A_\tau^{\otimes n}$ that corresponds to \(x_1 \otimes
\dotsb \otimes x_n \in A^{\otimes n}\).
The linear map $\alpha$ is extended to a linear map \(\alpha_F\colon
F_{\mathrm{HNAs}}(A) \longrightarrow F_{\mathrm{HNAs}}(A)\) by the rule
$$
  \alpha_F \bigl( (x_1 \otimes \dotsb \otimes x_n)_\tau \bigr) =
  (x_1 \otimes \dotsb \otimes x_n)_{\tau[1]}
  \qquad\text{for \(\tau \notin T_1\)}
$$
and the multiplication $m_F$ on $F_{\mathrm{HNAs}}(A)$ is defined by
$$
  m_F\bigl(
    (x_1 \otimes\dotsb\otimes x_n)_\tau,
    (x_{n+1} \otimes\dotsb\otimes x_{n+m})_\sigma
  \bigr) = (x_1 \otimes\dotsb\otimes x_{n+m})_{\tau\vee\sigma}
$$
and bilinearity. This $\bigl( F_{\mathrm{HNAs}}(A), m_F, \alpha_F
\bigr)$ is the free (nonassociative) $\mathcal{R}$-hom-algebra
generated by the hom-module $(A,\alpha)$.

From there, the corresponding free hom-associative algebra is
constructed as the quotient
$$
  F_{\mathrm{HAs}}(A) :=
  F_{\mathrm{HNAs}}(A) \big/ J^\infty
$$
where \(J^\infty = \bigcup_{n\geq 1} J^n\) and \(J^1 \subseteq
J^2 \subseteq \dotsb \subseteq J^\infty \subset
F_{\mathrm{HNAs}}(A)\) is an ascending chain of two-sided ideals
defined by
\begin{align*}
  J^1 ={}& \Bigl< \operatorname{Im}\bigl(
    m_F \circ (m_F \otimes \alpha_F - \alpha_F \otimes m_F
  \bigr) \Bigr>
  \text{,}\\
  J^{n+1} ={}& \bigl< J^n \cup \alpha_F(J^n) \bigr>
    \qquad\text{for \(n \geqslant 1\).}
\end{align*}
The \emph{universal enveloping algebra} of a hom-Lie algebra
\((\mathfrak{g}, b, \alpha)\) is similarly obtained as the
quotient
$$
  U_{\mathrm{HLie}}(\mathfrak{g}) :=
  F_{\mathrm{HNAs}}(\mathfrak{g}) \big/ I^\infty
$$
where $I^\infty$ is the two-sided ideal obtained if one starts with
$$
  I^1 = \Bigl< \operatorname{Im}\bigl(
    m_F \circ (m_F \otimes \alpha_F - \alpha_F \otimes m_F)
  \bigr) \cup \bigl\{
     m_F(x,y) - m_F(y,x) - b(x,y) \bigm|  x,y\in \mathfrak{g}
  \bigr\} \Bigr>
$$
and then similarly lets \(I^{n+1} = \bigl< I^n \cup\nobreak
\alpha_F(I^n) \bigr>\) for \(n \geqslant 1\) and \(I^\infty =
\bigcup_{n \geqslant 1} I^n\). Since \(I^n \supseteq J^n\) for all
\(n \geqslant 1\), it follows that
$U_{\mathrm{HLie}}(\mathfrak{g})$ may alternatively be
regarded as a quotient of $F_{\mathrm{HAs}}(\mathfrak{g})$. This
further justifies labelling the hom-associative algebra
$U_{\mathrm{HLie}}(\mathfrak{g})$ as the hom-analogue for a
hom-Lie algebra $\mathfrak{g}$ of the universal enveloping algebra of
a Lie algebra.

There is however one important question regarding this
$U_{\mathrm{HLie}}$ which has not been answered by the
above, and in fact seems to be open in the literature: \emph{Is
$U_{\mathrm{HLie}}(\mathfrak{g})$ for every hom-Lie algebra
$\mathfrak{g}$ an enveloping algebra of $\mathfrak{g}$?}
It follows from the form of the construction that there is a linear
map \(j\colon \mathfrak{g} \longrightarrow
U_{\mathrm{HLie}}(\mathfrak{g})\) with the properties that
\begin{align*}
  j\bigl( [x,y] \bigr) ={}& j(x)j(y) - j(y)j(x)
    && \text{for all \(x,y \in \mathfrak{g}\),}\\
  j\big( \alpha(x) \bigr) ={}& \alpha\bigl( j(x) \bigr)
    && \text{for all \(x \in \mathfrak{g}\),}
\end{align*}
and hence $j$ becomes a morphism of hom-Lie algebras \(\mathfrak{g}
\longrightarrow U_{\mathrm{HLie}}(\mathfrak{g})^-\), but it is
entirely unknown whether $j$ is injective. A failure to be injective
would obviously render these hom-associative enveloping algebras of
hom-Lie algebras less important than the ordinary associative
enveloping algebras of ordinary Lie algebras, as they would fail to
capture all the information encoded into the hom-Lie algebra.

Another way of phrasing the conjecture that the canonical homomorphism
is injective is that the ideal $I^\infty$ used to construct
$U_{\mathrm{HLie}}(\mf{g})$ does not contain any degree~$1$ elements;
such elements would correspond to linear dependencies in
$U_{\mathrm{HLie}}(\mf{g})$ between the images of basis elements in
$\mf{g}$. A simple argument for this conjecture would be that such
dependencies do not occur in the associative case, and since the
hom-associative case has ``more degrees of freedom'' than the
associative case, it shouldn't happen here either. An argument
\emph{against} it comes from the converse of the
Poincar\'e--Birk\-hoff--Witt Theorem~\cite{ViktorAlgebraVI}:
\emph{If the canonical homomorphism \(\mf{g} \longrightarrow
\mathrm{U}(\mf{g})\) is injective, then $\mf{g}$ is a Lie
algebra;} the ordinary universal enveloping algebra construction only
manages to envelop the algebra one starts with if that algebra is a
Lie algebra. What can be hoped for is of course that the conditions
inherent in $U_{\mathrm{HLie}}$ have precisely those deformations
relative to the conditions of $U_{\mathrm{Lie}}$ that makes
everything work out for hom-Lie algebras instead, but they could just
as well end up going some other way.

To positively resolve the envelopment problem, one would probably
have to prove a hom-analogue of the Poincar\'e--Birk\-hoff--Witt
Theorem. Methods for this---particularly the Diamond Lemma---are
available, but the calculations required seem to be rather extensive.
To negatively resolve the envelopment problem, it would be sufficient
to find one hom-Lie algebra $\mf{g}$ for which the canonical
homomorphism \(\mf{g} \longrightarrow U_{\mathrm{HLie}}(\mf{g})\) is not
injective. Yau does show in~\cite[Th.~2]{Yau:EnvLieAlg} that
$U_{\mathrm{HLie}}(\mf{g})$ satisfies an universal property with
respect to hom-associative enveloping algebras, so a hom-Lie algebra
$\mf{g}$ which constitutes a counterexample cannot arise as a
subalgebra of $\mathcal{A}^-$ for any hom-associative algebra
$\mathcal{A}$.

\section{Classical universal algebra: free algebras and their quotients}
\label{Sec:ClassicalUA}

\subsection{Discrete free algebras}

A basic concept in universal algebra is that of the \emph{signature}.
A signature $\Omega$ is a set of formal symbols, together with a
function \(\mathrm{arity}\colon \Omega \longrightarrow \mathbb{N}\) that gives the arity,
or ``wanted number of operands'', for each symbol. Symbols with arity
$0$ are called \emph{constants} (or said to be \emph{nullary}),
symbols with arity $1$ are said to be \emph{unary}, symbols with
arity $2$ are said to be \emph{binary}, symbols with arity $3$ are
said to be \emph{ternary}, and so on; one may also speak about a
symbol being $n$-ary. A convenient shorthand, used in for
example~\cite{TATA}, for specifying signatures is as a set of
``function prototypes'': symbols of positive arity are followed by a
parenthesis containing one comma less than the arity, whereas
constants are not followed by a parenthesis. Hence \(\Omega = \bigl\{
\mathsf{a}(), \mathsf{m}({,}), \mathsf{x}, \mathsf{y} \bigr\}\) is the signature of four
symbols $\mathsf{a}$, $\mathsf{m}$, $\mathsf{x}$, and $\mathsf{y}$, where $\mathsf{a}$ is unary,
$\mathsf{m}$ is binary, and the remaining two are constants. The signature
for a hom-algebra is thus $\bigl\{ \mathsf{a}(), \mathsf{m}(,) \bigr\}$, whereas
the signature for a unary hom-algebra would be $\bigl\{ \mathsf{a}(),
\mathsf{m}(,), \mathsf{1} \bigr\}$; a unit would be an extra constant symbol.

Given a signature $\Omega$, a set $A$ is said to be an
$\Omega$-algebra if it for every symbol \(x \in \Omega\) comes with a
map \(f_x\colon A^{\mathrm{arity}(x)} \longrightarrow A\); these maps are the
\emph{operations} of the algebra. Note that no claim is made that the
operations fulfill any particular property (beyond matching the
respective arities of their symbols), so the $\Omega$-algebra
structure is not determined by $A$ unless that set has cardinality
$1$; therefore one might want to be more formal and say it is \(\mc{A}
= \bigl( A, \{f_x\}_{x \in \Omega} \bigr)\) that is the
$\Omega$-algebra, but we shall in what follows generally be concerned
with only one $\Omega$-algebra structure at a time on each base set.

What the $\Omega$-algebra concept suffices for, despite imposing
virtually no structure upon the object in question, is the definition
of an \emph{$\Omega$-algebra homomorphism}: a map \(\phi\colon A
\longrightarrow B\) is an $\Omega$-algebra homomorphism from $\bigl( A,
\{f_x\}_{x\in\Omega} \bigr)$ to $\bigl( B, \{g_x\}_{x\in\Omega}
\bigr)$ if
\begin{multline}
  \phi\bigl( f_x(a_1,\dotsc,a_{\mathrm{arity}(x)}) \bigr) =
  g_x\bigl( \phi(a_1), \dotsc, \phi(a_{\mathrm{arity}(x)}) \bigr)
  \\\text{for all \(a_1,\dotsc,a_{\mathrm{arity}(x)} \in A\) and
    \(x \in \Omega\).}
\end{multline}
It is easy to verify that these homomorphisms obey the axioms for
being the morphisms in the category of $\Omega$-algebras, so that
category $\Omega\mathtt{\txthyphen algebra}$ is what one gets.
One may then define (up to isomorphism) \emph{the free
$\Omega$-algebra} as being the free object in this category, or more
technically state that $\mc{F}_\Omega(X)$ together with \(i\colon X
\longrightarrow \mc{F}_\Omega(X)\) is \emph{the free $\Omega$-algebra generated
by $X$} if there for every $\Omega$-algebra $\mc{A}$ and every map
\(j\colon X \longrightarrow \mc{A}\) exists a unique $\Omega$-algebra
homomorphism \(\phi\colon \mc{F}_\Omega(X) \longrightarrow \mc{A}\) such that
\(j = \phi \circ i\). An alternative claim to the same effect is that
$\mc{F}_\Omega$, interpreted as a functor from $\mathtt{Set}$ to
$\Omega\mathtt{\txthyphen algebra}$, is left adjoint of the forgetful
functor mapping an $\Omega$-algebra to its underlying set.

Although these definitions may seem frightfully abstract, the objects
in question are actually rather easy to construct: $\mc{F}_\Omega(X)$
is merely the set $T(\Omega,X)$ of all \emph{formal terms} in $\Omega
\mathbin{\dot{\cup}} X$, where the elements of $X$ are interpreted as
symbols of arity $0$. Hence the first few elements of $T\bigl( \bigl\{
\mathsf{a}(), \mathsf{m}({,})\bigr\}, \{\mathsf{x}, \mathsf{y}\} \bigr)$ are
$$
  \mathsf{x}, \mathsf{y}, \mathsf{a}(\mathsf{x}), \mathsf{a}(\mathsf{y}),
  \mathsf{a}(\mathsf{a}(\mathsf{x})), \mathsf{a}(\mathsf{a}(\mathsf{y})),
  \mathsf{m}(\mathsf{x},\mathsf{x}), \mathsf{m}(\mathsf{x},\mathsf{y}), \mathsf{m}(\mathsf{y},\mathsf{x}), \mathsf{m}(\mathsf{y},\mathsf{y}),
  \dotsc
$$
and the operations $\{f_x\}_{x\in\Omega}$ in the free $\Omega$-algebra
$T(\Omega,X)$ merely produce their formal terms counterparts:
\begin{multline*}
  f_x( t_1, \dotsc, t_{\mathrm{arity}(x)} ) := x ( t_1, \dotsc, t_{\mathrm{arity}(x)} )
  \\ \text{for all \(t_1,\dotsc,t_{\mathrm{arity}(x)} \in T(\Omega,X)\) and
    \(x \in \Omega\).}
\end{multline*}
Conversely, the unique morphism $\phi$ of the universal property
turns out to evaluate formal terms in the codomain $\Omega$-algebra,
so for any given \(j\colon X \longrightarrow B\) it can be defined recursively
through
\begin{equation*}
  \phi(t) = \begin{cases}
    j(x)& \text{if \(t=x \in X\),}\\
    g_x\bigl( \phi(t_1),\dotsc,\phi(t_n) \bigr)&
      \text{if \(t = x(t_1,\dotsc,t_n)\) where \(x \in \Omega\)}
  \end{cases}
\end{equation*}
for all \(t \in T(\Omega,X)\).

\subsection{Quotient algebras}

Completely free algebras might be cute, but most of the time one is
rather interested in something with a bit more structure, in the sense
that certain identities are known to hold; in an associative algebra,
the associativity identity holds, whereas in a hom-associative algebra
the hom-associative identity \eqref{Eq:hom-associativity} holds. One
approach to imposing such properties on one's algebras is to restrict
attention to the subcategory of $\Omega$-algebras which satisfy the
wanted identities, and then look at the free object of that subcategory.
Another approach is to take a suitable quotient of the free object
from the full category.

In general $\Omega$-algebras, the denominator in a quotient is a
\emph{congruence relation} on the numerator, and an $\Omega$-algebra
congruence relation is an equivalence relation which is preserved by
the operations; $\equiv$ is a congruence relation on \(\mc{A} =
\bigl( A, \{f_x\}_{x\in\Omega} \bigr)$ if it is an equivalence
relation on $A$ and
\begin{multline*}
  f_x(a_1,\dotsc, a_n) \equiv f_x(b_1,\dotsc,b_n)\\
  \text{for all \(a_1,\dotsc,a_n,b_1,\dotsc,b_n \in A\), \(x \in
  \Omega\), and \(n=\mathrm{arity}(x)\)}\\
  \text{such that \(a_1 \equiv b_1\), \(a_2 \equiv b_2\), \dots, and
  \(a_n \equiv b_n\).}
\end{multline*}
The quotient \(\bigl( B, \{g_x\}_{x\in\Omega}\bigr) := \mc{A}/{\equiv}\)
then has $B$ equal to the set of $\equiv$-equivalence classes in $A$,
and operations defined by
\begin{multline*}
  g_x\bigl( [a_1], \dotsc, [a_{\mathrm{arity}(x)}] \bigr) =
  \bigl[ f_x(a_1, \dotsc, a_{\mathrm{arity}(x)}) \bigr] \\
  \text{for all \(a_1,\dotsc,a_{\mathrm{arity}(x)} \in A\) and \(x \in \Omega\);}
\end{multline*}
congruence relations are precisely those for which this definition
makes sense. Conversely, the relation $\equiv$ defined on some
$\Omega$-algebra $\mc{A}$ by \(a \equiv b\) iff \(\phi(a) = \phi(b)\)
will be a congruence relation whenever $\phi$ is an $\Omega$-algebra
homomorphism.

It should at this point be observed that defining specific congruence
relations to that they respect particular identities is not an
entirely straightforward matter;
it would for example be wrong to expect a simple formula such
as `\(b \equiv b'\) iff \(b = \mathsf{m}\bigl( \mathsf{a}(b_1), \mathsf{m}(b_2,b_3)
\bigr)\) and \(b' = \mathsf{m}\bigl( \mathsf{m}(b_1,b_2), \mathsf{a}(b_3) \bigr)\) for
some \(b_1,b_2,b_3 \in \mc{F}_\Omega(X)\)' to set up the
congruence relation imposing hom-associativity on $\mc{F}_\Omega(X)$,
as it actually fails even to define an equivalence relation. Instead
one considers the family of \emph{all} congruence relations which
fulfill the wanted identities, and picks the smallest of these, which
also happens to be the intersection of the entire family; this makes
precisely those identifications of elements which would be logical
consequences of the given axioms, but nothing more. Thus to construct
the free hom-associative $\bigl\{ \mathsf{a}(), \mathsf{m}(,) \bigr\}$-algebra
generated by $X$, one would let \(\Omega = \bigl\{ \mathsf{a}(), \mathsf{m}(,)
\bigr\}\) and form $T(\Omega,X) \big/ \equiv$, where $\equiv$ is
defined by
\begin{equation} \label{Eq:Kongr.ass.algebra}
  t \equiv t' \quad\Longleftrightarrow\quad
  \parbox[t]{0.7\linewidth}{\raggedright
    \(t \sim t'\) for every congruence relation $\sim$ on
    $T(\Omega,X)$ satisfying \(\mathsf{m}\bigl( \mathsf{a}(t_1), \mathsf{m}(t_2,t_3)
    \bigr) \sim \mathsf{m}\bigl( \mathsf{m}(t_1,t_2), \mathsf{a}(t_3) \bigr)\) for all
    \(t_1,t_2,t_3 \in T(\Omega,X)\).
  }
\end{equation}

Another thing that should be observed is that this construction of
the free hom-associative algebra is not \emph{effective}, i.e., one
cannot use it to implement the algebra on a computer, nor to reliably
carry out calculations with pen and paper. The
construction does suggest both an encoding of arbitrary algebra
elements---since the algebra elements are equivalence classes, just
use any element of a class to represent it---and an implementation
of operations---just perform the corresponding operation of
$\mc{F}_\Omega(X)$ on the equivalence representatives---but it does
then not suggest any algorithm for deciding equality. Providing such
an algorithm is of course equivalent to solving the word problem for
the algebra\slash congruence relation in question, so there cannot be
a universal method which works for arbitrary algebras, but nothing
prevents seeking a solution that works a particular algebra, and
indeed one should always consider this an important problem to solve
for every class of algebras one considers.

One common form of solutions to the word problem is to device a
\emph{normal form map} for the congruence relation $\equiv$: a map
\(N\colon T(\Omega,X) \longrightarrow T(\Omega,X)\) such that \(N(t) \equiv t\)
for all \(t \in T(\Omega,X)\) and \(t \equiv t'\) iff \(N(t)=N(t')\);
this singles out one element from each equivalence class as being the
\emph{normal form} representative of that class, thereby reducing the
problem of deciding congruence to that of testing whether the
respective normal forms are equal. Normal form maps are often
realised as the limit of a system of \emph{rewrite rules} derived
directly from the defining relations; we shall return to this matter
in Subsection~\ref{Ssec:DL}.

\subsection{Algebras with linear structure}

One thing that has so far been glossed over is that e.g.~a
hom-associative algebra is not just supposed to have a
non-associative multiplication $\mathsf{m}$ and a homomorphism $\mathsf{a}$, it
is also supposed to have addition and multiplication by a scalar. The
general way to ensure this is of course to extend the signature with
operations for these, and then impose the corresponding axioms on the
congruence relation used, but a more practical approach is usually to
switch to a category where the wanted linear structure is in place
from the start. As it turns out the free object in the category of
algebras with a linear structure can be constructed as the set of
formal linear combinations of elements in the free (without linear
structure) algebra, our constructions above remain highly useful.

Let $\mc{R}$ be an associative and commutative ring with unit. An
$\Omega$-algebra \(\bigl( A, \{f_x\}_{x\in\Omega} \bigr)$ is
\emph{$\mc{R}$-linear} if $A$ is an $\mc{R}$-module and each operation
$f_x$ is $\mc{R}$-multilinear, i.e., it is $\mc{R}$-linear in each
argument. An $\Omega$-algebra homomorphism \(\phi\colon \mc{A} \longrightarrow
\mc{B}\) is an \emph{$\mc{R}$-linear $\Omega$-algebra homomorphism}
if $\mc{A}$ and $\mc{B}$ are $\mc{R}$-linear $\Omega$-algebras and
$\phi$ is an $\mc{R}$-module homomorphism. An \emph{$\mc{R}$-linear
$\Omega$-algebra congruence relation} $\equiv$ is an $\Omega$-algebra
congruence relation on an $\mc{R}$-linear $\Omega$-algebra which is
preserved also by module operations, i.e., \(a_1 \equiv b_1\) and
\(a_2 \equiv b_2\) implies \(ra_1 \equiv rb_1\) (for all \(r \in
\mc{R}\)) and \(a_1+a_2 \equiv b_1+b_2\).

The free $\mc{R}$-linear $\Omega$-algebra generated by a set $X$ can
be constructed as the set of all formal linear combinations of
elements of $T(\Omega,X)$, i.e., as the free $\mc{R}$-module with
basis $T(\Omega,X)$; we will denote this free algebra by
$\mc{R}\{\Omega,X\}$ (continuing the notation family $\mc{R}[X]$,
$\mc{R}(X)$, $\mc{R}\langle X\rangle$). The universal property it
satisfies is that any function \(j\colon X \longrightarrow \mc{A}\) where
$\mc{A}$ is an $\mc{R}$-linear $\Omega$-algebra gives rise to a
unique $\mc{R}$-linear $\Omega$-algebra homomorphism \(\phi\colon
\mc{R}\{\Omega,X\} \longrightarrow \mc{A}\) such that \(j = \phi \circ i\),
where $i$ is the function \(X \longrightarrow \mc{R}\{\Omega,X\}\) such that
$i(x)$ is $x$, or more precisely the linear combination which has
coefficient $1$ for the formal term $x$ and coefficient $0$ for all
other terms.

A consequence of the above is that $\mc{R}\{\varnothing,X\}$ is the
\emph{free} $\mc{R}$-module with basis $X$, which might be seen as
restrictive. There is an alternative concept of free $\mc{R}$-linear
$\Omega$-algebra which is generated by an $\mc{R}$-module $\mc{M}$
rather than a set $X$, in which case the above universal property must
instead hold for $j$ being an $\mc{R}$-module homomorphism \(\mc{M}
\longrightarrow \mc{A}\); in more categoric terms, this corresponds to the
functor producing the free algebra being left adjoint of not the
forgetful functor from \texttt{$\mc{R}$-linear $\Omega$-algebra} to
\texttt{Set}, but left adjoint of the forgetful functor from
\texttt{$\mc{R}$-linear $\Omega$-algebra} to
\texttt{$\mc{R}$-module}. It is however quite possible to get to that
also by going via $\mc{R}\{\Omega,X\}$, as all one has to do is take
\(X = \mc{M}\) and then consider the quotient by the smallest
congruence relation $\equiv$ which has \(i(a) + i(b) \equiv
i(a +\nobreak b)\) and \(r i(a) \equiv i(ra)\) for all \(a,b \in
\mc{M}\) and \(r \in \mc{R}\) (it is useful here to make the
function \(i\colon X \longrightarrow \mc{R}\{\Omega,X\}\) figuring in the
universal property of $\mc{R}\{\Omega,X\}$ explicit, as $\equiv$
would otherwise seem a triviality); the result is the free object in
the category of $\mc{R}$-linear $\Omega$-algebras that are equipped
with an $\mc{R}$-module homomorphism $i'$ from $\mc{M}$, just like
the alternative universal property would require.

No doubt some readers may find this construction wasteful---a
separate constant symbol for every element of the module $\mc{M}$,
with a host of identities just to make them ``remember'' this module
structure, immediately rendering most of the symbols redundant---and
would rather prefer to construct the free $\mc{R}$-linear
$\Omega$-algebra on the $\mc{R}$-module $\mc{M}$ by direct sums of
appropriate tensor products of $\mc{M}$ with itself, somehow
generalising the tensor algebra construction \(\mathrm{T}(\mc{M}) =
\bigoplus_{n=0}^\infty \mc{M}^{\otimes n}\). However, from the
perspectives of constructive set theory and effectiveness, such
constructions are guilty of the exact same wastefulness; they only
manage to sweep it under the proverbial rug that is the definition
of the tensor product. As is quite often the case, one ends up doing
the same thing either way, although the presentation may obscure the
correspondencies between the two approaches.

Another stylistic detail is that of whether the denominator in a
quotient should be a congruence relation or an ideal. For
$\mc{R}$-linear $\Omega$-algebras, the equivalence class of $0$ turns
out to be an ideal, and conversely a congruence relation $\equiv$ is
uniquely determined by its equivalence class of $0$ since \(a \equiv
b\) if and only if \(a-b \equiv 0\). In our experience, an important
advantage of the congruence relation formalism is that it makes the
dependency on the signature $\Omega$ more explicit, since it is not
uncommon to see authors continue to associate ``ideal'' and\slash or
related concepts with the definition these have in a more traditional
setting; particularly continuing to use `two-sided ideal' and
`$\langle S\rangle$' as they would be defined in an $\bigl\{ \mathsf{m}({,})
\bigr\}$-algebra even though all objects under consideration are
really $\bigl\{ \mathsf{m}({,}), \mathsf{a}() \bigr\}$-algebras. To be explicit,
an \emph{ideal} $\mc{I}$ in an $\mc{R}$-linear $\Omega$-algebra
$\bigl( A, \{f_x\}_{x\in\Omega} \bigr)$ is an $\mc{R}$-submodule of
$A$ with the property that
\begin{multline*}
  f_x(a_1,\dotsc,a_{\mathrm{arity}(x)}) \in \mc{I}
  \quad\text{whevener \(\{a_1,\dotsc,a_{\mathrm{arity}(x)}\} \cap \mc{I} \neq
    \varnothing\),}\\
  \text{for all \(a_1,\dotsc,a_{\mathrm{arity}(x)} \in A\) and \(x \in \Omega\).}
\end{multline*}
Note that for constants $x$, the left operand of $\cap$ above is
always empty, and thus this condition does not require that (the
values of) constants would be in every ideal. It does however imply
that unary operations map ideals into themselves, and higher arity
operations take values within the ideal as soon as any operand is in
the ideal.

\subsection{Algebra constructions revisited}

Modulo some minor details, this universal algebra machinery allows
us to reproduce quickly the constructions of free hom-nonassociative
algebras, free hom-associative algebras, and universal enveloping
hom-associative algebras from Subsection~\ref{Ssec:Yau-constructions},
as well as various others that~\cite{Yau:EnvLieAlg} treat more cursory.
\index{binary tree} The plane binary trees are simply an alternative encoding of formal
terms over the signature $\bigl\{ \mathsf{m}(,) \bigr\}$; the correspondence
of one to the other is arguably not entirely trivial, but well-known,
and it is clearly the binary trees that have the weaker link to the
algebra. There is perhaps a slight mismatch in that a formal term
would encode an actual constant within each leaf, whereas the binary
trees as specified rather take the leaves to mark places where a
constant can be inserted, but we shall return to that in the next
section.

The weighting added to the trees is a method of encoding also
the $\alpha$ operation of a hom-algebra; the unstated idea is that
the weight $w(v)$ of a node $v$ specifies how many times $\alpha$
should be applied to the partial result of that node. This is thus
why grafting creates new nodes with weight $0$---grafting is
multiplication, so when the outermost operation was a multiplication,
no additional $\alpha$s are to be applied---and why $\alpha$ raises
the weight of the root node $v_{\mathrm{low}}$ only. One would like
to think of \index{weighted tree} a weighted $n$-tree as a specification of how $n$
elements in a hom-algebra are being composed---for example the term
$\alpha^3( m( \alpha(m(\alpha^2(x_1),x_2)), \alpha^4(x_3) ))$ would
correspond to the weighted $3$-tree
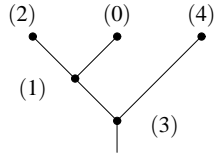
\begin{equation*}
  \setlength{\unitlength}{.14cm}
  \begin{picture}(18,15)(-2,3)
    \drawline(8,3)(8,6)(0,14)
    \drawline(4,10)(8,14)
    \drawline(8,6)(16,14)
    \put(0,14){\circle*{.7}}
    \put(16,14){\circle*{.7}}
    \put(8,14){\circle*{.7}}
    \put(8,6){\circle*{.7}}
    \put(4,10){\circle*{.7}}
    \put(-1,16){\makebox(0,0){\small$(2)$}}
    \put(8,16){\makebox(0,0){\small$(0)$}}
    \put(16,16){\makebox(0,0){\small$(4)$}}
    \put(0,9){\makebox(0,0){\small$(1)$}}
    \put(12.5,5.5){\makebox(0,0){\small$(3)$}}
  \end{picture}
\end{equation*}
---but there is a catch: weights were supposed to appear only on the
internal vertices, not on the leaves, so the above is not strictly a
weighted tree as defined in~\cite{Yau:EnvLieAlg}. This choice of
disallowing weights on leaves corresponds to the dichotomy in the
definition of $\alpha_F$ for $F_{\mathrm{HNAs}}$: as the underlying
$\alpha$ on $1$-tree terms, but as a shift $[1]$ on $n$-tree terms
for \(n>1\). This in turn corresponds to the choice of making
$F_{\mathrm{HNAs}}$ a functor from \texttt{$\mc{R}$-hom-module} to
\texttt{$\mc{R}$-hom-algebra} rather than a functor from
\texttt{$\mc{R}$-module} to \texttt{$\mc{R}$-hom-algebra}; the former
produces objects that are less free than those of the latter. It is
arguably a strength of the universal algebra method that this
distinction appears so clearly, and also a strength that it prefers
the more general approach.

What one would do in the universal algebra setting to recover the
exact same $F_{\mathrm{HNAs}}(A)$ as Yau defined is to impose
\(\mathsf{a}(x) \equiv \alpha(x)\) for all \(x \in A\) as conditions upon a
congruence relation $\equiv$, and then take the quotient by that.
Technically, one would start out with the free $\mc{R}$-linear
$\Omega$-algebra $\mathcal{R}\{\Omega,A\}$ and impose upon it (in
addition to \(\mathsf{a}(x) \equiv \alpha(x)\)) the silly-looking
congruences
\begin{align} \label{Eq:linearity}
  rx \equiv{}& (rx) \text{,}&
  x+y \equiv{}& (x+y) &&
  \text{for all \(x,y \in A\) and \(r \in \mc{R}\);}
\end{align}
the technical point here is that addition and multiplication in the
left hand sides refer to the operations in $\mathcal{R}\{\Omega,A\}$,
whereas those on the right hand side refer to operations in $A$. What
happens is effectively the same as in the set-theoretic construction
of tensor product of modules.
\begin{subequations} \label{Eqs:UHLie}
Similarly, to recover the hom-associative $F_{\mathrm{HAs}}(A)$ one
would start out with $\mathcal{R}\{\Omega,A\}$ for
\(\Omega = \{ \mathsf{a}(), \mathsf{m}(,) \}\) and
quotient that by the smallest $\mathcal{R}$-linear $\Omega$-algebra
congruence relation $\equiv$ satisfying the linearity condition
\eqref{Eq:linearity} and
\begin{align}
  \mathsf{a}(x) \equiv{}& \alpha(x) && \text{for all \(x \in A\),}
    \label{Eq:UHLie:alpha}\\
  \mathsf{m}\bigl( \mathsf{a}(t_1), \mathsf{m}(t_2,t_3) \bigr) \equiv{}&
    \mathsf{m}\bigl( \mathsf{m}(t_1,t_2) , \mathsf{a}(t_3) \bigr)
    && \text{for all \(t_1,t_2,t_3 \in T(\Omega,A)\).}
    \label{Eq:UHLie:hom-ass}
\end{align}
Finally, in order to recover $U_{\mathrm{HLie}}(\mathfrak{g})$ for the
hom-Lie algebra \(\mathfrak{g} = (A, b, \alpha)\), one needs only
impose also the condition
\begin{equation} \label{Eq:UHLie:commutator}
  \mathsf{m}(x,y) - \mathsf{m}(y,x) \equiv b(x,y)
  \qquad\text{for all \(x,y \in A\)}
\end{equation}
on the congruence relation $\equiv$. What in this step has been
noticeably simplified in comparison to the presentation of
Subsection~\ref{Ssec:Yau-constructions} is that the infinite sequence
of alternatingly generating two-sided ideals and applying $\alpha_F$
has been compressed into just one operation, namely that of forming
the generated congruence relation. This has not made the whole thing
more effective, but it greatly simplifies reasoning about it.
\end{subequations}

For the reader approaching the above as was it a deformation of the
associative universal enveloping algebra of a Lie algebra, it might
instead be more natural to impose the conditions in the order
\eqref{Eq:UHLie:hom-ass} first, \eqref{Eq:UHLie:commutator} second,
and \eqref{Eq:UHLie:alpha} last. Doing so might also raise the
question of why one should stop there, as opposed to imposing some
additional condition on $\mathsf{a}$, such as \(\mathsf{a}\bigl( \mathsf{m}(t_1,t_2)
\bigr) \equiv \mathsf{m}\bigl( \mathsf{a}(t_1), \mathsf{a}(t_2) \bigr)\)? The reason
not to ask for that particular condition is that it forces the
resulting hom-algebra to be multiplicative, and it is easily checked
that if $\mathcal{A}$ is a multiplicative hom-algebra, then
$\mathcal{A}^-$ is multiplicative as well; doing so would immediately
destroy all hope of getting an enveloping algebra, unless the hom-Lie
algebra one started with was already multiplicative.

For a hom-Lie algebra presented in terms of a basis, such as the
$q$-deformed $\mathfrak{sl}_2$ of \eqref{Eq:q-def-sl2}, it is usually
more natural to seek its $U_{\mathrm{HLie}}$ by starting with only the
basis elements as constant symbols. In that example one would instead
take \(X = \{\mathsf{e},\mathsf{f},\mathsf{h}\}\) and seek a
congruence relation on $\mathbb{K}\{\Omega,X\}$, namely that which satisfies
\begin{subequations}
  \begin{align} \label{Eq:UHLie-sl2:alpha}
    \mathsf{a}(\mathsf{e}) \equiv{}& q\,\mathsf{e} \text{,}&
    \mathsf{a}(\mathsf{f}) \equiv{}& q^2\mathsf{f} \text{,}&
    \mathsf{a}(\mathsf{h}) \equiv{}& q\,\mathsf{h} \text{,}
  \end{align}
  \begin{equation} \label{Eq:UHLie-sl2:hom-ass}
    \mathsf{m}\bigl( \mathsf{a}(t_1), \mathsf{m}(t_2,t_3) \bigr) \equiv
    \mathsf{m}\bigl( \mathsf{m}(t_1,t_2) , \mathsf{a}(t_3) \bigr)
    \qquad\text{for all \(t_1,t_2,t_3 \in T(\Omega,X)\),}
  \end{equation}
  \begin{multline} \label{Eq:UHLie-sl2:commutator}
    \mathsf{m}(\mathsf{e},\mathsf{f}) - \mathsf{m}(\mathsf{f},\mathsf{e})
      \equiv \tfrac{1}{2}(1+q)\mathsf{h} \text{,} \qquad
    \mathsf{m}(\mathsf{e},\mathsf{h}) - \mathsf{m}(\mathsf{h},\mathsf{e})
      \equiv -2\mathsf{e} \text{,} \\
    \mathsf{m}(\mathsf{h},\mathsf{f}) - \mathsf{m}(\mathsf{f},\mathsf{h})
      \equiv -2q\,\mathsf{f} \text{.}
  \end{multline}
\end{subequations}
It suffices to impose hom-associativity for monomial terms (those
that can be formed using $\mathsf{a}$, $\mathsf{m}$, and elements of $X$ only) as
anything else is a finite linear combination of such terms.

In these equations, it should be observed that
\eqref{Eq:UHLie-sl2:alpha} and \eqref{Eq:UHLie-sl2:commutator} are
three discrete conditions each,
whereas \eqref{Eq:UHLie-sl2:hom-ass} imposing hom-associativity is an
infinite family of conditions. This is mirrored in \eqref{Eqs:UHLie}
by the difference in ranges: in \eqref{Eq:UHLie:alpha}, $x$ ranges
only over elements of $A$ (i.e., terms that are constants), but in
\eqref{Eq:UHLie:hom-ass} the variables range over arbitrary
terms. Comparing this to presentations of \emph{associative} algebras
on the form $\mc{R}\langle x,y,z \mid\nobreak \dots \rangle$, the
discrete conditions are like prescribing a relation between the
generators $x$, $y$, and $z$, whereas the infinite family used for
hom-associativity is like prescribing a Polynomial Identity for the
algebra. In rewriting theory, one would rather say
\eqref{Eq:UHLie-sl2:alpha} and \eqref{Eq:UHLie-sl2:commutator} are
equations of \emph{ground terms} whereas \eqref{Eq:UHLie-sl2:hom-ass}
is an equation involving variables (note that this is a different
sense of `variable' than in `variable' as generator of $\mc{R}\langle
x,y,z\rangle$).

The exact same analysis can be carried out for the hom-dialgebras and
diweighted trees of~\cite[Secs.~5--6]{Yau:EnvLieAlg}; the main point
of deviation is merely that one starts out the signature $\bigl\{
\mathsf{a}(), \mathsf{l}(,), \mathsf{r}(,) \bigr\}$ (because a dialgebra has separate
left multiplication~$\dashv$ and right multiplication~$\vdash$)
rather than the hom-algebra signature $\bigl\{ \mathsf{a}(), \mathsf{m}(,)
\bigr\}$. The diweighted tree encoding takes another step away from
the canonical formal terms by bundling into the weight the left\slash
right nature of each multiplication with the number of $\alpha$s to
apply after it. This is not quite as \emph{ad hoc} as it may seem,
because in non-hom dialgebras the associativity-like axioms have the
effect that general products of $n$ elements look like $(\dotsb (x_1
\vdash x_2) \vdash \dotsb ) \vdash x_m \dashv ( \dotsb \dashv
(x_{n-1} \dashv x_n) \dotsb )$; the left\slash right nature of a
multiplication is pretty much determined by its position in relation
to the switchover factor $x_m$, so there it makes sense to seek a
mostly unified encoding of the two. It is however far from clear that
the same would be true also for general hom-dialgebras; free
hom-associative algebras are certainly far more complicated than free
associative algebras.

\section{A newer setting: free operads}
\label{Sec:OperadUA}

One awkward point above is that for example the hom-associativity
axiom, despite in some sense being just one identity, required an
infinite family of equations to be imposed upon the free hom-associative
algebra; shouldn't there be a way of imposing it in just one step?
Indeed there is, but it requires broadening one's view, and to think
in terms of operads rather than algebras. A programme for this was
outlined in~\cite{ROP}.

\subsection{What is an operad?}

Nowadays, many introductions to the operad concept are available,
for example \cite{LodayValette,MSS:Operads,WhatIsAnOperad}.
What is important for us to stress is the analogy with associative
algebras: Operators acting on (say) a vector space can be added
together, taken scalar multiples of, and composed; any given set of
operators will generate an associative algebra under these
operations. When viewed as functions, operators are only univariate
however, so one might wonder what happens if we instead consider
multivariate functions (still mapping some number of elements from a
vector space into that same space)? One way of answering that
question is that we get an operad.

Composition in operads work as when one uses dots `$\cdot$' to mark
the position of ``an argument'' in an expression: From the bivariate
functions $f(\cdot,\cdot)$ and $g(\cdot,\cdot)$, one may construct
the compositions $f(g(\cdot,\cdot),\cdot)$, $f(\cdot,g(\cdot,\cdot))$,
$g(f(\cdot,\cdot),\cdot)$, and $g(\cdot,f(\cdot,\cdot))$, which are
all trivariate. Note in particular that the ``variable-based'' style
of composition that permits forming e.g.~the bivariate function
\((x,y) \mapsto f\bigl( g(x,y), y\bigr)\) from $f$ and $g$ is not
allowed in an operad, because it destroys multilinearity; \(f(x,y) =
xy\) is a bilinear map \(\mathbb{R}^2 \longrightarrow \mathbb{R}\), but \(h(x) = f(x,x) = x^2\)
is nonlinear.\footnote{
  It may then seem serendipitous that Cohn~\cite[p.~127]{Cohn:UA}
  citing Hall calls an algebraic structure with the variable-based
  form of composition a \emph{clone}, since it gets its extra power
  from being able to ``clone'' input data, but he explains it as being
  a contraction of `closed set of operations'. In the world of
  Quantum Mechanics, the well-known `No cloning' theorem forbids that
  kind of behaviour (essentially because it violates multilinearity),
  so by sticking to operads we take the narrow road.
} In an expression that composes several operad elements into one,
one is however usually allowed to choose where the various arguments
are used: $g(f(x_1,x_2),x_3)$, $g(f(x_2,x_1),x_3)$,
$g(f(x_3,x_1),x_2)$, etc.\@ are all possible as operad elements. This
is formalised by postulating a right action of the group $\Sigma_n$ of
permutations of $\{1,\dotsc,n\}$ on those operad elements which take
$n$ arguments; in function notation one would have
\(f(x_{\sigma^{-1}(1)}, \dotsc, x_{\sigma^{-1}(n)}) =
(f\sigma)(x_1,\dotsc,x_n)\).

More formally, \index{operad} an operad $\mc{P}$ is a family $\bigl\{ \mc{P}(n)
\bigr\}_{n\in\mathbb{N}}$ of sets, where $\mc{P}(n)$ is ``the set of those
operad elements which have arity $n$''. Alternatively, an operad
$\mc{P}$ can be viewed as a set with an arity function, in which case
$\mc{P}(n)$ is a shorthand for $\setOf[\big]{ a \in \mc{P} }{
\mathrm{arity}(a)=n }$. Both approaches are (modulo some formal nonsense)
equivalent, and we will employ both since some concepts are easier
under one approach and others are easier under the other.

Composition can be given the
form of composing one element \(a \in \mc{P}(m)\) with the $m$
elements \(b_i \in \mc{P}(n_i)\) for \(i=1,\dotsc,m\) (i.e., one for
each ``argument'' of $a$) to form \(a \circ b_1 \otimes \dotsb
\otimes b_m \in \mc{P}\bigl( \sum_{i=1}^m n_i \bigr)\); note that the
`$\circ$' and the $m-1$ `$\otimes$' are all part of the same operad
composition. (There is a more general concept called PROP where
$b_1 \otimes \dotsb \otimes b_m$ would be an actual element, but we
won't go into that here.) Operad composition is associative in the
sense that the unparenthesized expression
\begin{equation*}
  a \circ b_1 \otimes \dotsb \otimes b_\ell \circ
  c_1 \otimes \dotsb \otimes c_m
\end{equation*}
is the same whether it is interpreted as
\begin{equation*}
  (a \circ b_1 \otimes \dotsb \otimes b_\ell) \circ
  c_1 \otimes \dotsb \otimes c_m
\end{equation*}
or as
\begin{equation*}
  a \circ
  (b_1 \circ c_1 \otimes \dotsb \otimes c_{m_1}) \otimes
  \dotsb \otimes
  (b_\ell \circ c_{m_1+\dotsb+m_{\ell-1}+1} \otimes \dotsb \otimes
    c_{m_1+\dotsb+m_\ell})
\end{equation*}
where \(m = \sum_{i=1}^\ell m_i\) and \(b_i \in \mc{P}(m_i)\) for
\(i=1,\dotsc,\ell\).\footnote{
  As the number of ellipses ($\dots$) above indicate, the axioms for
  operads are somewhat awkward to state, even though they only
  express familiar properties of multivariate functions. The PROP
  formalism may therefore be preferable even if one is only
  interested in an operad setting, since the PROP axioms can be
  stated without constantly going `$\dots$'.
}

Since $\Sigma_n$ acts on the right of each $\mc{P}(n)$, this action
satisfies \((a\sigma)\tau = a(\sigma\tau)\) for all
\(a \in \mc{P}(n)\) and \(\sigma,\tau \in \Sigma_n\). There is also a
condition called equivariance that
\begin{equation*}
  (a \sigma) \circ b_1 \otimes \dotsb \otimes b_m =
  (a \circ b_{\sigma^{-1}(1)} \otimes \dotsb \otimes
  b_{\sigma^{-1}(m)}) \tau
\end{equation*}
where $\tau$ is a block version of $\sigma$, such that the $k$th
block has size equal to the arity of $b_k$. Finally, it is usually
also required that there is an identity element \(\mathrm{id} \in \mc{P}(1)\)
such that \(\mathrm{id} \circ a = a = a \circ \mathrm{id}^{\otimes n}\) for all \(a
\in \mc{P}(n)\) and \(n \in \mathbb{N}\).

\begin{example}
  For every set $A$, there is an operad $\mathrm{Map}_A$ such that
  $\mathrm{Map}_A(n)$ is the set of all maps \(A^n \longrightarrow A\); in particular,
  $\mathrm{Map}_A(0)$ may be identified with $A$. For \(a \in \mathrm{Map}_A(m)\) and
  \(b_i \in \mathrm{Map}_A(n_i)\) for \(i=1,\dotsc,m\), the composition $a
  \circ b_1 \otimes \dotsb \otimes b_m$ is defined by
  \begin{multline*}
    (a \circ b_1 \otimes \dotsb \otimes b_m)
      (x_{1,1},\dotsc,x_{1,n_1},\dotsc,x_{m,1},\dotsc,x_{m,n_m})
    = \\ =
    a\bigl( b_1(x_{1,1},\dotsc,x_{1,n_1}), \dotsc,
      b_m(x_{m,1},\dotsc,x_{m,n_m}) \bigr)
  \end{multline*}
  for all \(x_{1,1},\dotsc,x_{m,n_m} \in A\). \(\mathrm{id} \in
  \mathrm{Map}_A(1)\) is the identity map on $A$. The permutation action
  is defined by \((a\sigma)(x_1,\dotsc,x_n) =
  a(x_{\sigma^{-1}(1)}, \dotsc, x_{\sigma^{-1}(n)})\).
\end{example}

An alternative notation for composition is \(\gamma(a,b_1,\dotsc,b_m)
= a \circ b_1 \otimes \dotsb \otimes b_m\); that $\gamma$ is then
called the \emph{structure map}, or \emph{structure maps} if one
requires each map to have a signature on the form \(\mc{P}(m) \times
\mc{P}(n_1) \times \dotsb \times \mc{P}(n_m) \longrightarrow
\mc{P}\bigl( \sum_{i=1}^m n_i \bigr)\). An alternative composition
\emph{concept} is the \emph{$i$th composition} $\circ_i$, which
satisfies \(a \circ_i b = a \circ \mathrm{id}^{\otimes (i-1)} \otimes b \otimes
\mathrm{id}^{\otimes (m-i)}\) for \(a \in \mc{P}(m)\) and \(i=1,\dotsc,m\).
Note that $i$th composition, despite being a binary operation, is not
at all associative in the usual sense and expressions involving it must
therefore be explicitly parenthesized; operad associativity does
however imply that subexpressions can be regrouped (informally:
``parentheses can be moved around'') provided that the position
indices are adjusted accordingly.

An \emph{operad homomorphism} \(\phi\colon \mc{P} \longrightarrow \mc{Q}\) is a
map that is compatible with the operad structures of $\mc{P}$ and
$\mc{Q}$: \(\mathrm{arity}_{\mc{Q}}\bigl( \phi(a) \bigr) = \mathrm{arity}_{\mc{P}}(a)\),
\(\phi(a \circ\nobreak b_1 \otimes\nobreak \dotsb \otimes\nobreak b_m) =
\phi(a) \circ \phi(b_1) \otimes \dotsb \otimes \phi(b_m)\),
\(\phi(a\sigma) = \phi(a)\sigma\), and \(\phi(\mathrm{id}_{\mc{P}}) =
\mathrm{id}_{\mc{Q}}\) for all \(a \in \mc{P}(m)\), \(b_i \in \mc{P}\) for
\(i=1,\dotsc,m\), \(\sigma \in \Sigma_m\), and \(m\in\mathbb{N}\). A
\emph{suboperad} of $\mc{P}$ is a subset of $\mc{P}$ that is closed
under composition, closed under permutation action, and contains the
identity element. The operad \emph{generated} by some \(\Omega
\subseteq \mc{P}\) is the smallest suboperad of $\mc{P}$ that
contains $\Omega$.

Let $\mc{R}$ be an associative and commutative unital ring. An operad
$\mc{P}$ is said to be \emph{$\mc{R}$-linear} if (i)~each $\mc{P}(n)$
is an $\mc{R}$-module, (ii)~every structure map \((a,b_1,\dotsb,b_m)
\mapsto a \circ b_1 \otimes \dotsb \otimes b_m\) is $\mc{R}$-linear
in each argument separately, and (iii)~each action of a permutation
is $\mc{R}$-linear.

\begin{example}
  The $\mathrm{Map}_A$ operad is in general not $\mc{R}$-linear, but if $A$
  is an $\mc{R}$-module, then the suboperad $\mathcal{E}\mkern-2mu\mathit{nd}_A$ where
  $\mathcal{E}\mkern-2mu\mathit{nd}_A(n)$ consists of all $\mc{R}$-multilinear maps \(A^n \longrightarrow
  A\) will be $\mc{R}$-linear. $\mathcal{E}\mkern-2mu\mathit{nd}_A(0)$ can also be identified
  with $A$.
\end{example}

The operad concept defined above is sometimes called \index{symmetric operad} a
\emph{symmetric} operad, because of the actions on it of the symmetric
groups. Dropping everything involving permutations above, one instead
arrives at the concept of a \emph{nonsymmetric} or
\emph{non-$\Sigma$} operad. Much of what is done below could just as
well be done in the non-$\Sigma$ setting, but we find the symmetric
setting to be more akin to classical universal algebra.

\subsection{Universal algebra for operads}

Regarding universal algebra, an interesting thing about operads is
that they may serve as generalisations of both the algebra concept
and the signature concept. The way that an operad $\mc{P}$ may
generalise a signature $\Omega$ is that a set $A$ is said to be a
\emph{$\mc{P}$-algebra} if it is given with an operad homomorphism
\(\phi\colon \mc{P} \longrightarrow \mathrm{Map}_A\); the operation $f_x$ of some \(x
\in \mc{P}\) is then simply $\phi(x)$. Being an operad-algebra is
however a stronger condition than being a signature-algebra, because
the map $\phi$ will only be a homomorphism if every identity in
$\mc{P}$ is also satisfied in $\phi(\mc{P})$; this can be used to
impose ``laws'' on algebras, and several elementary operads are
defined to precisely this purpose: an algebra is an $\mathcal{A}\mkern-1mu\mathit{ss}$-algebra
iff it is associative, a $\mathcal{C}\mkern-1mu\mathit{om}$-algebra iff it is commutative, a
$\mathcal{L}\mkern-1mu\mathit{ie}$-algebra iff it is a Lie algebra, a $\mathcal{L}\mkern-1mu\mathit{eib}$-algebra
iff it is a Leibniz algebra, and so on. It is therefore only natural
that we will shortly construct an operad $\mathcal{H\mkern-4muA}\mkern-1mu\mathit{ss}$ whose algebras
are precisely the hom-associative algebras.

Before taking on that problem, we should however give an example of
how identities in an operad become laws of its algebras. To that end,
consider $\mathbb{N}$ as an operad by making \(\mathrm{arity}(n)=n\); this uniquely
defines the operad structure, since the arity of any particular
composition is given by the axioms, and that in turn determines the
value since every $\mathbb{N}(n)$ only has one element. What can now be said
about an $\mathbb{N}$-algebra $A$ if \(f\colon \mathbb{N} \longrightarrow \mathrm{Map}_A\) is the given
operad homomorphism? Clearly $f(2)\colon A^2 \longrightarrow A$ is a binary
operation. If \(\tau \in \Sigma_2\) is the transposition, one
furthermore finds that
\begin{equation*}
  f(2)(x,y) = f(2 \tau)(x,y) = \bigl( f(2) \tau \bigr)(x,y) =
  f(2)(y,x)
\end{equation*}
for all \(x,y \in A\), so $f(2)$ is commutative. Similarly it follows
from \(2 \circ 1 \otimes 2 = 3 = 2 \circ 2 \otimes 1\) that
\(f(2)\bigl( x, f(2)(y,z) \bigr) = f(2)\bigl( f(2)(x,y), z \bigr)\)
for all \(x,y,z \in A\), and thus $f(2)$ is associative. Finally one
may deduce from \(\mathrm{id} = 1 = 2 \circ 0 \otimes 1\) that $f(0)$ is a
unit element with respect to $f(2)$, so in summary any $\mathbb{N}$-operad
algebra carries an abelian monoid structure. This is almost the same
as $\mathcal{C}\mkern-1mu\mathit{om}$ is supposed to accomplish, so one might ask whether
in fact \(\mathcal{C}\mkern-1mu\mathit{om} = \mathbb{N}\), but traditionally $\mathcal{C}\mkern-1mu\mathit{om}$,
$\mathcal{A}\mkern-1mu\mathit{ss}$, etc.\@ are taken to be the $\mc{R}$-linear (for
whatever ring $\mc{R}$ of scalars is being considered) operads that
impose the indicated laws on their algebras. $\mathcal{C}\mkern-1mu\mathit{om}$ is thus
rather characterised by having \(\dim\mathcal{C}\mkern-1mu\mathit{om}(n) = 1\) for all $n$,
and may if one wishes be constructed as \(\mc{R} \times \mathbb{N}\).

While specific operads may sometimes be constructed through elementary
methods as above, the general approach to constructing an operad that
corresponds to a specific set of laws is instead the universal
algebraic one, which rather employs the point of view that an
operad is a generalisation of an algebra. Obviously any specific
$\Omega$-algebra $\bigl( A, \{f_x\}_{x\in\Omega} \bigr)$ gives rise
to the operad $\mathrm{Map}_A$, but the operad that more naturally
generalises $A$ as an $\Omega$-algebra is the suboperad of $\mathrm{Map}_A$
that is generated by $\{f_x\}_{x\in\Omega}$. Conversely, if $A$ is
supposed to be some kind of free algebra, one may choose to construct
it as the constant component of the corresponding free operad.

An equivalence relation $\equiv$ on an operad $\mc{P}$ is \index{operad congruence relation} an
\emph{operad congruence relation} if:
\begin{enumerate}
  \item
    \(a \equiv a'\) implies \(\mathrm{arity}(a) = \mathrm{arity}(a')\),
  \item
    \(a \equiv a'\) and \(b_i \equiv b_i'\) for \(i=1,\dotsc,\mathrm{arity}(a)\)
    implies \(a \circ b_1 \otimes \dotsb \otimes b_{\mathrm{arity}(a)} \equiv
    a' \circ b_1' \otimes \dotsb \otimes b_{\mathrm{arity}(a)}'\), and
  \item
    \(a \equiv a'\) implies \(a\sigma \equiv a'\sigma\) for all
    \(\sigma \in \Sigma_{\mathrm{arity}(a)}\).
\end{enumerate}
As for algebras, it follows that the quotient \(\mc{P}/{\equiv}\)
carries an operad structure, and the canonical map \(\mc{P} \longrightarrow
\mc{P}/{\equiv}\) is an operad homomorphism. If additionally $\mc{P}$
is $\mc{R}$-linear and $\equiv$ is an $\mc{R}$-module congruence
relation on each $\mc{P}(n)$, then $\equiv$ is an
\emph{$\mc{R}$-linear operad congruence relation} and the
corresponding \index{operad ideal} \emph{operad ideal} $\mc{I}$ is defined by \(\mc{I}(n)
= \setOf[\big]{ a \in \mc{P}(n) }{ a \equiv 0 }\) for all \(n \in
\mathbb{N}\) (note that each $\mc{P}(n)$ has a separate $0$ element).
Equivalently, \(\mc{I} \subseteq \mc{P}\) is an operad ideal if each
$\mc{I}(n)$ is a submodule of $\mc{P}(n)$, each $\mc{I}(n)$ is closed
under the action of $\Sigma_n$, and \(a \circ b_1 \otimes \dotsb
\otimes b_m \in \mc{I}\) whenever at least one of $a,b_1,\dotsc,b_m$
is an element of $\mc{I}$.

So far, the operad formalism is very similar to that for algebras,
but an important difference occurs when one wishes to impose laws on
a congruence. For an algebra, the hom-associativity condition
\eqref{Eq:Kongr.ass.algebra} required an infinite family of
identities. The corresponding condition in the operad $\mathrm{Map}_A$
requires only the single identity \(f_\mathsf{m} \circ f_\mathsf{a} \otimes
f_\mathsf{m} \equiv f_\mathsf{m} \circ f_\mathsf{m} \otimes f_\mathsf{a}\), as the infinite
family is recovered from this using composition on the right:
\(f_\mathsf{m} \circ f_\mathsf{a} \otimes f_\mathsf{m} \circ t_1 \otimes t_2 \otimes t_3
\equiv f_\mathsf{m} \circ f_\mathsf{m} \otimes f_\mathsf{a} \circ t_1 \otimes t_2
\otimes t_3\). The $\mathcal{A}\mkern-1mu\mathit{ss}$, $\mathcal{C}\mkern-1mu\mathit{om}$, $\mathcal{L}\mkern-1mu\mathit{eib}$, etc.\@
operads can all be seen to be finitely presented, and the same holds
for their free algebras if generated as the arity~$0$ component of an
operad, even though they are not finitely presented within the
$\Omega$-algebra formalism!

The universal property satisfied by the free operad $\mc{F}$ on
$\Omega$ is that it is given with an arity-preserving map \(i\colon
\Omega \longrightarrow \mc{F}\) such that there for every operad $\mc{P}$ and
every arity-preserving map \(j\colon \Omega \longrightarrow \mc{P}\) exists a
unique operad homomorphism \(\phi\colon \mc{F} \longrightarrow \mc{P}\) such
that \(j = \phi \circ i\). A practical construction of that free
operad is to let $\mc{F}(n)$ be the set of all $n$-variable
\emph{contexts}~\cite[p.~17]{TATA}, but since we'll anyway need some
notation for these, we might as well give an explicit definition
based on Polish notation for expressions.

\begin{definition}
  A (left-)\emph{Polish term} on the signature $\Omega$ is a finite
  word on $\Omega \cup \{\Box_i\}_{i=1}^\infty$ (where it is presumed
  that \(\Box_i \notin \Omega\) and \(\mathrm{arity}(\Box_i) = 0\) for all $i$),
  which is either $\Box_i$ for some \(i \geqslant 1\), or
  $x \mu_1 \dotsb \mu_n$ where \(x \in \Omega\), \(n = \mathrm{arity}(x)\), and
  \(\mu_1,\dotsc,\mu_m\) are themselves Polish terms on $\Omega$.
  A Polish term is an \emph{$n$-context} if each symbol
  $\Box_i$ for \(i=1,\dotsc,n\) occurs exactly once and no symbol
  $\Box_i$ with \(i>n\) occurs at all.
  For $\Box_1,\dotsc,\Box_9$ we will write $1,\dotsc,9$ for short.
  Denote by $\mc{Y}_\Omega(n)$ the set of all $n$-contexts on
  $\Omega$.

  The action of \(\sigma \in \Sigma_n\) on $\mc{Y}_\Omega(n)$ is that
  each $\Box_i$ is replaced by $\Box_{\sigma^{-1}(i)}$. The
  composition $\mu \circ \nu_1 \otimes \dotsb \otimes \nu_n$ is a
  combined substitution and renumbering: first each $\Box_i$ in $\mu$
  is replaced by the corresponding $\nu_i$, then the $\Box_k$'s in
  the composite term are renumbered so that the term becomes a
  context---preserving the differences within each $\nu_i$ and giving
  $\Box_k$'s from $\nu_i$ lower indices than those from $\nu_j$
  whenever \(i<j\).

  For any associative and commutative unital ring $\mc{R}$, and for
  every \(n\in\mathbb{N}\), denote by $\mc{R}\{\Omega\}(n)$ the set of all
  formal $\mc{R}$-linear combinations of elements of
  $\mc{Y}_\Omega(n)$. Extend the action of \(\sigma \in \Sigma_n\) on
  $\mc{Y}_\Omega(n)$ to $\mc{R}\{\Omega\}(n)$ by linearity. Let
  \(\mc{R}\{\Omega\} = \bigcup_{n\in\mathbb{N}} \mc{R}\{\Omega\}(n)\). Extend
  the composition on $\mc{Y}_\Omega$ to $\mc{R}\{\Omega\}$ by
  multilinearity. When $\mc{Y}_\Omega$ is viewed as a subset of
  $\mc{R}\{\Omega\}$, its elements are called \emph{monomials}.

  With \(\mathrm{id} = 1 = \Box_1\), this makes $\mc{Y}_\Omega$ the free
  operad on $\Omega$ and $\mc{R}\{\Omega\}$ is the free $\mc{R}$-linear
  operad on $\Omega$.
\end{definition}

For \(\Omega = \bigl\{ \mathsf{x}, \mathsf{a}(), \mathsf{m}({,}) \bigr\}\), one may
thus find in $\mc{Y}_\Omega(0)$ elements such as $\mathsf{x}$, $\mathsf{a}\mathsf{x}$,
$\mathsf{m}\mathsf{x}\mathsf{x}$, $\mathsf{a}\mathsf{m}\mathsf{x}\mathsf{x}$, and $\mathsf{m}\mathsf{a}\mathsf{x}\mathsf{x}$ which in
parenthesized notation would rather have been written as $\mathsf{x}$,
$\mathsf{a}(\mathsf{x})$, $\mathsf{m}(\mathsf{x},\mathsf{x})$, $\mathsf{a}(\mathsf{m}(\mathsf{x},\mathsf{x}))$, and
$\mathsf{m}(\mathsf{a}(\mathsf{x}),\mathsf{x})$ respectively. In $\mc{Y}_\Omega(1)$ we
similarly find $1$, $\mathsf{a}1$, $\mathsf{a}\mathsf{a}1$, $\mathsf{m}\mathsf{x}1$, $\mathsf{m}1\mathsf{x}$,
and $\mathsf{m}\mathsf{a}\mathsf{x}\mathsf{m}1\mathsf{x}$ which in parenthesized notation could
have been written as $\Box_1$, $\mathsf{a}(\Box_1)$, $\mathsf{a}(\mathsf{a}(\Box_1))$,
$\mathsf{m}(\mathsf{x},\Box_1)$, $\mathsf{m}(\Box_1,\mathsf{x})$, and
$\mathsf{m}(\mathsf{a}(\mathsf{x}),\mathsf{m}(\Box_1,\mathsf{x}))$. In $\mc{R}\{\Omega\}(2)$ there
are elements such as $\mathsf{m}12 - \mathsf{m}21$ and $\mathsf{m}12 + \mathsf{m}21$ which
would be mapped to $0$ by any operad homomorphism $f$ to $\mathrm{Map}_A$ for
which $f(\mathsf{m})$ is commutative or anticommutative respectively.
Finally there is in $\mc{R}\{\Omega\}(3)$ the elements $\mathsf{m}1\mathsf{m}23 -
\mathsf{m}\mathsf{m}123$ and $\mathsf{m}\mathsf{a}1\mathsf{m}23 - \mathsf{m}\mathsf{m}12\mathsf{a}3$ which have
similar roles with respect to associativity and hom-associativity
respectively.

A practical problem, which is mostly common to the Polish and the
parenthesized notations, is that it can be difficult to grasp the
structure of one of these expressions just from a quick glance at the
written forms of them; small expressions may be immediately
recognised by the trained eye, but larger expressions almost always
require a conscious effort to parse. This is unfortunate, as the
exact structure is very important when working in a setting this
general. The structure can however be made more visible by
\emph{drawing} expressions rather than \emph{writing} them;
informally one depicts an expression using its abstract syntax tree,
but those of a more formalistic persuasion may think of these
drawings as graph-theoretical objects underlying the trees~(in the
sense of~\cite[pp.~15--16]{TATA}) of these terms. A few examples can
be
\begin{align*}
  \mathsf{m}12 ={}& \left[ \begin{sdpgf}{0}{0}{60}{-124}{0.2pt}
    \m 30 -119 \L 0 41 \S \m 19 -51 \C -4 4 0 27 0 15 \S \m 41 -51 \C
    4 4 0 27 0 15 \S \ov 14 -78 32 32 \S
  \end{sdpgf} \right] &
  \mathsf{m}21 ={}& \left[ \begin{sdpgf}{0}{0}{60}{-124}{0.2pt}
    \m 41 -51 \C 20 20 -46 10 0 16 \S \m 19 -51 \C -20 20 46 10 0 16
    \S \m 30 -119 \L 0 41 \S \ov 14 -78 32 32 \S
  \end{sdpgf} \right] &
  \mathsf{m}\mathsf{m}312 ={}& \left[ \begin{sdpgf}{0}{0}{90}{-196}{0.2pt}
    \m 41 -51 \C 20 20 -46 10 0 16 \S \m 56 -123 \C 13 13 6 20 0 18 \L
    0 20 \C 0 19 -30 9 0 19 \S \m 19 -51 \C -34 34 90 -6 0 18 \S \m 34
    -123 \C -4 4 0 26 0 15 \S \m 45 -191 \L 0 41 \S \ov 14 -78 32 32
    \S \ov 29 -150 32 32 \S
  \end{sdpgf} \right]
\end{align*}
and several more can be found below. A Polish term may even be read
as a direct instruction for how to draw these trees: in order to draw
\(\mu = x \nu_1 \dotsb \nu_{\mathrm{arity}(x)}\), first draw a vertex for $x$ as
the root, and then draw the subtrees $\nu_1$ through $\nu_{\mathrm{arity}(x)}$
above the $x$ vertex and side by side, letting the order of edges
along the top of a vertex show the order of the subexpressions. The
``inputs'' $\Box_k$ of a context are represented by edges to the top
side of the drawing, with $\Box_1$ being leftmost, $\Box_2$ being
second to left, and so on.

\begin{definition}
  An element of $\mc{Y}_\Omega(n)$ is said to be \emph{plane} if the
  $\Box_i$ symbols (if any) occur in ascending order: none to the
  left of $\Box_1$, only $\Box_1$ to the left of $\Box_2$, and so on.
  (Equivalently, the drawing procedure described above will not
  produce any crossing edges.) An element of $\mc{R}\{\Omega\}(n)$ is
  \emph{plane} if it is a linear combination of plane elements.
  An element of $\mc{R}\{\Omega\}(n)$ is \emph{planar} if it is of
  the form $a \sigma$ for some plane \(a \in \mc{R}\{\Omega\}(n)\)
  and \(\sigma \in \Sigma_n\). Finally, an ideal in
  $\mc{R}\{\Omega\}$ is said to be \emph{planar} if it is generated
  by planar elements.
\end{definition}

Elements in a planar ideal need not be planar, but every element in
a planar ideal can be written as a sum of planar elements that are
themselves in the ideal.

\subsection{The Diamond Lemma for operads}
\label{Ssec:DL}

This and the following sections rely heavily on results and concepts
from~\cite{GFDL}. We try to always give a reference, where a concept is
first used that will not be explained further here, to the exact
definition in~\cite{GFDL} of that concept.

Let a signature $\Omega$ and an associative and commutative unital
ring $\mc{R}$ be given. Consider the free $\mc{R}$-linear operad
$\mc{R}\{\Omega\}$ and its suboperad of monomials $\mc{Y}_\Omega$.
Let $V(i,j)$ be the set of all maps \(\mc{R}\{\Omega\}(j) \longrightarrow
\mc{R}\{\Omega\}(i)\) that are on the form
\begin{equation} \label{Eq:V-map}
  a \mapsto \bigl( \lambda \circ_k
  (a \circ \nu_1 \otimes \dotsb \otimes \nu_j) \bigr) \sigma
\end{equation}
where \(\nu_r \in \mc{Y}_\Omega(n_r)\) for \(r=1,\dotsc,j\),
\(\lambda \in \mc{Y}_\Omega(\ell)\), \(\ell \geqslant k \geqslant 1\),
\(\sigma \in \Sigma_i\), and \(i = \ell-1 + n_1 + \dotsc + n_j\).
The family \(V = \bigcup_{i,j \in \mathbb{N}} V(i,j)\)
is then a category~\cite[Def.~6.8]{GFDL}, and each \(v \in V(i,j)\) is
an injection \(\mc{Y}_\Omega(j) \longrightarrow \mc{Y}_\Omega(i)\). Also note
that with respect to the tree (drawing) forms of monomials, each \(v
\in V(i,j)\) defines an embedding of \(\mu \in \mc{Y}_\Omega(j)\)
into $v(\mu)$; this will be important for identifying $V$-critical
ambiguities.

\begin{definition}
  A \emph{rewriting system}  \index{rewriting system} for $\mc{R}\{\Omega\}$ is a set \(S =
  \bigcup_{i \in \mathbb{N}} S(i)\) such that \(S(i) \subseteq \mc{Y}_\Omega(i)
  \times \mc{R}\{\Omega\}(i)\) for all \(i\in\mathbb{N}\). The elements of a
  rewrite system are called \emph{(rewrite) rules}. The components of
  a rule $s$ are often denoted $\mu_s$ and $a_s$, meaning \(s =
  (\mu_s,a_s)\) for all rules $s$.

  For a given rewriting system $S$, define \(T_1(S)(i) =
  \bigcup_{j \in \mathbb{N}} \{t_{v,s}\}_{v \in V(i,j), s \in S(j)}\), where
  $t_{v,s}$ is the $\mc{R}$-linear map \(\mc{R}\{\Omega\}(i) \longrightarrow
  \mc{R}\{\Omega\}(i)\) which satisfies
  \begin{equation} \label{KonEq:t_v,s-reduktion}
    t_{v,s}(\lambda) = \begin{cases}
      v(a_s)& \text{if \(\lambda = v(\mu_s)\),}\\
      \lambda& \text{otherwise,}
    \end{cases}
    \qquad\text{for all \(\lambda \in \mc{Y}_\Omega(i)\).}
  \end{equation}
  The elements of $T_1(S)(i)$ are called the \emph{simple reductions}
  (with respect to $S$) on $\mc{R}\{\Omega\}(i)$.
  For each \(i \in \mathbb{N}\), let $T(S)(i)$ be the set of all finite
  compositions of maps in $T_1(S)(i)$.

  Sometimes, a claim that \(t_{v,s}(a) = b\) is more conveniently
  written as \(a \stackrel{s}{\rightarrow} b\) (for example when
  several such claims are being chained, as in \(a
  \stackrel{s_1}{\rightarrow} b \stackrel{s_2}{\rightarrow} c\)).
  When doing that, we may indicate what $v$ is by inserting
  parentheses into the Polish term on the tail side of the arrow that
  is being changed by the simple reduction: the outer parenthesis
  then surrounds the $\mu_s \circ \nu_1 \otimes \dotsb \otimes \nu_j$
  part, whereas inner
  parentheses surround the various $\nu_k$ subterms of it, although
  these inner parentheses are for brevity omitted where \(\nu_k =
  \mathrm{id}\). See Example~\ref{Ex:AssOperad} for some examples of this.
\end{definition}

With respect to $T(S)$, all
maps in $V$ are absolutely advanceable~\cite[Def.~6.1]{GFDL}. The
following subsets of $\mc{R}\{\Omega\}$ are defined
in~\cite[Def.~3.4]{GFDL}, but so important that we include the
definitions here:
\begin{align*}
  \mathrm{Irr}(S)(i) ={}& \setOf[\big]{ a \in \mc{R}\{\Omega\}(i) }{
    \text{\(t(a) = a\) for all \(t \in T(S)(i)\)} }\text{,}\\
  \mc{I}(S)(i) ={}& \sum_{t \in T(S)(i)}
    \setOf[\big]{ a - t(a) }{ a \in \mc{R}\{\Omega\}(i) }
\end{align*}
for all \(i \in \mathbb{N}\). We write \(a \equiv b \pmod{S}\)
for \(a - b \in \mc{I}(S)\). An \(a \in \mathrm{Irr}(S)\) is said to be a
\emph{normal form} of \(b \in \mc{R}\{\Omega\}\) if
\(a \equiv b \pmod{S}\).

$\mc{I}(S)$ is the operad ideal in $\mc{R}\{\Omega\}$ that is
generated by $\setOf{ \mu_s - a_s }{ s \in S}$. $\mathrm{Irr}(S)$ is what we
want to use as model for the quotient $\mc{R}\{\Omega\} \big/
\mc{I}(S)$, and we use Theorem~\ref{S:ODL} below to tell us that it
really is. An \emph{ambiguity}~\cite[Def.~5.9]{GFDL} of $T_1(S)(i)$
is a triplet $(t_{v_1,s_1},\mu,t_{v_2,s_2})$ such that
\(v_1(\mu_{s_1}) = \mu = v_2(\mu_{s_2})\). The ambiguity is
\emph{plane} if $\mu$ is plane.

\begin{theorem}[Basic Diamond Lemma for Symmetric Operads]
  \label{S:ODL}
\index{Diamond Lemma for Symmetric Operads}

  If $P(i)$ is a well-founded partial order on $\mc{Y}_\Omega(i)$
  such that \(a_s \in \mathrm{DSM}\bigl( \mu_s, P(i) \bigr)\) for all
  \(i\in\mathbb{N}\), and moreover for all \(i,j \in \mathbb{N}\) every \(v \in
  V(i,j)\) is monotone~\cite[Def.~6.4]{GFDL} with respect to $P(j)$
  and $P(i)$, then the following claims are equivalent:
  \begin{description}
    \item[\parenthetic{a}]
      For all \(i\in\mathbb{N}\), every ambiguity of $T_1(S)(i)$ is
      resolvable~\cite[Def.~5.9]{GFDL}.
    \item[\parenthetic{a\('\)}]
      For all \(i\in\mathbb{N}\), every $V$-critical~\cite[Def.~6.8]{GFDL}
      ambiguity of $T_1(S)(i)$ is resolvable.
    \item[\parenthetic{a\('\)\('\)}]
      For all \(i\in\mathbb{N}\), every plane $V$-critical ambiguity of
      $T_1(S)(i)$ is resolvable.
    \item[\parenthetic{b}]
      Every element of $\mc{R}\{\Omega\}$ is
      persistently~\cite[Def.~4.1]{GFDL} and uniquely~\cite[Def.~4.6]{GFDL}
      reducible, with normal form map $t^S$~\cite[Def.~4.6]{GFDL}.
    \item[\parenthetic{c}]
      Every element of $\mc{R}\{\Omega\}$ has a unique normal form, i.e.,
      \(\mc{R}\{\Omega\}(i) = \mc{I}(S)(i) \oplus \mathrm{Irr}(S)(i)\) for
      all \(i\in\mathbb{N}\).
  \end{description}
\end{theorem}
\begin{proof}
  Taking \(\mc{M}(i) = \mc{R}\{\Omega\}(i)\) and \(\mc{Y}(i) =
  \mc{Y}_\Omega(i)\), this is mostly a combination of Theorem~5.11,
  Theorem~6.9, and Construction~7.2 of~\cite{GFDL}. Theorem~5.11
  provides the basic equivalence of \parenthetic{a}, \parenthetic{b},
  and \parenthetic{c}. Theorem~6.9 says \parenthetic{a\('\)} is
  sufficient, as resolvability implies resolvability relative to $P$.
  Construction~7.2 shows the $V$, $P$, and $T_1(S)$ defined above
  fulfill the conditions of these two theorems.

  What remains to show is that \parenthetic{a\('\)\('\)} implies
  \parenthetic{a\('\)}. Let $(t_{v_1,s_1},\mu,t_{v_2,s_2})$ be a
  $V$-critical ambiguity of some $T_1(S)(i)$, and let \(\sigma \in
  \Sigma_i\) be such that $\mu\sigma$ is plane. Then \(w\colon a \mapsto
  a \sigma\) and \(w^{-1}\colon a \mapsto a \sigma^{-1}\) are both
  elements of $V(i,i)$, and hence $(t_{v_1,s_1},\mu,t_{v_2,s_2})$ is
  an absolute shadow of the plane and $V$-critical ambiguity
  $(t_{w \circ v_1,s_1},\mu\sigma,t_{w \circ v_2,s_2})$. The latter is
  resolvable by \parenthetic{a\('\)\('\)}, so it follows from
  \cite[Lemma~6.2]{GFDL} that the former is resolvable as well.
\end{proof}

\begin{remark}
  Theorem~\ref{S:ODL} may also be viewed as a slightly streamlined
  version of~\cite[Cor.~10.26]{NetworkRewriting}, but that approach
  is probably overkill for readers uninterested in the PROP setting.
\end{remark}

It may be observed that $\mathrm{Irr}(S)(i)$ is closed under the action of
$\Sigma_i$, regardless of $S$; this is thus a restriction of the
applicability of this diamond lemma, as its conditions can never be
fulfilled when $\mc{R}\{\Omega\}(i) \big/ \mc{I}(S)(i)$ is fixed
under a non-identity element of $\Sigma_i$. All of that is however a
consequence of the choice of $V$, and a different choice of $V$
(e.g.~excluding the permutation $\sigma$ from \eqref{Eq:V-map}) will
result in a different (but very similar-looking) diamond lemma, with
a different set of critical ambiguities and a different domain of
applicability.

For an ambiguity $(t_{v_1,s_1},\mu,t_{v_2,s_2})$ to be $V$-critical
in this basic diamond lemma, it is necessary that the graph-theoretical
embeddings into $\mu$ of $\mu_{s_1}$ and $\mu_{s_2}$ have at least one
vertex in common (otherwise the ambiguity is a montage) and furthermore
these two embeddings must cover $\mu$ (otherwise the ambiguity is a
proper $V$-shadow). Enumerating the critical ambiguities formed by
two given rules $s_1$ and $s_2$ is thus mostly a matter of listing
the ways of superimposing the two trees $\mu_{s_1}$ and $\mu_{s_2}$.

\begin{example}[$\mathcal{A}\mkern-1mu\mathit{ss}$ operad] \label{Ex:AssOperad}
  Let \(\Omega = \bigl\{ \mathsf{m}({,}) \bigr\}\). Consider the rewriting
  system \(S = \{s\}\) where \(s = (\mathsf{m} 1\mathsf{m} 2 3, \mathsf{m}\mathsf{m} 1 2
  3)\). Graphically, this rule takes the form
  \begin{equation} \label{Eq:associativitet}
    \left[ \begin{sdpgf}{0}{0}{90}{-196}{0.2pt}
      \m 34 -123 \C -13 13 -6 20 0 18 \L 0 20 \L 0 47 \S \m 56 -123 \C
      4 4 0 26 0 15 \S \m 49 -51 \C -4 4 0 27 0 15 \S \m 71 -51 \C 4 4
      0 27 0 15 \S \m 45 -191 \L 0 41 \S \ov 44 -78 32 32 \S \ov 29
      -150 32 32 \S
    \end{sdpgf} \right]
    \rightarrow
    \left[ \begin{sdpgf}{0}{0}{90}{-196}{0.2pt}
      \m 34 -123 \C -4 4 0 26 0 15 \S \m 19 -51 \C -4 4 0 27 0 15 \S
      \m 41 -51 \C 4 4 0 27 0 15 \S \m 56 -123 \C 13 13 6 20 0 18 \L 0
      20 \L 0 47 \S \m 45 -191 \L 0 41 \S \ov 14 -78 32 32 \S \ov 29
      -150 32 32 \S
    \end{sdpgf} \right]
  \end{equation}
  The (non-unital) \emph{associative operad} $\mathcal{A}\mkern-1mu\mathit{ss}$ over $\mc{R}$
  can then be defined as the quotient $\mc{R}\{\Omega\} \big/ \mc{I}(S)$.

  One way of partially ordering trees that will be compatible with
  this rule is to count, separately for each input, the number of
  times the path from that input to the root enters an $\mathsf{m}$ vertex
  from the right; denote this number for input $i$ of the tree $\mu$
  by $h_i(\mu)$. Then define \(\mu \geqslant \nu \pin{P'(n)}\) if and
  only if \(h_i(\mu) \geqslant h_i(\nu)\) for all $i=1,\dotsc,n$, and
  define a partial order $P(n)$ by \(\mu > \nu \pin{P(n)}\) if and
  only if \(\mu \geqslant \nu \pin{P'(n)}\) and \(\mu \not\leqslant
  \nu \pin{P'(n)}\), i.e., let $P(n)$ be the restriction to a partial
  order of the quasi-order $P'(n)$. For the left hand side of $s$
  above one has \(h_1=0\), \(h_2=1\), and \(h_3=2\) whereas the left
  hand side has \(h_1=0\), \(h_2=1\), and \(h_3=1\), so $S$ is indeed
  compatible with $P$. Furthermore $P(n)$ is clearly well-founded;
  $\sum_{i=1}^n h_i(\mu)$ is simply the rank of $\mu$ in the poset
  $\bigl( \mc{Y}_\Omega(n), P(n) \bigr)$.

  The only plane critical ambiguity of $S$ is $(t_{v_1,s}, \mathsf{m} 1 \mathsf{m}
  2 \mathsf{m} 3 4, t_{v_2,s})$, where \(v_1(\mu) = \mu \circ_3 \mathsf{m} 12\)
  and \(v_2(\mu) = \mathsf{m}12 \circ_2 \mu\). This is resolved as follows:
  
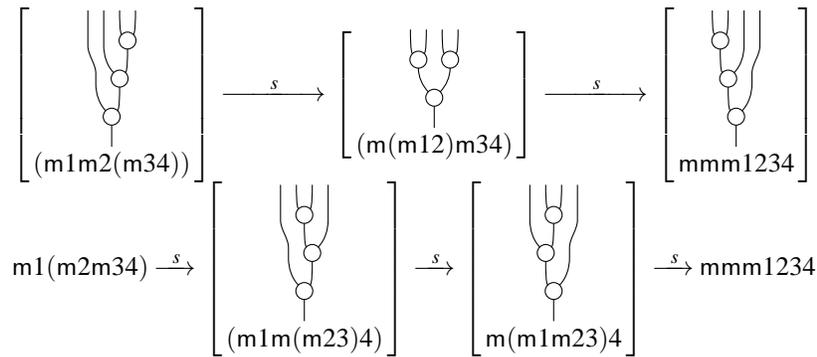
\[
    \begin{array}{c}
      \left[ \begin{array}{c}
        \begin{sdpgf}{0}{0}{120}{-268}{0.2pt}
          \m 49 -195 \C -13 13 -6 20 0 18 \L 0 20 \C 0 18 -15 16 0 18
          \L 0 20 \L 0 47 \S \m 64 -123 \C -13 13 -6 20 0 18 \L 0 20
          \L 0 47 \S \m 86 -123 \C 4 4 0 26 0 15 \S \m 79 -51 \C -4 4
          0 27 0 15 \S \m 101 -51 \C 4 4 0 27 0 15 \S \m 71 -195 \C 4
          4 0 26 0 15 \S \m 60 -263 \L 0 41 \S \ov 74 -78 32 32 \S \ov
          59 -150 32 32 \S \ov 44 -222 32 32 \S
        \end{sdpgf} \\
        (\mathsf{m} 1 \mathsf{m} 2 (\mathsf{m} 3 4))
      \end{array} \right]
      \xfmapsto{s}
      \left[ \begin{array}{c}
        \begin{sdpgf}{0}{0}{120}{-196}{0.2pt}
          \m 49 -123 \C -12 12 -7 17 0 16 \S \m 19 -51 \C -4 4 0 27
          0 15 \S \m 41 -51 \C 4 4 0 27 0 15 \S \m 71 -123 \C 12 12
          7 17 0 16 \S \m 79 -51 \C -4 4 0 27 0 15 \S \m 101 -51 \C
          4 4 0 27 0 15 \S \m 60 -191 \L 0 41 \S \ov 14 -78 32 32 \S
          \ov 74 -78 32 32 \S \ov 44 -150 32 32 \S
        \end{sdpgf} \\
        (\mathsf{m} (\mathsf{m} 1 2) \mathsf{m} 3 4)
      \end{array} \right]
      \xfmapsto{s}
      \left[ \begin{array}{c}
        \begin{sdpgf}{0}{0}{120}{-268}{0.2pt}
          \m 34 -123 \C -4 4 0 26 0 15 \S \m 19 -51 \C -4 4 0 27 0 15
          \S \m 41 -51 \C 4 4 0 27 0 15 \S \m 56 -123 \C 13 13 6 20 0
          18 \L 0 20 \L 0 47 \S \m 71 -195 \C 13 13 6 20 0 18 \L 0 20
          \C 0 18 15 16 0 18 \L 0 20 \L 0 47 \S \m 49 -195 \C -4 4 0
          26 0 15 \S \m 60 -263 \L 0 41 \S \ov 14 -78 32 32 \S \ov 29
          -150 32 32 \S \ov 44 -222 32 32 \S
        \end{sdpgf} \\
        \mathsf{m} \mathsf{m} \mathsf{m} 1 2 3 4
      \end{array} \right]
      \\
      \mathsf{m} 1 (\mathsf{m} 2 \mathsf{m} 3 4)
      \xfmapsto{s}
      \left[ \begin{array}{c}
        \begin{sdpgf}{0}{0}{120}{-268}{0.2pt}
          \m 49 -195 \C -13 13 -6 20 0 18 \L 0 20 \C 0 18 -15 16 0 18
          \L 0 20 \L 0 47 \S \m 64 -123 \C -4 4 0 26 0 15 \S \m 49 -51
          \C -4 4 0 27 0 15 \S \m 71 -51 \C 4 4 0 27 0 15 \S \m 86
          -123 \C 13 13 6 20 0 18 \L 0 20 \L 0 47 \S \m 71 -195 \C 4 4
          0 26 0 15 \S \m 60 -263 \L 0 41 \S \ov 44 -78 32 32 \S \ov
          59 -150 32 32 \S \ov 44 -222 32 32 \S
        \end{sdpgf} \\
        (\mathsf{m} 1 \mathsf{m} (\mathsf{m} 2 3) 4)
      \end{array} \right]
      \xfmapsto{s}
      \left[ \begin{array}{c}
        \begin{sdpgf}{0}{0}{120}{-268}{0.2pt}
          \m 34 -123 \C -13 13 -6 20 0 18 \L 0 20 \L 0 47 \S \m 56
          -123 \C 4 4 0 26 0 15 \S \m 49 -51 \C -4 4 0 27 0 15 \S \m
          71 -51 \C 4 4 0 27 0 15 \S \m 71 -195 \C 13 13 6 20 0 18 \L
          0 20 \C 0 18 15 16 0 18 \L 0 20 \L 0 47 \S \m 49 -195 \C -4
          4 0 26 0 15 \S \m 60 -263 \L 0 41 \S \ov 44 -78 32 32 \S \ov
          29 -150 32 32 \S \ov 44 -222 32 32 \S
        \end{sdpgf} \\
        \mathsf{m} (\mathsf{m} 1 \mathsf{m} 2 3) 4
      \end{array} \right]
      \xfmapsto{s}
      \mathsf{m} \mathsf{m} \mathsf{m} 1 2 3 4
    \end{array}
  \]
  Hence the conditions of Theorem~\ref{S:ODL} are fulfilled,
  \(\mc{R}\{\Omega\}(n) = \mc{I}(S)(n) \oplus \mathrm{Irr}(S)(n)\) for all
  \(n \in \mathbb{N}\), and \(\mathcal{A}\mkern-1mu\mathit{ss}(n) \cong \mathrm{Irr}(S)(n)\) as
  $\mc{R}$-modules for all \(n \in \mathbb{N}\). Since a monomial $\mu$ is
  irreducible iff it does not contain an $\mathsf{m}$ as right child of an
  $\mathsf{m}$, i.e., iff every right child of an $\mathsf{m}$ is an input, it
  follows that the only thing that distinguishes two irreducible
  elements of $\mc{Y}(n)$ is the order of the inputs. On the other
  hand, every permutation of the inputs gives rise to a distinct
  irreducible element, so \(\dim \mathcal{A}\mkern-1mu\mathit{ss}(n) = \lvert \Sigma_n
  \rvert = n!\) for all \(n \geqslant 1\), exactly as one would
  expect.

  For \(n=0\) one gets \(\dim \mathcal{A}\mkern-1mu\mathit{ss}(0) = \dim \mc{R}\{\Omega\}(0)
  = \bigl\lvert \mc{Y}(0) \bigr\rvert = 0\) however, which is perhaps
  not quite what the textbooks say $\mathcal{A}\mkern-1mu\mathit{ss}$ should have. The
  reason it comes out this way is that we took $\mathcal{A}\mkern-1mu\mathit{ss}$ to be the
  operad for associative algebras, period; had we instead taken it to
  be the operad for \emph{unital} associative algebras then
  \(\dim \mathcal{A}\mkern-1mu\mathit{ss}(n) = n!\) would have held also for \(n=0\).
  Obviously \(\dim \mc{R}\{\Omega\}(0) = 0\) because $\Omega$ doesn't
  contain any constants, but requiring a unit introduces such a
  constant $\mathsf{u}$. Making that constant behave like a unit requires two
  additional rules $(\mathsf{m}\mathsf{u}1,1)$ and $(\mathsf{m}1\mathsf{u},1)$ in the
  rewriting system however, and we felt the resolution of the
  resulting ambiguities are perhaps better left as exercises.
\end{example}

Another useful exercise is to similarly construct the $\mathcal{L}\mkern-1mu\mathit{eib}$
operad, which merely amounts to replacing the rewriting system $S$
with \(S' = \{s'\}\),  where \(s = (\mathsf{m} 1\mathsf{m} 2 3,
\discretionary{}{}{} \mathsf{m}\mathsf{m} 1 2 3
-\nobreak \mathsf{m}\mathsf{m} 1 3 2)\). Using brackets as notation for the
operation in a Leibniz algebra, this rule corresponds to the law that
\(\bigl[ x, [y,z] \bigr] = \bigl[ [x,y], z \bigr] - \bigl[ [x,z], y
\bigr]\). Graphically, $s'$ takes the form
\begin{equation}
  \left[ \begin{sdpgf}{0}{0}{90}{-196}{0.2pt}
    \m 34 -123 \C -13 13 -6 20 0 18 \L 0 20 \L 0 47 \S \m 56 -123 \C
    4 4 0 26 0 15 \S \m 49 -51 \C -4 4 0 27 0 15 \S \m 71 -51 \C 4 4
    0 27 0 15 \S \m 45 -191 \L 0 41 \S \ov 44 -78 32 32 \S \ov 29
    -150 32 32 \S
  \end{sdpgf} \right]
  \rightarrow
  \left[ \begin{sdpgf}{0}{0}{90}{-196}{0.2pt}
    \m 34 -123 \C -4 4 0 26 0 15 \S \m 19 -51 \C -4 4 0 27 0 15 \S
    \m 41 -51 \C 4 4 0 27 0 15 \S \m 56 -123 \C 13 13 6 20 0 18 \L 0
    20 \L 0 47 \S \m 45 -191 \L 0 41 \S \ov 14 -78 32 32 \S \ov 29
    -150 32 32 \S
  \end{sdpgf} \right]
  -
  \left[ \begin{sdpgf}{0}{0}{90}{-196}{0.2pt}
    \m 19 -51 \C -4 4 0 27 0 15 \S \m 56 -123 \C 13 13 6 20 0 18 \L 0
    20 \C 0 19 -30 9 0 19 \S \m 41 -51 \C 14 14 20 20 0 12 \S \m 34
    -123 \C -4 4 0 26 0 15 \S \m 45 -191 \L 0 41 \S \ov 14 -78 32 32
    \S \ov 29 -150 32 32 \S
  \end{sdpgf} \right]
\end{equation}
which unlike associativity is not planar, but that makes no
difference for the Diamond Lemma machinery. The left hand side of $s'$
is the same as the left hand side of $s$, so \(\mathrm{Irr}(S') = \mathrm{Irr}(S)\)
and both rewriting systems have the same sites of ambiguities. What
is different are the resolutions, where the resolution in the Leibniz
case is longer since it involves more terms; a compact notation such
as the Polish one is highly recommended when reporting the
calculations. Still, it is well within the realm of what can be done
by hand.

\subsection{The hom-associative operad}     \index{hom-associative operad}
When pursuing the same approach for the hom-associative identity, one
of course needs an extra symbol for the unary operation, so \(\Omega
= \bigl\{ \mathsf{m}({,}), \mathsf{a}() \bigr\}\). Drawing $\mathsf{m}$ as a circle and
$\mathsf{a}$ as a square, hom-associativity is then the congruence
\begin{equation} \label{Eq:rule1}
  \left[ \begin{sdpgf}{0}{0}{120}{-196}{0.2pt}
    \m 60 -191 \L 0 41 \S \m 30 -46 \L 0 41 \S \m 79 -51 \C -12 12 -7
    18 0 16 \S \m 101 -51 \C 14 14 -25 17 0 15 \S \m 49 -123 \C -12 12
    -7 17 0 16 \S \m 71 -123 \C 12 12 7 17 0 16 \S \ov 44 -150 32 32
    \S \re 14 -78 32 32 \S \ov 74 -78 32 32 \S
  \end{sdpgf} \right]
  \equiv
  \left[ \begin{sdpgf}{0}{0}{120}{-196}{0.2pt}
    \m 60 -191 \L 0 41 \S \m 19 -51 \C -14 14 25 17 0 15 \S \m 41 -51
    \C 12 12 7 18 0 16 \S \m 90 -46 \L 0 41 \S \m 49 -123 \C -12 12 -7
    17 0 16 \S \m 71 -123 \C 12 12 7 17 0 16 \S \ov 44 -150 32 32 \S
    \ov 14 -78 32 32 \S \re 74 -78 32 32 \S
  \end{sdpgf} \right]
\end{equation}
which can be expressed as a rule \(s = ( \mathsf{m} \mathsf{a} 1 \mathsf{m} 2 3, \mathsf{m}
\mathsf{m} 1 2 \mathsf{a} 3 )\). It is thus straightforward to define
\(\mathcal{H\mkern-4muA}\mkern-1mu\mathit{ss} = \mc{R}\{\Omega\} \big/ \mc{I}\bigl(\{s\}\bigr)\), but
not quite so straightforward to decide whether two elements of
$\mc{R}\{\Omega\}$ are congruent modulo $\mc{I}\bigl(\{s\}\bigr)$,
because $\{s\}$ is not a complete rewriting system; the ambiguity one
has to check fails to resolve:
\begin{equation} \label{Eq:rule2}
  \left[ \begin{sdpgf}{0}{0}{150}{-268}{0.2pt}
    \m 75 -263 \L 0 41 \S \m 34 -123 \C -13 13 -6 20 0 18 \L 0 20 \C 0
    16 15 15 0 16 \S \m 60 -46 \L 0 41 \S \m 109 -51 \C -12 12 -7 18 0
    16 \S \m 131 -51 \C 14 14 -25 17 0 15 \S \m 64 -195 \C -12 12 -7
    17 0 16 \S \m 86 -195 \C 12 12 7 17 0 16 \S \m 56 -123 \C 4 4 0 26
    0 15 \S \m 105 -118 \C 0 14 15 12 0 14 \S \ov 59 -222 32 32 \S \ov
    29 -150 32 32 \S \re 44 -78 32 32 \S \ov 104 -78 32 32 \S \re 89
    -150 32 32 \S
  \end{sdpgf} \right]
  \stackrel{s}{\leftarrow}
  \left[ \begin{sdpgf}{0}{0}{150}{-268}{0.2pt}
    \m 75 -263 \L 0 41 \S \m 45 -118 \C 0 19 -30 8 0 19 \L 0 20 \C 0
    16 15 15 0 16 \S \m 60 -46 \L 0 41 \S \m 109 -51 \C -12 12 -7 18 0
    16 \S \m 131 -51 \C 14 14 -25 17 0 15 \S \m 64 -195 \C -12 12 -7
    17 0 16 \S \m 86 -195 \C 12 12 7 17 0 16 \S \m 94 -123 \C -14 14
    -20 20 0 11 \S \m 116 -123 \C 4 4 0 26 0 15 \S \ov 59 -222 32 32
    \S \re 29 -150 32 32 \S \re 44 -78 32 32 \S \ov 104 -78 32 32 \S
    \ov 89 -150 32 32 \S
  \end{sdpgf} \right]
  \stackrel{s}{\rightarrow}
  \left[ \begin{sdpgf}{0}{0}{150}{-268}{0.2pt}
    \m 75 -263 \L 0 41 \S \m 45 -118 \C 0 19 -30 8 0 19 \L 0 20 \C 0
    16 15 15 0 16 \S \m 49 -51 \C -14 14 25 17 0 15 \S \m 71 -51 \C 12
    12 7 18 0 16 \S \m 120 -46 \L 0 41 \S \m 64 -195 \C -12 12 -7 17 0
    16 \S \m 86 -195 \C 12 12 7 17 0 16 \S \m 94 -123 \C -14 14 -20 20
    0 11 \S \m 116 -123 \C 4 4 0 26 0 15 \S \ov 59 -222 32 32 \S \re
    29 -150 32 32 \S \ov 44 -78 32 32 \S \re 104 -78 32 32 \S \ov 89
    -150 32 32 \S
  \end{sdpgf} \right]
  \stackrel{s}{\rightarrow}
  \left[ \begin{sdpgf}{0}{0}{150}{-268}{0.2pt}
    \m 75 -263 \L 0 41 \S \m 34 -123 \C -13 13 -6 20 0 18 \L 0 20 \C 0
    16 15 15 0 16 \S \m 49 -51 \C -14 14 25 17 0 15 \S \m 71 -51 \C 12
    12 7 18 0 16 \S \m 120 -46 \L 0 41 \S \m 64 -195 \C -12 12 -7 17 0
    16 \S \m 86 -195 \C 12 12 7 17 0 16 \S \m 56 -123 \C 4 4 0 26 0 15
    \S \m 105 -118 \C 0 14 15 12 0 14 \S \ov 59 -222 32 32 \S \ov 29
    -150 32 32 \S \ov 44 -78 32 32 \S \re 104 -78 32 32 \S \re 89 -150
    32 32 \S
  \end{sdpgf} \right]
  \text{.}
\end{equation}
Failed resolutions should not be taken as disasters however; they are
in fact opportunities to learn, since what the above demonstrates is
that \(\mathsf{m} \mathsf{m} 1 \mathsf{a} 2 \mathsf{a} \mathsf{m} 3 4 \equiv \mathsf{m} \mathsf{m} 1 \mathsf{m} 2 3
\mathsf{a} \mathsf{a} 4 \pmod{\{s\}}\) (or as a law: \(\bigl( x \alpha(y) \bigr)
\alpha(zw) = \bigl( x (y z) \bigr) \alpha\bigl( \alpha(w) \bigr)\)
for all $x,y,z,w$), which was probably not apparent from the definition
of hom-associativity. Therefore one's response to this discovery
should be to make a new rule out of this new and nontrivial
congruence, so that one can use it to better understand congruence
modulo hom-associativity.

A problem with this congruence is however that the left and right
hand sides are not comparable under the same partial order as worked
fine for the associative and Leibniz operads: \((h_0,h_1,h_2,h_3) =
(0,1,1,2)\) for the left hand side but \((h_0,h_1,h_2,h_3) =
(0,1,2,1)\) for the right hand side; finding a compatible order can
be a rather challenging problem for complex congruences. In the case
of hom-associativity though, the fact that all inputs are at the same
height in the left and right hand sides makes it possible to use
something very classical: a lexicographic order. Recursively it may
be defined as having \(\mu > \nu\) if:
\begin{itemize}
  \item
    \(\mu = \mathsf{m} \mu' \mu''\) and \(\nu = \mathsf{m} \nu' \nu''\), where
    \(\mu' > \nu'\), or
  \item
    \(\mu = \mathsf{m} \mu' \mu''\) and \(\nu = \mathsf{m} \nu' \nu''\), where
    \(\mu'=\nu'\) and \(\mu''>\nu''\), or
  \item
    \(\mu = \mathsf{a} \mu'\) and \(\nu = \mathsf{m} \nu' \nu''\), or
  \item
    \(\mu = \mathsf{a} \mu'\) and \(\nu = \mathsf{a} \nu'\), where \(\mu' >
    \nu'\).
\end{itemize}
Equivalently, one may define it as the word-lexicographic order on
the Polish notation, over the order on letters which has \(\mathsf{m} <
\mathsf{a}\) and each $\Box_i$ unrelated to all other letters. With this
order, it is clear that the congruence \eqref{Eq:rule2} should be
turned into the rule \((\mathsf{m} \mathsf{m} 1 \mathsf{a} 2 \mathsf{a} \mathsf{m} 3 4,
\mathsf{m} \mathsf{m} 1 \mathsf{m} 2 3 \mathsf{a} \mathsf{a} 4)\).

In general, the idea to ``find all ambiguities, try to resolve them,
make new rules out of everything that doesn't resolve, and repeat
until everything resolves'' is called the \emph{Critical Pairs\slash
Completion}~(CPC) procedure; its most famous instance is the Buchberger
algorithm for computing Gr\"obner bases. `Critical pairs' corresponds
to identifying ambiguities, whereas `completion' is the step of adding
new rules; a rewriting system is said to be \emph{complete} when all
ambiguities are resolvable.

In the case at hand, the calculations quickly become extensive, so we
make use of a program~\cite{cmplutil} one of us has written that
automates the CPC procedure in the operadic setting (actually, in the
more general PROP setting). Running it with \eqref{Eq:rule1} as input
quickly leads to the discovery of \eqref{Eq:rule2} and several more
identities:
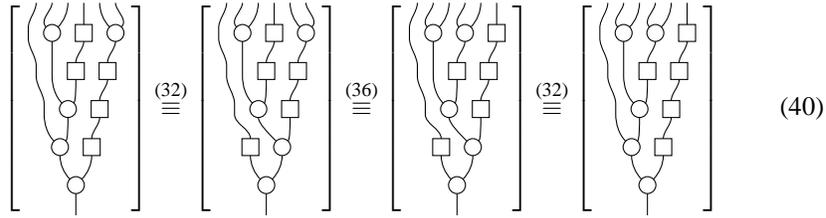
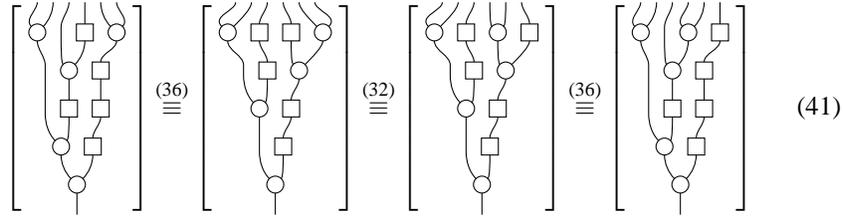
\begin{gather} \label{Eq:rule3}
  \left[ \begin{sdpgf}{0}{0}{180}{-340}{0.2pt}
    \m 90 -335 \L 0 41 \S \m 49 -195 \C -13 13 -6 20 0 18 \L 0 20 \C 0
    18 -15 16 0 18 \L 0 20 \C 0 16 15 15 0 16 \S \m 64 -123 \C -13 13
    -6 20 0 18 \L 0 20 \C 0 16 15 15 0 16 \S \m 90 -46 \L 0 41 \S \m
    139 -51 \C -12 12 -7 18 0 16 \S \m 161 -51 \C 14 14 -25 17 0 15 \S
    \m 79 -267 \C -12 12 -7 17 0 16 \S \m 101 -267 \C 12 12 7 17 0 16
    \S \m 71 -195 \C 4 4 0 26 0 15 \S \m 86 -123 \C 4 4 0 26 0 15 \S
    \m 135 -118 \C 0 14 15 12 0 14 \S \m 120 -190 \C 0 14 15 12 0 14
    \S \ov 74 -294 32 32 \S \ov 44 -222 32 32 \S \ov 59 -150 32 32 \S
    \re 74 -78 32 32 \S \ov 134 -78 32 32 \S \re 104 -222 32 32 \S \re
    119 -150 32 32 \S
  \end{sdpgf} \right]
  \stackrel{\eqref{Eq:rule1}}{\equiv}
  \left[ \begin{sdpgf}{0}{0}{180}{-340}{0.2pt}
    \m 90 -335 \L 0 41 \S \m 60 -190 \C 0 19 -30 8 0 19 \L 0 20 \C 0
    18 -15 16 0 18 \L 0 20 \C 0 16 15 15 0 16 \S \m 64 -123 \C -13 13
    -6 20 0 18 \L 0 20 \C 0 16 15 15 0 16 \S \m 90 -46 \L 0 41 \S \m
    139 -51 \C -12 12 -7 18 0 16 \S \m 161 -51 \C 14 14 -25 17 0 15 \S
    \m 79 -267 \C -12 12 -7 17 0 16 \S \m 101 -267 \C 12 12 7 17 0 16
    \S \m 109 -195 \C -14 14 -20 20 0 11 \S \m 86 -123 \C 4 4 0 26 0
    15 \S \m 135 -118 \C 0 14 15 12 0 14 \S \m 131 -195 \C 4 4 0 26 0
    15 \S \ov 74 -294 32 32 \S \re 44 -222 32 32 \S \ov 59 -150 32 32
    \S \re 74 -78 32 32 \S \ov 134 -78 32 32 \S \ov 104 -222 32 32 \S
    \re 119 -150 32 32 \S
  \end{sdpgf} \right]
  \stackrel{\eqref{Eq:rule2}}{\equiv}
  \left[ \begin{sdpgf}{0}{0}{180}{-340}{0.2pt}
    \m 90 -335 \L 0 41 \S \m 60 -190 \C 0 19 -30 8 0 19 \L 0 20 \C 0
    18 -15 16 0 18 \L 0 20 \C 0 16 15 15 0 16 \S \m 64 -123 \C -13 13
    -6 20 0 18 \L 0 20 \C 0 16 15 15 0 16 \S \m 79 -51 \C -14 14 25 17
    0 15 \S \m 101 -51 \C 12 12 7 18 0 16 \S \m 150 -46 \L 0 41 \S \m
    79 -267 \C -12 12 -7 17 0 16 \S \m 101 -267 \C 12 12 7 17 0 16 \S
    \m 109 -195 \C -14 14 -20 20 0 11 \S \m 86 -123 \C 4 4 0 26 0 15
    \S \m 135 -118 \C 0 14 15 12 0 14 \S \m 131 -195 \C 4 4 0 26 0 15
    \S \ov 74 -294 32 32 \S \re 44 -222 32 32 \S \ov 59 -150 32 32 \S
    \ov 74 -78 32 32 \S \re 134 -78 32 32 \S \ov 104 -222 32 32 \S \re
    119 -150 32 32 \S
  \end{sdpgf} \right]
  \stackrel{\eqref{Eq:rule1}}{\equiv}
  \left[ \begin{sdpgf}{0}{0}{180}{-340}{0.2pt}
    \m 90 -335 \L 0 41 \S \m 49 -195 \C -13 13 -6 20 0 18 \L 0 20 \C 0
    18 -15 16 0 18 \L 0 20 \C 0 16 15 15 0 16 \S \m 64 -123 \C -13 13
    -6 20 0 18 \L 0 20 \C 0 16 15 15 0 16 \S \m 79 -51 \C -14 14 25 17
    0 15 \S \m 101 -51 \C 12 12 7 18 0 16 \S \m 150 -46 \L 0 41 \S \m
    79 -267 \C -12 12 -7 17 0 16 \S \m 101 -267 \C 12 12 7 17 0 16 \S
    \m 71 -195 \C 4 4 0 26 0 15 \S \m 86 -123 \C 4 4 0 26 0 15 \S \m
    135 -118 \C 0 14 15 12 0 14 \S \m 120 -190 \C 0 14 15 12 0 14 \S
    \ov 74 -294 32 32 \S \ov 44 -222 32 32 \S \ov 59 -150 32 32 \S \ov
    74 -78 32 32 \S \re 134 -78 32 32 \S \re 104 -222 32 32 \S \re 119
    -150 32 32 \S
  \end{sdpgf} \right]
  \displaybreak[0]\\
  \label{Eq:rule4}
  \left[ \begin{sdpgf}{0}{0}{210}{-340}{0.2pt}
    \m 105 -335 \L 0 41 \S \m 34 -123 \C -13 13 -6 20 0 18 \L 0 20 \C
    0 16 15 15 0 16 \S \m 60 -46 \L 0 41 \S \m 109 -51 \C -12 12 -7 18
    0 16 \S \m 131 -51 \C 14 14 -25 17 0 15 \S \m 116 -123 \C 17 17 32
    32 0 2 \L 0 20 \C 0 16 -15 15 0 16 \S \m 165 -118 \C 0 19 30 8 0
    19 \L 0 20 \C 0 16 -15 15 0 16 \S \m 94 -267 \C -12 12 -7 17 0 16
    \S \m 116 -267 \C 12 12 7 17 0 16 \S \m 64 -195 \C -12 12 -7 17 0
    16 \S \m 56 -123 \C 4 4 0 26 0 15 \S \m 94 -123 \C -20 20 46 10 0
    15 \S \m 86 -195 \C 12 12 7 17 0 16 \S \m 135 -190 \C 0 17 30 6 0
    17 \S \ov 89 -294 32 32 \S \ov 29 -150 32 32 \S \re 44 -78 32 32
    \S \ov 104 -78 32 32 \S \ov 89 -150 32 32 \S \re 149 -150 32 32 \S
    \ov 59 -222 32 32 \S \re 119 -222 32 32 \S
  \end{sdpgf} \right]
  \stackrel{\eqref{Eq:rule2}}{\equiv}
  \left[ \begin{sdpgf}{0}{0}{210}{-340}{0.2pt}
    \m 105 -335 \L 0 41 \S \m 49 -123 \C -15 15 -19 20 0 16 \L 0 20 \C
    0 16 15 15 0 16 \S \m 60 -46 \L 0 41 \S \m 109 -51 \C -12 12 -7 18
    0 16 \S \m 131 -51 \C 14 14 -25 17 0 15 \S \m 169 -51 \C -12 12 -7
    18 0 16 \S \m 191 -51 \C 14 14 -25 17 0 15 \S \m 94 -267 \C -12 12
    -7 17 0 16 \S \m 116 -267 \C 12 12 7 17 0 16 \S \m 64 -195 \C -4 4
    0 26 0 15 \S \m 71 -123 \C 14 14 -25 16 0 15 \S \m 120 -118 \L 0
    40 \S \m 135 -190 \C 0 19 30 8 0 19 \L 0 20 \C 0 16 15 14 0 16 \S
    \m 86 -195 \C 14 14 20 20 0 11 \S \ov 89 -294 32 32 \S \ov 44 -150
    32 32 \S \re 44 -78 32 32 \S \ov 104 -78 32 32 \S \ov 164 -78 32
    32 \S \ov 59 -222 32 32 \S \re 119 -222 32 32 \S \re 104 -150 32
    32 \S
  \end{sdpgf} \right]
  \stackrel{\eqref{Eq:rule2}}{\equiv}
  \left[ \begin{sdpgf}{0}{0}{210}{-340}{0.2pt}
    \m 105 -335 \L 0 41 \S \m 49 -123 \C -15 15 -19 20 0 16 \L 0 20 \C
    0 16 15 15 0 16 \S \m 49 -51 \C -14 14 25 17 0 15 \S \m 71 -51 \C
    12 12 7 18 0 16 \S \m 120 -46 \L 0 41 \S \m 169 -51 \C -12 12 -7
    18 0 16 \S \m 191 -51 \C 14 14 -25 17 0 15 \S \m 94 -267 \C -12 12
    -7 17 0 16 \S \m 116 -267 \C 12 12 7 17 0 16 \S \m 64 -195 \C -4 4
    0 26 0 15 \S \m 71 -123 \C 14 14 -25 16 0 15 \S \m 120 -118 \L 0
    40 \S \m 135 -190 \C 0 19 30 8 0 19 \L 0 20 \C 0 16 15 14 0 16 \S
    \m 86 -195 \C 14 14 20 20 0 11 \S \ov 89 -294 32 32 \S \ov 44 -150
    32 32 \S \ov 44 -78 32 32 \S \re 104 -78 32 32 \S \ov 164 -78 32
    32 \S \ov 59 -222 32 32 \S \re 119 -222 32 32 \S \re 104 -150 32
    32 \S
  \end{sdpgf} \right]
  \stackrel{\eqref{Eq:rule2}}{\equiv}
  \left[ \begin{sdpgf}{0}{0}{210}{-340}{0.2pt}
    \m 105 -335 \L 0 41 \S \m 34 -123 \C -13 13 -6 20 0 18 \L 0 20 \C
    0 16 15 15 0 16 \S \m 49 -51 \C -14 14 25 17 0 15 \S \m 71 -51 \C
    12 12 7 18 0 16 \S \m 120 -46 \L 0 41 \S \m 116 -123 \C 17 17 32
    32 0 2 \L 0 20 \C 0 16 -15 15 0 16 \S \m 165 -118 \C 0 19 30 8 0
    19 \L 0 20 \C 0 16 -15 15 0 16 \S \m 94 -267 \C -12 12 -7 17 0 16
    \S \m 116 -267 \C 12 12 7 17 0 16 \S \m 64 -195 \C -12 12 -7 17 0
    16 \S \m 56 -123 \C 4 4 0 26 0 15 \S \m 94 -123 \C -20 20 46 10 0
    15 \S \m 86 -195 \C 12 12 7 17 0 16 \S \m 135 -190 \C 0 17 30 6 0
    17 \S \ov 89 -294 32 32 \S \ov 29 -150 32 32 \S \ov 44 -78 32 32
    \S \re 104 -78 32 32 \S \ov 89 -150 32 32 \S \re 149 -150 32 32 \S
    \ov 59 -222 32 32 \S \re 119 -222 32 32 \S
  \end{sdpgf} \right]
  \displaybreak[0]\\
  \label{Eq:rule5}
  \left[ \begin{sdpgf}{0}{0}{180}{-340}{0.2pt}
    \m 90 -335 \L 0 41 \S \m 19 -51 \C -14 14 25 17 0 15 \S \m 41 -51
    \C 12 12 7 18 0 16 \S \m 90 -46 \L 0 41 \S \m 139 -51 \C -12 12 -7
    18 0 16 \S \m 161 -51 \C 14 14 -25 17 0 15 \S \m 79 -267 \C -12 12
    -7 17 0 16 \S \m 101 -267 \C 12 12 7 17 0 16 \S \m 49 -195 \C -13
    13 -6 20 0 18 \L 0 20 \L 0 46 \S \m 75 -118 \C 0 14 15 12 0 14 \S
    \m 135 -118 \C 0 14 15 12 0 14 \S \m 71 -195 \C 4 4 0 26 0 15 \S
    \m 120 -190 \C 0 14 15 12 0 14 \S \ov 74 -294 32 32 \S \ov 14 -78
    32 32 \S \re 74 -78 32 32 \S \ov 134 -78 32 32 \S \ov 44 -222 32
    32 \S \re 104 -222 32 32 \S \re 59 -150 32 32 \S \re 119 -150 32
    32 \S
  \end{sdpgf} \right]
  \stackrel{\eqref{Eq:rule1}}{\equiv}
  \left[ \begin{sdpgf}{0}{0}{180}{-340}{0.2pt}
    \m 90 -335 \L 0 41 \S \m 30 -118 \C 0 16 -15 14 0 16 \L 0 20 \C 0
    16 15 15 0 16 \S \m 79 -123 \C -15 15 -19 20 0 16 \L 0 20 \C 0 16
    15 15 0 16 \S \m 90 -46 \L 0 41 \S \m 139 -51 \C -12 12 -7 18 0 16
    \S \m 161 -51 \C 14 14 -25 17 0 15 \S \m 79 -267 \C -12 12 -7 17 0
    16 \S \m 101 -267 \C 12 12 7 17 0 16 \S \m 49 -195 \C -12 12 -7 17
    0 16 \S \m 71 -195 \C 12 12 7 17 0 16 \S \m 101 -123 \C 14 14 -25
    16 0 15 \S \m 150 -118 \L 0 40 \S \m 120 -190 \C 0 17 30 6 0 17 \S
    \ov 74 -294 32 32 \S \re 14 -150 32 32 \S \ov 74 -150 32 32 \S \re
    74 -78 32 32 \S \ov 134 -78 32 32 \S \ov 44 -222 32 32 \S \re 104
    -222 32 32 \S \re 134 -150 32 32 \S
  \end{sdpgf} \right]
  \stackrel{\eqref{Eq:rule3}}{\equiv}
  \left[ \begin{sdpgf}{0}{0}{180}{-340}{0.2pt}
    \m 90 -335 \L 0 41 \S \m 30 -118 \C 0 16 -15 14 0 16 \L 0 20 \C 0
    16 15 15 0 16 \S \m 79 -123 \C -15 15 -19 20 0 16 \L 0 20 \C 0 16
    15 15 0 16 \S \m 79 -51 \C -14 14 25 17 0 15 \S \m 101 -51 \C 12
    12 7 18 0 16 \S \m 150 -46 \L 0 41 \S \m 79 -267 \C -12 12 -7 17 0
    16 \S \m 101 -267 \C 12 12 7 17 0 16 \S \m 49 -195 \C -12 12 -7 17
    0 16 \S \m 71 -195 \C 12 12 7 17 0 16 \S \m 101 -123 \C 14 14 -25
    16 0 15 \S \m 150 -118 \L 0 40 \S \m 120 -190 \C 0 17 30 6 0 17 \S
    \ov 74 -294 32 32 \S \re 14 -150 32 32 \S \ov 74 -150 32 32 \S \ov
    74 -78 32 32 \S \re 134 -78 32 32 \S \ov 44 -222 32 32 \S \re 104
    -222 32 32 \S \re 134 -150 32 32 \S
  \end{sdpgf} \right]
  \stackrel{\eqref{Eq:rule1}}{\equiv}
  \left[ \begin{sdpgf}{0}{0}{180}{-340}{0.2pt}
    \m 90 -335 \L 0 41 \S \m 19 -51 \C -14 14 25 17 0 15 \S \m 41 -51
    \C 12 12 7 18 0 16 \S \m 79 -51 \C -14 14 25 17 0 15 \S \m 101 -51
    \C 12 12 7 18 0 16 \S \m 150 -46 \L 0 41 \S \m 79 -267 \C -12 12
    -7 17 0 16 \S \m 101 -267 \C 12 12 7 17 0 16 \S \m 49 -195 \C -13
    13 -6 20 0 18 \L 0 20 \L 0 46 \S \m 75 -118 \C 0 14 15 12 0 14 \S
    \m 135 -118 \C 0 14 15 12 0 14 \S \m 71 -195 \C 4 4 0 26 0 15 \S
    \m 120 -190 \C 0 14 15 12 0 14 \S \ov 74 -294 32 32 \S \ov 14 -78
    32 32 \S \ov 74 -78 32 32 \S \re 134 -78 32 32 \S \ov 44 -222 32
    32 \S \re 104 -222 32 32 \S \re 59 -150 32 32 \S \re 119 -150 32
    32 \S
  \end{sdpgf} \right]
  \displaybreak[0]\\
  \label{Eq:rule6}
  \left[ \begin{sdpgf}{0}{0}{210}{-412}{0.2pt}
    \m 105 -407 \L 0 41 \S \m 64 -267 \C -13 13 -6 20 0 18 \L 0 20 \C
    0 18 -15 16 0 18 \L 0 20 \C 0 18 -15 16 0 18 \L 0 20 \C 0 16 15 15
    0 16 \S \m 79 -195 \C -13 13 -6 20 0 18 \L 0 20 \C 0 18 -15 16 0
    18 \L 0 20 \C 0 16 15 15 0 16 \S \m 94 -123 \C -13 13 -6 20 0 18
    \L 0 20 \C 0 16 15 15 0 16 \S \m 120 -46 \L 0 41 \S \m 169 -51 \C
    -12 12 -7 18 0 16 \S \m 191 -51 \C 14 14 -25 17 0 15 \S \m 94 -339
    \C -12 12 -7 17 0 16 \S \m 116 -339 \C 12 12 7 17 0 16 \S \m 86
    -267 \C 4 4 0 26 0 15 \S \m 101 -195 \C 4 4 0 26 0 15 \S \m 116
    -123 \C 4 4 0 26 0 15 \S \m 165 -118 \C 0 14 15 12 0 14 \S \m 135
    -262 \C 0 14 15 12 0 14 \S \m 150 -190 \C 0 14 15 12 0 14 \S \ov
    89 -366 32 32 \S \ov 59 -294 32 32 \S \ov 74 -222 32 32 \S \ov 89
    -150 32 32 \S \re 104 -78 32 32 \S \ov 164 -78 32 32 \S \re 119
    -294 32 32 \S \re 149 -150 32 32 \S \re 134 -222 32 32 \S
  \end{sdpgf} \right]
  \stackrel{\eqref{Eq:rule2}}{\equiv}
  \left[ \begin{sdpgf}{0}{0}{240}{-412}{0.2pt}
    \m 120 -407 \L 0 41 \S \m 49 -123 \C -15 15 -19 20 0 16 \L 0 20 \C
    0 19 30 9 0 19 \S \m 60 -46 \C 0 14 15 13 0 14 \S \m 109 -123 \C
    -4 4 0 30 0 17 \L 0 20 \L 0 47 \S \m 150 -46 \C 0 14 -15 13 0 14
    \S \m 199 -51 \C -14 14 -20 20 0 12 \S \m 221 -51 \C 20 20 -46 10
    0 16 \S \m 109 -339 \C -13 13 -6 20 0 18 \L 0 20 \L 0 52 \L 0 20
    \C 0 19 -30 8 0 19 \S \m 131 -339 \C 4 4 0 26 0 15 \S \m 71 -123
    \C 14 14 -25 16 0 15 \S \m 124 -195 \C -4 4 0 26 0 15 \S \m 131
    -123 \C 12 12 7 17 0 16 \S \m 180 -118 \C 0 17 30 6 0 17 \S \m 135
    -262 \L 0 40 \S \m 146 -195 \C 14 14 20 20 0 11 \S \ov 104 -366 32
    32 \S \ov 44 -150 32 32 \S \re 44 -78 32 32 \S \ov 104 -150 32 32
    \S \re 134 -78 32 32 \S \ov 194 -78 32 32 \S \re 119 -294 32 32 \S
    \ov 119 -222 32 32 \S \re 164 -150 32 32 \S
  \end{sdpgf} \right]
  \stackrel{\eqref{Eq:rule2}}{\equiv}
  \left[ \begin{sdpgf}{0}{0}{240}{-412}{0.2pt}
    \m 120 -407 \L 0 41 \S \m 49 -123 \C -15 15 -19 20 0 16 \L 0 20 \C
    0 19 30 9 0 19 \S \m 60 -46 \C 0 14 15 13 0 14 \S \m 109 -123 \C
    -4 4 0 30 0 17 \L 0 20 \L 0 47 \S \m 139 -51 \C -4 4 0 27 0 15 \S
    \m 161 -51 \C 4 4 0 27 0 15 \S \m 210 -46 \C 0 14 -15 13 0 14 \S
    \m 109 -339 \C -13 13 -6 20 0 18 \L 0 20 \L 0 52 \L 0 20 \C 0 19
    -30 8 0 19 \S \m 131 -339 \C 4 4 0 26 0 15 \S \m 71 -123 \C 14 14
    -25 16 0 15 \S \m 124 -195 \C -4 4 0 26 0 15 \S \m 131 -123 \C 12
    12 7 17 0 16 \S \m 180 -118 \C 0 17 30 6 0 17 \S \m 135 -262 \L 0
    40 \S \m 146 -195 \C 14 14 20 20 0 11 \S \ov 104 -366 32 32 \S \ov
    44 -150 32 32 \S \re 44 -78 32 32 \S \ov 104 -150 32 32 \S \ov 134
    -78 32 32 \S \re 194 -78 32 32 \S \re 119 -294 32 32 \S \ov 119
    -222 32 32 \S \re 164 -150 32 32 \S
  \end{sdpgf} \right]
  \stackrel{\eqref{Eq:rule2}}{\equiv}
  \left[ \begin{sdpgf}{0}{0}{210}{-412}{0.2pt}
    \m 105 -407 \L 0 41 \S \m 64 -267 \C -13 13 -6 20 0 18 \L 0 20 \C
    0 18 -15 16 0 18 \L 0 20 \C 0 18 -15 16 0 18 \L 0 20 \C 0 16 15 15
    0 16 \S \m 79 -195 \C -13 13 -6 20 0 18 \L 0 20 \C 0 18 -15 16 0
    18 \L 0 20 \C 0 16 15 15 0 16 \S \m 94 -123 \C -13 13 -6 20 0 18
    \L 0 20 \C 0 16 15 15 0 16 \S \m 109 -51 \C -14 14 25 17 0 15 \S
    \m 131 -51 \C 12 12 7 18 0 16 \S \m 180 -46 \L 0 41 \S \m 94 -339
    \C -12 12 -7 17 0 16 \S \m 116 -339 \C 12 12 7 17 0 16 \S \m 86
    -267 \C 4 4 0 26 0 15 \S \m 101 -195 \C 4 4 0 26 0 15 \S \m 116
    -123 \C 4 4 0 26 0 15 \S \m 165 -118 \C 0 14 15 12 0 14 \S \m 135
    -262 \C 0 14 15 12 0 14 \S \m 150 -190 \C 0 14 15 12 0 14 \S \ov
    89 -366 32 32 \S \ov 59 -294 32 32 \S \ov 74 -222 32 32 \S \ov 89
    -150 32 32 \S \ov 104 -78 32 32 \S \re 164 -78 32 32 \S \re 119
    -294 32 32 \S \re 149 -150 32 32 \S \re 134 -222 32 32 \S
  \end{sdpgf} \right]
  \displaybreak[0]\\
  \label{Eq:rule7}
  \left[ \begin{sdpgf}{0}{0}{240}{-412}{0.2pt}
    \m 120 -407 \L 0 41 \S \m 49 -195 \C -13 13 -6 20 0 18 \L 0 20 \C
    0 18 -15 16 0 18 \L 0 20 \C 0 16 15 15 0 16 \S \m 64 -123 \C -13
    13 -6 20 0 18 \L 0 20 \C 0 16 15 15 0 16 \S \m 90 -46 \L 0 41 \S
    \m 139 -51 \C -12 12 -7 18 0 16 \S \m 161 -51 \C 14 14 -25 17 0 15
    \S \m 131 -195 \C 17 17 32 32 0 2 \L 0 20 \C 0 18 15 16 0 18 \L 0
    20 \C 0 16 -15 15 0 16 \S \m 180 -190 \C 0 19 30 8 0 19 \L 0 20 \C
    0 18 15 16 0 18 \L 0 20 \C 0 16 -15 15 0 16 \S \m 109 -339 \C -12
    12 -7 17 0 16 \S \m 131 -339 \C 12 12 7 17 0 16 \S \m 79 -267 \C
    -12 12 -7 17 0 16 \S \m 71 -195 \C 4 4 0 26 0 15 \S \m 86 -123 \C
    4 4 0 26 0 15 \S \m 135 -118 \C 0 14 15 12 0 14 \S \m 101 -267 \C
    12 12 7 17 0 16 \S \m 109 -195 \C -20 20 46 10 0 15 \S \m 150 -262
    \C 0 17 30 6 0 17 \S \ov 104 -366 32 32 \S \ov 44 -222 32 32 \S
    \ov 59 -150 32 32 \S \re 74 -78 32 32 \S \ov 134 -78 32 32 \S \ov
    104 -222 32 32 \S \re 164 -222 32 32 \S \ov 74 -294 32 32 \S \re
    134 -294 32 32 \S \re 119 -150 32 32 \S
  \end{sdpgf} \right]
  \stackrel{\eqref{Eq:rule2}}{\equiv}
  \left[ \begin{sdpgf}{0}{0}{240}{-412}{0.2pt}
    \m 120 -407 \L 0 41 \S \m 64 -195 \C -13 13 -6 20 0 18 \L 0 20 \C
    0 20 -30 12 0 20 \L 0 20 \C 0 16 15 15 0 16 \S \m 79 -123 \C -15
    15 -19 20 0 16 \L 0 20 \C 0 16 15 15 0 16 \S \m 90 -46 \L 0 41 \S
    \m 139 -51 \C -12 12 -7 18 0 16 \S \m 161 -51 \C 14 14 -25 17 0 15
    \S \m 199 -51 \C -12 12 -7 18 0 16 \S \m 221 -51 \C 14 14 -25 17 0
    15 \S \m 109 -339 \C -12 12 -7 17 0 16 \S \m 131 -339 \C 12 12 7
    17 0 16 \S \m 79 -267 \C -4 4 0 26 0 15 \S \m 86 -195 \C 4 4 0 26
    0 15 \S \m 101 -123 \C 14 14 -25 16 0 15 \S \m 150 -118 \L 0 40 \S
    \m 150 -262 \C 0 19 30 8 0 19 \L 0 20 \C 0 18 15 16 0 18 \L 0 20
    \C 0 16 15 14 0 16 \S \m 101 -267 \C 14 14 20 20 0 11 \S \m 135
    -190 \C 0 14 15 12 0 14 \S \ov 104 -366 32 32 \S \ov 59 -222 32 32
    \S \ov 74 -150 32 32 \S \re 74 -78 32 32 \S \ov 134 -78 32 32 \S
    \ov 194 -78 32 32 \S \ov 74 -294 32 32 \S \re 134 -294 32 32 \S
    \re 134 -150 32 32 \S \re 119 -222 32 32 \S
  \end{sdpgf} \right]
  \stackrel{\eqref{Eq:rule3}}{\equiv}
  \left[ \begin{sdpgf}{0}{0}{240}{-412}{0.2pt}
    \m 120 -407 \L 0 41 \S \m 64 -195 \C -13 13 -6 20 0 18 \L 0 20 \C
    0 20 -30 12 0 20 \L 0 20 \C 0 16 15 15 0 16 \S \m 79 -123 \C -15
    15 -19 20 0 16 \L 0 20 \C 0 16 15 15 0 16 \S \m 79 -51 \C -14 14
    25 17 0 15 \S \m 101 -51 \C 12 12 7 18 0 16 \S \m 150 -46 \L 0 41
    \S \m 199 -51 \C -12 12 -7 18 0 16 \S \m 221 -51 \C 14 14 -25 17 0
    15 \S \m 109 -339 \C -12 12 -7 17 0 16 \S \m 131 -339 \C 12 12 7
    17 0 16 \S \m 79 -267 \C -4 4 0 26 0 15 \S \m 86 -195 \C 4 4 0 26
    0 15 \S \m 101 -123 \C 14 14 -25 16 0 15 \S \m 150 -118 \L 0 40 \S
    \m 150 -262 \C 0 19 30 8 0 19 \L 0 20 \C 0 18 15 16 0 18 \L 0 20
    \C 0 16 15 14 0 16 \S \m 101 -267 \C 14 14 20 20 0 11 \S \m 135
    -190 \C 0 14 15 12 0 14 \S \ov 104 -366 32 32 \S \ov 59 -222 32 32
    \S \ov 74 -150 32 32 \S \ov 74 -78 32 32 \S \re 134 -78 32 32 \S
    \ov 194 -78 32 32 \S \ov 74 -294 32 32 \S \re 134 -294 32 32 \S
    \re 134 -150 32 32 \S \re 119 -222 32 32 \S
  \end{sdpgf} \right]
  \stackrel{\eqref{Eq:rule2}}{\equiv}
  \left[ \begin{sdpgf}{0}{0}{240}{-412}{0.2pt}
    \m 120 -407 \L 0 41 \S \m 49 -195 \C -13 13 -6 20 0 18 \L 0 20 \C
    0 18 -15 16 0 18 \L 0 20 \C 0 16 15 15 0 16 \S \m 64 -123 \C -13
    13 -6 20 0 18 \L 0 20 \C 0 16 15 15 0 16 \S \m 79 -51 \C -14 14 25
    17 0 15 \S \m 101 -51 \C 12 12 7 18 0 16 \S \m 150 -46 \L 0 41 \S
    \m 131 -195 \C 17 17 32 32 0 2 \L 0 20 \C 0 18 15 16 0 18 \L 0 20
    \C 0 16 -15 15 0 16 \S \m 180 -190 \C 0 19 30 8 0 19 \L 0 20 \C 0
    18 15 16 0 18 \L 0 20 \C 0 16 -15 15 0 16 \S \m 109 -339 \C -12 12
    -7 17 0 16 \S \m 131 -339 \C 12 12 7 17 0 16 \S \m 79 -267 \C -12
    12 -7 17 0 16 \S \m 71 -195 \C 4 4 0 26 0 15 \S \m 86 -123 \C 4 4
    0 26 0 15 \S \m 135 -118 \C 0 14 15 12 0 14 \S \m 101 -267 \C 12
    12 7 17 0 16 \S \m 109 -195 \C -20 20 46 10 0 15 \S \m 150 -262 \C
    0 17 30 6 0 17 \S \ov 104 -366 32 32 \S \ov 44 -222 32 32 \S \ov
    59 -150 32 32 \S \ov 74 -78 32 32 \S \re 134 -78 32 32 \S \ov 104
    -222 32 32 \S \re 164 -222 32 32 \S \ov 74 -294 32 32 \S \re 134
    -294 32 32 \S \re 119 -150 32 32 \S
  \end{sdpgf} \right]
  \displaybreak[0]\\
  \label{Eq:rule8}
  \left[ \begin{sdpgf}{0}{0}{240}{-412}{0.2pt}
    \m 120 -407 \L 0 41 \S \m 79 -267 \C -13 13 -6 20 0 18 \L 0 20 \C
    0 20 -30 12 0 20 \L 0 20 \C 0 18 -15 16 0 18 \L 0 20 \C 0 16 15 15
    0 16 \S \m 64 -123 \C -13 13 -6 20 0 18 \L 0 20 \C 0 16 15 15 0 16
    \S \m 90 -46 \L 0 41 \S \m 139 -51 \C -12 12 -7 18 0 16 \S \m 161
    -51 \C 14 14 -25 17 0 15 \S \m 146 -123 \C 17 17 32 32 0 2 \L 0 20
    \C 0 16 -15 15 0 16 \S \m 195 -118 \C 0 19 30 8 0 19 \L 0 20 \C 0
    16 -15 15 0 16 \S \m 109 -339 \C -12 12 -7 17 0 16 \S \m 131 -339
    \C 12 12 7 17 0 16 \S \m 101 -267 \C 4 4 0 26 0 15 \S \m 94 -195
    \C -12 12 -7 17 0 16 \S \m 86 -123 \C 4 4 0 26 0 15 \S \m 124 -123
    \C -20 20 46 10 0 15 \S \m 116 -195 \C 12 12 7 17 0 16 \S \m 165
    -190 \C 0 17 30 6 0 17 \S \m 150 -262 \C 0 14 15 12 0 14 \S \ov
    104 -366 32 32 \S \ov 74 -294 32 32 \S \ov 59 -150 32 32 \S \re 74
    -78 32 32 \S \ov 134 -78 32 32 \S \ov 119 -150 32 32 \S \re 179
    -150 32 32 \S \re 134 -294 32 32 \S \ov 89 -222 32 32 \S \re 149
    -222 32 32 \S
  \end{sdpgf} \right]
  \stackrel{\eqref{Eq:rule3}}{\equiv}
  \left[ \begin{sdpgf}{0}{0}{240}{-412}{0.2pt}
    \m 120 -407 \L 0 41 \S \m 79 -267 \C -13 13 -6 20 0 18 \L 0 20 \C
    0 18 -15 16 0 18 \L 0 20 \C 0 20 -30 12 0 20 \L 0 20 \C 0 16 15 15
    0 16 \S \m 79 -123 \C -15 15 -19 20 0 16 \L 0 20 \C 0 16 15 15 0
    16 \S \m 90 -46 \L 0 41 \S \m 139 -51 \C -12 12 -7 18 0 16 \S \m
    161 -51 \C 14 14 -25 17 0 15 \S \m 199 -51 \C -12 12 -7 18 0 16 \S
    \m 221 -51 \C 14 14 -25 17 0 15 \S \m 109 -339 \C -12 12 -7 17 0
    16 \S \m 131 -339 \C 12 12 7 17 0 16 \S \m 101 -267 \C 4 4 0 26 0
    15 \S \m 94 -195 \C -4 4 0 26 0 15 \S \m 101 -123 \C 14 14 -25 16
    0 15 \S \m 150 -118 \L 0 40 \S \m 165 -190 \C 0 19 30 8 0 19 \L 0
    20 \C 0 16 15 14 0 16 \S \m 150 -262 \C 0 14 15 12 0 14 \S \m 116
    -195 \C 14 14 20 20 0 11 \S \ov 104 -366 32 32 \S \ov 74 -294 32
    32 \S \ov 74 -150 32 32 \S \re 74 -78 32 32 \S \ov 134 -78 32 32
    \S \ov 194 -78 32 32 \S \re 134 -294 32 32 \S \ov 89 -222 32 32 \S
    \re 134 -150 32 32 \S \re 149 -222 32 32 \S
  \end{sdpgf} \right]
  \stackrel{\eqref{Eq:rule2}}{\equiv}
  \left[ \begin{sdpgf}{0}{0}{240}{-412}{0.2pt}
    \m 120 -407 \L 0 41 \S \m 79 -267 \C -13 13 -6 20 0 18 \L 0 20 \C
    0 18 -15 16 0 18 \L 0 20 \C 0 20 -30 12 0 20 \L 0 20 \C 0 16 15 15
    0 16 \S \m 79 -123 \C -15 15 -19 20 0 16 \L 0 20 \C 0 16 15 15 0
    16 \S \m 79 -51 \C -14 14 25 17 0 15 \S \m 101 -51 \C 12 12 7 18 0
    16 \S \m 150 -46 \L 0 41 \S \m 199 -51 \C -12 12 -7 18 0 16 \S \m
    221 -51 \C 14 14 -25 17 0 15 \S \m 109 -339 \C -12 12 -7 17 0 16
    \S \m 131 -339 \C 12 12 7 17 0 16 \S \m 101 -267 \C 4 4 0 26 0 15
    \S \m 94 -195 \C -4 4 0 26 0 15 \S \m 101 -123 \C 14 14 -25 16 0
    15 \S \m 150 -118 \L 0 40 \S \m 165 -190 \C 0 19 30 8 0 19 \L 0 20
    \C 0 16 15 14 0 16 \S \m 150 -262 \C 0 14 15 12 0 14 \S \m 116
    -195 \C 14 14 20 20 0 11 \S \ov 104 -366 32 32 \S \ov 74 -294 32
    32 \S \ov 74 -150 32 32 \S \ov 74 -78 32 32 \S \re 134 -78 32 32
    \S \ov 194 -78 32 32 \S \re 134 -294 32 32 \S \ov 89 -222 32 32 \S
    \re 134 -150 32 32 \S \re 149 -222 32 32 \S
  \end{sdpgf} \right]
  \stackrel{\eqref{Eq:rule3}}{\equiv}
  \left[ \begin{sdpgf}{0}{0}{240}{-412}{0.2pt}
    \m 120 -407 \L 0 41 \S \m 79 -267 \C -13 13 -6 20 0 18 \L 0 20 \C
    0 20 -30 12 0 20 \L 0 20 \C 0 18 -15 16 0 18 \L 0 20 \C 0 16 15 15
    0 16 \S \m 64 -123 \C -13 13 -6 20 0 18 \L 0 20 \C 0 16 15 15 0 16
    \S \m 79 -51 \C -14 14 25 17 0 15 \S \m 101 -51 \C 12 12 7 18 0 16
    \S \m 150 -46 \L 0 41 \S \m 146 -123 \C 17 17 32 32 0 2 \L 0 20 \C
    0 16 -15 15 0 16 \S \m 195 -118 \C 0 19 30 8 0 19 \L 0 20 \C 0 16
    -15 15 0 16 \S \m 109 -339 \C -12 12 -7 17 0 16 \S \m 131 -339 \C
    12 12 7 17 0 16 \S \m 101 -267 \C 4 4 0 26 0 15 \S \m 94 -195 \C
    -12 12 -7 17 0 16 \S \m 86 -123 \C 4 4 0 26 0 15 \S \m 124 -123 \C
    -20 20 46 10 0 15 \S \m 116 -195 \C 12 12 7 17 0 16 \S \m 165 -190
    \C 0 17 30 6 0 17 \S \m 150 -262 \C 0 14 15 12 0 14 \S \ov 104
    -366 32 32 \S \ov 74 -294 32 32 \S \ov 59 -150 32 32 \S \ov 74 -78
    32 32 \S \re 134 -78 32 32 \S \ov 119 -150 32 32 \S \re 179 -150
    32 32 \S \re 134 -294 32 32 \S \ov 89 -222 32 32 \S \re 149 -222
    32 32 \S
  \end{sdpgf} \right]
  \displaybreak[0]\\
  \label{Eq:rule9}
  \left[ \begin{sdpgf}{0}{0}{210}{-412}{0.2pt}
    \m 105 -407 \L 0 41 \S \m 64 -267 \C -13 13 -6 20 0 18 \L 0 20 \C
    0 18 -15 16 0 18 \L 0 20 \C 0 18 -15 16 0 18 \L 0 20 \C 0 16 15 15
    0 16 \S \m 49 -51 \C -14 14 25 17 0 15 \S \m 71 -51 \C 12 12 7 18
    0 16 \S \m 120 -46 \L 0 41 \S \m 169 -51 \C -12 12 -7 18 0 16 \S
    \m 191 -51 \C 14 14 -25 17 0 15 \S \m 94 -339 \C -12 12 -7 17 0 16
    \S \m 116 -339 \C 12 12 7 17 0 16 \S \m 86 -267 \C 4 4 0 26 0 15
    \S \m 79 -195 \C -13 13 -6 20 0 18 \L 0 20 \L 0 46 \S \m 105 -118
    \C 0 14 15 12 0 14 \S \m 165 -118 \C 0 14 15 12 0 14 \S \m 135
    -262 \C 0 14 15 12 0 14 \S \m 101 -195 \C 4 4 0 26 0 15 \S \m 150
    -190 \C 0 14 15 12 0 14 \S \ov 89 -366 32 32 \S \ov 59 -294 32 32
    \S \ov 44 -78 32 32 \S \re 104 -78 32 32 \S \ov 164 -78 32 32 \S
    \re 119 -294 32 32 \S \ov 74 -222 32 32 \S \re 89 -150 32 32 \S
    \re 149 -150 32 32 \S \re 134 -222 32 32 \S
  \end{sdpgf} \right]
  \stackrel{\eqref{Eq:rule1}}{\equiv}
  \left[ \begin{sdpgf}{0}{0}{210}{-412}{0.2pt}
    \m 105 -407 \L 0 41 \S \m 75 -262 \C 0 19 -30 8 0 19 \L 0 20 \C 0
    18 -15 16 0 18 \L 0 20 \C 0 18 -15 16 0 18 \L 0 20 \C 0 16 15 15 0
    16 \S \m 49 -51 \C -14 14 25 17 0 15 \S \m 71 -51 \C 12 12 7 18 0
    16 \S \m 120 -46 \L 0 41 \S \m 169 -51 \C -12 12 -7 18 0 16 \S \m
    191 -51 \C 14 14 -25 17 0 15 \S \m 94 -339 \C -12 12 -7 17 0 16 \S
    \m 116 -339 \C 12 12 7 17 0 16 \S \m 79 -195 \C -13 13 -6 20 0 18
    \L 0 20 \L 0 46 \S \m 105 -118 \C 0 14 15 12 0 14 \S \m 165 -118
    \C 0 14 15 12 0 14 \S \m 124 -267 \C -14 14 -20 20 0 11 \S \m 146
    -267 \C 4 4 0 26 0 15 \S \m 101 -195 \C 4 4 0 26 0 15 \S \m 150
    -190 \C 0 14 15 12 0 14 \S \ov 89 -366 32 32 \S \re 59 -294 32 32
    \S \ov 44 -78 32 32 \S \re 104 -78 32 32 \S \ov 164 -78 32 32 \S
    \ov 119 -294 32 32 \S \ov 74 -222 32 32 \S \re 89 -150 32 32 \S
    \re 149 -150 32 32 \S \re 134 -222 32 32 \S
  \end{sdpgf} \right]
  \stackrel{\eqref{Eq:rule5}}{\equiv}
  \left[ \begin{sdpgf}{0}{0}{210}{-412}{0.2pt}
    \m 105 -407 \L 0 41 \S \m 75 -262 \C 0 19 -30 8 0 19 \L 0 20 \C 0
    18 -15 16 0 18 \L 0 20 \C 0 18 -15 16 0 18 \L 0 20 \C 0 16 15 15 0
    16 \S \m 49 -51 \C -14 14 25 17 0 15 \S \m 71 -51 \C 12 12 7 18 0
    16 \S \m 109 -51 \C -14 14 25 17 0 15 \S \m 131 -51 \C 12 12 7 18
    0 16 \S \m 180 -46 \L 0 41 \S \m 94 -339 \C -12 12 -7 17 0 16 \S
    \m 116 -339 \C 12 12 7 17 0 16 \S \m 79 -195 \C -13 13 -6 20 0 18
    \L 0 20 \L 0 46 \S \m 105 -118 \C 0 14 15 12 0 14 \S \m 165 -118
    \C 0 14 15 12 0 14 \S \m 124 -267 \C -14 14 -20 20 0 11 \S \m 146
    -267 \C 4 4 0 26 0 15 \S \m 101 -195 \C 4 4 0 26 0 15 \S \m 150
    -190 \C 0 14 15 12 0 14 \S \ov 89 -366 32 32 \S \re 59 -294 32 32
    \S \ov 44 -78 32 32 \S \ov 104 -78 32 32 \S \re 164 -78 32 32 \S
    \ov 119 -294 32 32 \S \ov 74 -222 32 32 \S \re 89 -150 32 32 \S
    \re 149 -150 32 32 \S \re 134 -222 32 32 \S
  \end{sdpgf} \right]
  \stackrel{\eqref{Eq:rule1}}{\equiv}
  \left[ \begin{sdpgf}{0}{0}{210}{-412}{0.2pt}
    \m 105 -407 \L 0 41 \S \m 64 -267 \C -13 13 -6 20 0 18 \L 0 20 \C
    0 18 -15 16 0 18 \L 0 20 \C 0 18 -15 16 0 18 \L 0 20 \C 0 16 15 15
    0 16 \S \m 49 -51 \C -14 14 25 17 0 15 \S \m 71 -51 \C 12 12 7 18
    0 16 \S \m 109 -51 \C -14 14 25 17 0 15 \S \m 131 -51 \C 12 12 7
    18 0 16 \S \m 180 -46 \L 0 41 \S \m 94 -339 \C -12 12 -7 17 0 16
    \S \m 116 -339 \C 12 12 7 17 0 16 \S \m 86 -267 \C 4 4 0 26 0 15
    \S \m 79 -195 \C -13 13 -6 20 0 18 \L 0 20 \L 0 46 \S \m 105 -118
    \C 0 14 15 12 0 14 \S \m 165 -118 \C 0 14 15 12 0 14 \S \m 135
    -262 \C 0 14 15 12 0 14 \S \m 101 -195 \C 4 4 0 26 0 15 \S \m 150
    -190 \C 0 14 15 12 0 14 \S \ov 89 -366 32 32 \S \ov 59 -294 32 32
    \S \ov 44 -78 32 32 \S \ov 104 -78 32 32 \S \re 164 -78 32 32 \S
    \re 119 -294 32 32 \S \ov 74 -222 32 32 \S \re 89 -150 32 32 \S
    \re 149 -150 32 32 \S \re 134 -222 32 32 \S
  \end{sdpgf} \right]
  \displaybreak[0]\\
  \label{Eq:rule10}
  \left[ \begin{sdpgf}{0}{0}{210}{-412}{0.2pt}
    \m 105 -407 \L 0 41 \S \m 19 -51 \C -14 14 25 17 0 15 \S \m 41 -51
    \C 12 12 7 18 0 16 \S \m 79 -123 \C -4 4 0 30 0 17 \L 0 20 \C 0 16
    15 15 0 16 \S \m 120 -46 \L 0 41 \S \m 169 -51 \C -12 12 -7 18 0
    16 \S \m 191 -51 \C 14 14 -25 17 0 15 \S \m 94 -339 \C -12 12 -7
    17 0 16 \S \m 116 -339 \C 12 12 7 17 0 16 \S \m 64 -267 \C -13 13
    -6 20 0 18 \L 0 20 \L 0 52 \L 0 20 \C 0 16 -15 14 0 16 \S \m 90
    -190 \L 0 40 \S \m 101 -123 \C 12 12 7 17 0 16 \S \m 150 -118 \C 0
    17 30 6 0 17 \S \m 86 -267 \C 4 4 0 26 0 15 \S \m 135 -262 \C 0 14
    15 12 0 14 \S \m 150 -190 \L 0 40 \S \ov 89 -366 32 32 \S \ov 14
    -78 32 32 \S \ov 74 -150 32 32 \S \re 104 -78 32 32 \S \ov 164 -78
    32 32 \S \ov 59 -294 32 32 \S \re 119 -294 32 32 \S \re 74 -222 32
    32 \S \re 134 -150 32 32 \S \re 134 -222 32 32 \S
  \end{sdpgf} \right]
  \stackrel{\eqref{Eq:rule5}}{\equiv}
  \left[ \begin{sdpgf}{0}{0}{240}{-412}{0.2pt}
    \m 120 -407 \L 0 41 \S \m 19 -51 \C -20 20 46 10 0 16 \S \m 41 -51
    \C 14 14 20 20 0 12 \S \m 90 -46 \C 0 14 15 13 0 14 \S \m 150 -46
    \C 0 14 -15 13 0 14 \S \m 199 -51 \C -14 14 -20 20 0 12 \S \m 221
    -51 \C 20 20 -46 10 0 16 \S \m 109 -339 \C -13 13 -6 20 0 18 \L 0
    20 \L 0 46 \S \m 131 -339 \C 4 4 0 26 0 15 \S \m 79 -195 \C -13 13
    -6 20 0 18 \L 0 20 \C 0 19 -30 8 0 19 \S \m 105 -118 \C 0 14 -15
    12 0 14 \S \m 154 -123 \C -4 4 0 26 0 15 \S \m 176 -123 \C 14 14
    20 20 0 11 \S \m 101 -195 \C 4 4 0 26 0 15 \S \m 135 -262 \C 0 14
    15 12 0 14 \S \m 150 -190 \C 0 14 15 12 0 14 \S \ov 104 -366 32 32
    \S \ov 14 -78 32 32 \S \re 74 -78 32 32 \S \re 134 -78 32 32 \S
    \ov 194 -78 32 32 \S \ov 74 -222 32 32 \S \re 119 -294 32 32 \S
    \re 89 -150 32 32 \S \ov 149 -150 32 32 \S \re 134 -222 32 32 \S
  \end{sdpgf} \right]
  \stackrel{\eqref{Eq:rule1}}{\equiv}
  \left[ \begin{sdpgf}{0}{0}{240}{-412}{0.2pt}
    \m 120 -407 \L 0 41 \S \m 19 -51 \C -20 20 46 10 0 16 \S \m 41 -51
    \C 14 14 20 20 0 12 \S \m 90 -46 \C 0 14 15 13 0 14 \S \m 139 -51
    \C -4 4 0 27 0 15 \S \m 161 -51 \C 4 4 0 27 0 15 \S \m 210 -46 \C
    0 14 -15 13 0 14 \S \m 109 -339 \C -13 13 -6 20 0 18 \L 0 20 \L 0
    46 \S \m 131 -339 \C 4 4 0 26 0 15 \S \m 79 -195 \C -13 13 -6 20 0
    18 \L 0 20 \C 0 19 -30 8 0 19 \S \m 105 -118 \C 0 14 -15 12 0 14
    \S \m 154 -123 \C -4 4 0 26 0 15 \S \m 176 -123 \C 14 14 20 20 0
    11 \S \m 101 -195 \C 4 4 0 26 0 15 \S \m 135 -262 \C 0 14 15 12 0
    14 \S \m 150 -190 \C 0 14 15 12 0 14 \S \ov 104 -366 32 32 \S \ov
    14 -78 32 32 \S \re 74 -78 32 32 \S \ov 134 -78 32 32 \S \re 194
    -78 32 32 \S \ov 74 -222 32 32 \S \re 119 -294 32 32 \S \re 89
    -150 32 32 \S \ov 149 -150 32 32 \S \re 134 -222 32 32 \S
  \end{sdpgf} \right]
  \stackrel{\eqref{Eq:rule5}}{\equiv}
  \left[ \begin{sdpgf}{0}{0}{210}{-412}{0.2pt}
    \m 105 -407 \L 0 41 \S \m 19 -51 \C -14 14 25 17 0 15 \S \m 41 -51
    \C 12 12 7 18 0 16 \S \m 79 -123 \C -4 4 0 30 0 17 \L 0 20 \C 0 16
    15 15 0 16 \S \m 109 -51 \C -14 14 25 17 0 15 \S \m 131 -51 \C 12
    12 7 18 0 16 \S \m 180 -46 \L 0 41 \S \m 94 -339 \C -12 12 -7 17 0
    16 \S \m 116 -339 \C 12 12 7 17 0 16 \S \m 64 -267 \C -13 13 -6 20
    0 18 \L 0 20 \L 0 52 \L 0 20 \C 0 16 -15 14 0 16 \S \m 90 -190 \L
    0 40 \S \m 101 -123 \C 12 12 7 17 0 16 \S \m 150 -118 \C 0 17 30 6
    0 17 \S \m 86 -267 \C 4 4 0 26 0 15 \S \m 135 -262 \C 0 14 15 12 0
    14 \S \m 150 -190 \L 0 40 \S \ov 89 -366 32 32 \S \ov 14 -78 32 32
    \S \ov 74 -150 32 32 \S \ov 104 -78 32 32 \S \re 164 -78 32 32 \S
    \ov 59 -294 32 32 \S \re 119 -294 32 32 \S \re 74 -222 32 32 \S
    \re 134 -150 32 32 \S \re 134 -222 32 32 \S
  \end{sdpgf} \right]
\end{gather}
And so on\dots\ When we stopped it, the program had $1$ rule
\eqref{Eq:rule1} of order~(number of vertices) $3$, $1$ rule
\eqref{Eq:rule2} of order~$5$, $1$ rule \eqref{Eq:rule3} of order~$7$,
$2$ rules (\ref{Eq:rule4},\ref{Eq:rule5}) of order~$8$, $1$ rule
\eqref{Eq:rule6} of order~$9$, $4$ rules
(\ref{Eq:rule7}--\ref{Eq:rule10}) of order~$10$, $7$ rules of
order~$11$, $12$ rules of order~$12$, $19$ rules of order~$13$, and $38$
rules of order~$14$. Besides those $85$ ambiguities that had given rise
to new rules, $280$ had turned out to be resolvable and $22417$ had
still not been processed; obviously the program wasn't going to
finish anytime soon, and it's a fair guess that the complete
rewriting system it sought to compute is in fact infinite. Certainly
\eqref{Eq:rule1}, \eqref{Eq:rule2}, \eqref{Eq:rule3}, and
\eqref{Eq:rule6} look suspiciously like the beginning of an infinite
family of rules, and indeed the expected sequence with one tower of
$\mathsf{m}$'s and another tower of $\mathsf{a}$'s continues for as long as we have
run the computations.


What is now our next step, when automated deduction has failed to
deliver a complete answer? One approach is to try to guess the
general pattern for these rules, and from that construct a provably
complete rewriting system; we shall return to that problem in a later
article.
Right here and now, it is however possible to wash out several pieces
of hard information even from the incomplete rewriting system
presented above.

\subsection{Hilbert series and formal languages}
\label{Ssec:Hilbert}

A useful observation about the hom-associativity axiom
\eqref{Eq:rule1} is that it is homogeneous in pretty much every sense
imaginable: there are the same number of $\mathsf{m}$'s in the left and
right hand sides, there are the same number of $\mathsf{a}$'s in the left
and right hand sides, and the inputs are all at the same height in
the left as in the right hand sides. (The last is not even true for the
ordinary associativity rule~\eqref{Eq:associativitet}, so from a very
abstract symbolic point of view, hom-associativity may actually be
regarded as a homogenised form of ordinary associativity.) It is a
well-known principle in Gr\"obner basis calculations that CPC
procedures working on homogeneous rewriting systems \emph{only generates
homogeneous rules} and \emph{never derives smaller rules from larger
ones}; once the procedure has processed all ambiguities up to a
particular order, one knows for sure that no more rules of that order
remain to be discovered. Hence the ten rules shown above are all
there are of order $10$ or less, and since no advanceable map of those
used for Theorem~\ref{S:ODL} can reduce the order, it follows that
those rules do effectively describe $\mathcal{H\mkern-4muA}\mkern-1mu\mathit{ss}$ up to order~$10$.
There is of course nothing special about order~$10$, so we may state
these observations more formally as follows.

\begin{lemma} \label{L:ApproxHAss}
  Let $Y_{k,\ell}$ be the subset of $\mc{Y}_\Omega$ whose elements contain
  exactly $k$ vertices $\mathsf{a}$ and exactly $\ell$ vertices $\mathsf{m}$; it
  follows that \(\mc{Y}_\Omega(n) = \bigcup_{k=0}^\infty Y_{k,n-1}\).
  Let $S_{k,\ell}$ be the set of rules the CPC procedure has generated
  from $(\mathsf{m}\mathsf{a}1\mathsf{m}23,\mathsf{m}\mathsf{m}12\mathsf{a}3)$ after processing all
  ambiguities at sites in \(\bigcup_{i=0}^k \bigcup_{j=0}^\ell Y_{i,j}\)
  but no ambiguities with sites outside this set. Let
  \(S = \bigcup_{k,\ell\in\mathbb{N}} S_{k,\ell}\). Then the following holds:
  \begin{enumerate}
    \item
      \(S_{1,2} = \bigl\{ (\mathsf{m}\mathsf{a}1\mathsf{m}23,\mathsf{m}\mathsf{m}12\mathsf{a}3) \bigr\}\).
    \item
      \(S_{k,\ell} \subseteq S_{k+1,\ell}\) and \(S_{k,\ell} \subseteq
      S_{k,\ell+1}\) for all \(k,\ell \in \mathbb{N}\).
    \item \label{Monotonicitet}
      \(\mathrm{Irr}(S_{k,\ell}) \supseteq \mathrm{Irr}(S_{k+1,\ell})\) and
      \(\mathrm{Irr}(S_{k,\ell}) \supseteq \mathrm{Irr}(S_{k,\ell+1})\) for all
      \(k,\ell \in \mathbb{N}\).
    \item
      All ambiguities of $S$ are resolvable.
    \item \label{Konvergens}
      \(\mathrm{Irr}(S) \cap Y_{k,\ell} = \mathrm{Irr}(S_{k,\ell}) \cap Y_{k,\ell}\).
    \item
      Every element of $Y_{k,\ell}$ has a unique normal form modulo
      $S_{i,j}$, for all \(i \geqslant k\) and \(j \geqslant \ell\).
  \end{enumerate}
\end{lemma}


To finish off, we shall apply a bit of formal
language theory to compute the beginning of \index{Hilbert series for hom-associative operad}  the \emph{Hilbert series}
of $\mathcal{H\mkern-4muA}\mkern-1mu\mathit{ss}$. The kind of information encoded in this is, just
like the \(\dim \mathcal{A}\mkern-1mu\mathit{ss}(n) = n!\) result mentioned above, basically
the numbers of dimensions of the various components of the operad,
although in the case of $\mathcal{H\mkern-4muA}\mkern-1mu\mathit{ss}$ it is trivial to see that \(\dim
\mathcal{H\mkern-4muA}\mkern-1mu\mathit{ss}(n) = \infty\) for all \(n \geqslant 0\) since inserting
more $\mathsf{a}$'s into an expression does not change its arity. Instead
one should partition by both $\mathsf{a}$ and $\mathsf{m}$ to get
finite-dimensional components. Furthermore there is a rather boring
factorial factor which is due to the action of $\Sigma_n$, so we
restrict attention to plane monomials, factor out that factorial,
and define the Hilbert series of $\mathcal{H\mkern-4muA}\mkern-1mu\mathit{ss}$ to be the formal power
series
\begin{equation}
  H(a,m) = \sum_{i,j\in\mathbb{N}}
  \frac{ \bigl| Y_{i,j} \cap \mathrm{Irr}(S) \bigr| }{ (j+1)! } a^i m^j
  \text{.}
\end{equation}
Note that this is also the Hilbert series of the free hom-associative
algebra with one generator on which $f_\mathsf{a}$ acts freely. Indeed,
that algebra is preferably constructed as $\mc{R}\{\Omega'\}(0) \big/
\mc{I}(S)(0)$ where \(\Omega' = \bigl\{ \mathsf{m}({,}), \mathsf{a}(), \mathsf{x}
\bigr\}\), and since no rule in $S$ changes $\mathsf{x}$ in any way, it
follows that \(\mathrm{Irr}(S) \subseteq \mc{R}\{\Omega\}\) is in bijective
correspondence to \(\mathrm{Irr}(S)(0) \subseteq
\mc{R}\{\Omega'\}(0)\)---just put an $\mathsf{x}$ in every input! However,
if one prefers have the Hilbert series for the free algebra counting
$\mathsf{a}$ and $\mathsf{x}$ rather than $\mathsf{a}$ and $\mathsf{m}$, then it should
instead be stated as $x H(a,x)$, since there is always one $\mathsf{x}$
more in an element of $\mc{Y}_{\Omega'}(0)$ than there are $\mathsf{m}$'s.
Finally, the Hilbert series for the free hom-associative algebra with
$k$ generators $\mathsf{x}_1,\dotsc,\mathsf{x}_k$ is $kx H(a,kx)$, since there
for every constant symbol (which is what $x$ becomes the counting
variable for) are $k$ choices of what that symbol should be.

As approximations of $H(a,m)$, we furthermore define
\begin{equation}
  H_{k,\ell}(a,m) = \sum_{i,j\in\mathbb{N}}
  \frac{ \bigl| Y_{i,j} \cap \mathrm{Irr}(S_{k,\ell}) \bigr| }{ (j+1)! } a^i m^j
  \qquad\text{for all \(k,\ell\in\mathbb{N}\).}
\end{equation}
By claim~\ref{Monotonicitet} of Lemma~\ref{L:ApproxHAss}, \(H_{i,j}
\geqslant H_{k,\ell}\) coefficient by coefficient whenever \(i \leqslant
k\) and \(j \leqslant \ell\). By claim~\ref{Konvergens}, the coefficient
of $a^i m^j$ in $H(a,m)$ is equal to the coefficient in $H_{k,\ell}(a,m)$
whenever \(i \leqslant k\) and \(j \leqslant \ell\). Therefore, when one
wishes to compute the beginning of $H(a,m)$, one may alternatively
compute the beginning of $H_{k,\ell}(a,m)$ for sufficiently large $k$
and $\ell$.

To get an initial bound, let us first compute $H_{0,0}(a,m)$. From
the basic observation that a plane element of $\mc{Y}_\Omega$ is either
$\mathrm{id}$, $\mathsf{a}1 \circ \nu$ for some plane \(\nu \in \mc{Y}_\Omega\), or
$\mathsf{m}12 \circ \nu_1 \otimes \nu_2$ for some plane \(\nu_1,\nu_2 \in
\mc{Y}_\Omega\), it follows that the language\footnote{
  In formal language theory, a `language' is simply some set of the
  kind of objects being considered.
} $L$ of all plane elements of $\mc{Y}_\Omega$ satisfies the equation
\(L = \{\mathrm{id}\} \cup ( \mathsf{a}1 \circ L ) \cup ( \mathsf{m}12 \circ L \otimes L
)\), and consequently that $H_{0,0}$ satisfies
the functional equation
\begin{equation} \label{Eq1:H_{0,0}}
  H_{0,0}(a,m) = 1 + a H_{0,0}(a,m) + m H_{0,0}(a,m)^2;
\end{equation}
the details of this correspondence between combinatorial
constructions and functional equations can be found in for example
\cite[Ch.~1]{AnalyticCombinatorics}. Solving that equation
symbolically yields
\begin{equation}
  H_{0,0}(a,m) = \frac{ 1 - a - \sqrt{ (1-a)^2 - 4m } }{ 2m }
\end{equation}
and using Newton's generalised binomial theorem one can even get a
closed form formula for the coefficients:
\begin{align*}
  H_{0,0}(a,m) ={}&
  \frac{1}{2m} \Biggl( 1 - a - \sum_{n=0}^\infty
    \binom{\frac{1}{2}}{n} \bigl( (1-a)^2 \bigr)^{\frac{1}{2}-n}
    (-4m)^n \Biggr) = \\ ={}&
  - \frac{1}{2m} \sum_{n=1}^\infty
    \binom{\frac{1}{2}}{n} (1-a)^{1-2n}
    (-4m)^n \,
  \stackrel{(\ell = n-1)}{=} \displaybreak[0]\\={}&
  2 \sum_{\ell=0}^\infty
    \binom{\frac{1}{2}}{\ell+1} (1-a)^{-2\ell-1} (-4m)^\ell
  = \displaybreak[0]\\ ={}&
  \sum_{\ell=0}^\infty \sum_{k=0}^\infty 2 \binom{\frac{1}{2}}{\ell+1}
    \binom{-2\ell-1}{k} (-a)^k (-4m)^\ell
  = \\ ={}&
  \sum_{k,\ell\in\mathbb{N}}
    \frac{1}{\ell+1} \binom{k+2\ell}{k,\ell,\ell} a^k m^\ell
  \text{.}
\end{align*}
As expected, the coefficients for \(\ell=0\) are all $1$ and the
coefficients for \(k=0\) are the Catalan numbers. These remain that
way in all $H_{k,\ell}$, but away from the axes the various rules makes
a difference. In order to determine how much, it is time to take some
rules into account.

For any finite set of rules, it is straightforward to set up a system
of equations for the language $L_0$ of plane monomials that are
reducible by at least one of these rules; in the case of $S_{1,2}$,
one such equation system is
\begin{subequations} \label{Eq1:H_{1,2}}
\begin{align}
  L_0 ={}& (\mathsf{a}1 \circ L_0 ) \cup (\mathsf{m}12 \circ L_0 \otimes L_1)
    \cup (\mathsf{m}12 \circ L_1 \otimes L_0) \cup
    (\mathsf{m}12 \circ L_2 \otimes L_3) \text{,}\\
  L_1 ={}& (\mathsf{m}12 \circ L_1 \otimes L_1) \cup (\mathsf{a}1 \circ L_1 )
    \cup \{\mathrm{id}\} \text{,}\\
  L_2 ={}& \mathsf{a}1 \circ L_1 \text{,}\\
  L_3 ={}& \mathsf{m}12 \circ L_1 \otimes L_1
\end{align}
\end{subequations}
(where we as usual consider operad composition of sets to denote the
sets of operad elements that can be produced by applying the
composition to elements of the given sets).
A more suggestive presentation might however be as the BNF grammar
\begin{align*}
  \meta{reducible} ::={}&
    \mathsf{a}\meta{reducible} \mid
    \mathsf{m}\meta{reducible}\meta{arbitrary} \mid
    \mathsf{m}\meta{arbitrary}\meta{reducible} \\ & \quad{} \mid
    \mathsf{m}\meta{left}\meta{right}\\
  \meta{arbitrary} ::={}& \mathsf{a}\meta{arbitrary} \mid
    \mathsf{m}\meta{arbitrary}\meta{arbitrary} \mid \Box_i\\
  \meta{left} ::={}& \mathsf{a}\meta{arbitrary}\\
  \meta{right} ::={}& \mathsf{m}\meta{arbitrary}\meta{arbitrary}
\end{align*}
whose informal interpretation is that a Polish term is
\meta{reducible} by $S_{1,2}$ if one of the children of the root node
is itself \meta{reducible}, or if the root node is an $\mathsf{m}$ whose
\meta{left} child is an $\mathsf{a}$ and whose \meta{right} child is an
$\mathsf{m}$. This can be trivially
extended to larger sets of rules by adding to the formula for $L_0$
one production for each new rule (describing the root of the $m _s$
of that rule) and one new variable (together with its defining
equation) for every internal edge in the $m _s$ of the new rule.
Hence if also taking \eqref{Eq:rule2} into account, the system grows
to
\begin{align*}
  L_0 ={}& (\mathsf{a}1 \circ L_0 ) \cup (\mathsf{m}12 \circ L_0 \otimes L_1)
    \cup (\mathsf{m}12 \circ L_1 \otimes L_0) \\ &\qquad{} \cup
    (\mathsf{m}12 \circ L_2 \otimes L_3) \cup
    (\mathsf{m}12 \circ L_4 \otimes L_5) \text{,}\\
  L_1 ={}& (\mathsf{m}12 \circ L_1 \otimes L_1) \cup (\mathsf{a}1 \circ L_1 )
    \cup \{\mathrm{id}\} \text{,}\\
  L_2 ={}& \mathsf{a}1 \circ L_1 \text{,}\\
  L_3 ={}& \mathsf{m}12 \circ L_1 \otimes L_1 \text{,}\\
  L_4 ={}& \mathsf{m}12 \circ L_1 \otimes L_6 \text{,}\\
  L_5 ={}& \mathsf{a}1 \circ L_7 \text{,}\\
  L_6 ={}& \mathsf{a}1 \circ L_1 \text{,}\\
  L_7 ={}& \mathsf{m}12 \circ L_1 \otimes L_1 \text{.}
\end{align*}
Smaller systems for the same $L_0$ are often possible (and can save
work in the next step), but here we are content with observing that
a finite system exists.

While the system \eqref{Eq1:H_{1,2}} is of the same general type as
the equation that was used to derive \eqref{Eq1:H_{0,0}}, it would
not be correct to simply convert it in the same way to an equation
system for $H_{1,2}$, since there is a qualitative difference: the
unions in \eqref{Eq1:H_{1,2}} are not in general disjoint, for
example because \(L_0 \subset L_1\) and thus \(\mathsf{m}12 \circ L_0
\otimes L_0 \subseteq \mathsf{m}12 \circ L_0 \otimes L_1, \mathsf{m}12 \circ L_1
\otimes L_0\). This may be possible to overcome through
inclusion--exclusion style combinatorics, but we would rather like to
attack this issue using tools from formal language theory. In the
terminology of~\cite{TATA}, an equation system such as
\eqref{Eq1:H_{1,2}} defines a nondeterministic finite top-down tree
automaton; it is finite because the set of states is $\{0,1,2,3\}$
(finite) and it is the nondeterminism that can cause the unions to
be non-disjoint. By the Subset Construction~\cite[Th.~1.1.9]{TATA}
however, there exists an equivalent deterministic finite bottom-up
tree automaton whose states are subsets of the set of top-down
states; moreover this bottom-up automaton may be regarded as an
$\Omega$-algebra $\bigl( A, \{f_x\}_{x\in\Omega} \bigr)$. In the case
of \eqref{Eq1:H_{1,2}}, this $\Omega$-algebra has
\begin{equation*}
  A = \bigl\{
  \{1\}, \{1,2\}, \{1,3\}, \{0,1,3\}, \{0,1,2\}
  \bigr\}
\end{equation*}
and operations given by the tables
\begin{center}
  \begin{tabular}{c|c|ccccc} 
    first&&\multicolumn{5}{l}{$f_\mathsf{m}$ when second operad is:}\\
    operand& $f_\mathsf{a}$&
      $\{1\}$& $\{1,2\}$& $\{1,3\}$& $\{0,1,3\}$& $\{0,1,2\}$\\
    \hline
    $\{1\}$& $\{1,2\}$&
      $\{1,3\}$& $\{1,3\}$& $\{1,3\}$& $\{0,1,3\}$& $\{0,1,3\}$\\
    $\{1,2\}$& $\{1,2\}$&
      $\{1,3\}$& $\{1,3\}$& $\{0,1,3\}$& $\{0,1,3\}$& $\{0,1,3\}$\\
    $\{1,3\}$& $\{1,2\}$&
      $\{1,3\}$& $\{1,3\}$& $\{1,3\}$& $\{0,1,3\}$& $\{0,1,3\}$\\
    $\{0,1,3\}$& $\{0,1,2\}$&
      $\{0,1,3\}$& $\{0,1,3\}$& $\{0,1,3\}$& $\{0,1,3\}$& $\{0,1,3\}$\\
    $\{0,1,2\}$& $\{0,1,2\}$&
      $\{0,1,3\}$& $\{0,1,3\}$& $\{0,1,3\}$& $\{0,1,3\}$& $\{0,1,3\}$
  \end{tabular}
\end{center}
When such an $\Omega$-algebra $\bigl( A, \{f_x\}_{x\in\Omega} \bigr)$
is given, the equation system of generating functions takes the form
\begin{multline}
  G_b(a,m) = a \!\sum_{\substack{c \in A\\f_\mathsf{a}(c)=b}}\! G_c(a,m) +
  m \!\!\!\sum_{\substack{c,d \in A\\f_\mathsf{m}(c,d)=b}}\!\!\!
    G_c(a,m) G_d(a,m) +
  \begin{cases}
    1& \text{if \(b = \{1\}\),}\\
    0 & \text{otherwise}
  \end{cases}
  \\\text{for all \(b \in A\)}
\end{multline}
where the extra term for \(b=\{1\}\) is because that is the state
that inputs are considered to be in. The generating function for
reducible plane monomials is the sum of all $G_b$ such that \(b \owns
0\), since $0$ was the top-down \meta{reducible} state, and
consequently the generating function for irreducible plane monomials
is the sum of all $G_b$ such that \(b \not\owns 0\). Thus we have
\begin{align*}
  H_{1,2}(a,m) ={}& G_{\{1\}}(a,m) + G_{\{1,2\}}(a,m) +
    G_{\{1,3\}}(a,m) \text{,}\\
  G_{\{1\}}(a,m) ={}& 1 \text{,}\\
  G_{\{1,2\}}(a,m) ={}& a H_{1,2}(a,m) \text{,}\\
  G_{\{1,3\}}(a,m) ={}& m H_{1,2}(a,m)^2 - m G_{\{1,2\}}(a,m)G_{\{1,3\}}(a,m)
\end{align*}
where the definition of $H_{1,2}(a,m)$ was used to shorten the last
two right hand sides a bit. Solving as above is still possible, but
results in the somewhat messier expression
\begin{align*}
  H_{1,2}(a,m) ={}& \frac{1 - a - am^2 - \sqrt{ (1 - a - am^2)^2 +
    4 (1 - am + a^2m) m } }{ 2 (1 - am + a^2m) m }
  = \\ ={}&
  \sum_{k=0}^\infty 2 \binom{\frac{1}{2}}{k+1} (1 - a - am^2)^{-1-2k}
    4^k (1 - am + a^2m)^k m^k = \dotsb
\end{align*}
which is probably not so important to put on closed form; the
interesting quantity is $H(a,m)$, and the terms in $H_{1,2}$ which
coincide with their counterparts in $H(a,m)$ can be determined by an
ansatz in the equation system already.

\begin{theorem}
  The Hilbert series $H(a,m)$ for the hom-associative operad
  $\mathcal{H\mkern-4muA}\mkern-1mu\mathit{ss}$ satisfies
  \begin{math}
    H(a,m) =
    1 + m + a + 2 m^2 + 3 a m + a^2 + 5 m^3 + 9 a m^2 + 6 a^2 m + a^3 +
      14 m^4 + 30 a m^3 + 26 a^2 m^2 + 10 a^3 m + a^4 + 42 m^5 +
      105 a m^4 + 110 a^2 m^3 + 60 a^3 m^2 + 15 a^4 m + a^5 + 132 m^6 +
      378 a m^5 + 465 a^2 m^4 + 315 a^3 m^3 + 120 a^4 m^2 + 21 a^5 m +
      a^6 + 429 m^7 + 1386 a m^6 + 1960 a^2 m^5 + 1575 a^3 m^4 +
      770 a^4 m^3 + 217 a^5 m^2 + 28 a^6 m + a^7 + 1430 m^8 + 5148 a m^7
      + 8232 a^2 m^6 + 7644 a^3 m^5 + 4494 a^4 m^4 + 1680 a^5 m^3 +
      364 a^6 m^2 + 36 a^7 m + a^8 + \dotsb
  \end{math}. In particular, the difference to the Hilbert series
  $H_{0,0}(a,m)$ for the free hom-algebra operad is
  \begin{multline*}
    H_{0,0}(a,m) - H(a,m) = \\
    \begin{array}[t]{r@{}l @{{}+{}} r@{}l @{{}+{}} r@{}l
      @{{}+{}} r@{}l @{{}+{}} r@{}l @{{}+{}} r@{}l}
    \multicolumn{1}{l}{=}
    & a m^2 & 4&a^2 m^2& 10& a^3 m^2& 20& a^4 m^2 & 35& a^5 m^2&
      56& a^6 m^2 +\\
    5&a m^3 & 30& a^2 m^3& 105& a^3 m^3& 280& a^4 m^3& 630& a^5 m^3\\
    21& a m^4& 165& a^2 m^4& 735& a^3 m^4 & 2436& a^4 m^4\\
    84& a m^5& 812& a^2 m^5& 4368& a^3 m^5&\\
    330&a m^6& 3780& a^2 m^6&\\
    1287& a m^7& \dotsb
    \end{array}
  \end{multline*}
\end{theorem}
\begin{remark}
  The interpretation of for example the term $4368a^3m^5$ above is
  thus that imposing the hom-associativity identity \eqref{Eq:rule1}
  reduces by $4368$ the dimension of the space of plane operad
  elements that can be formed with $3$ operations $\alpha$ and $5$
  multiplications.
\end{remark}
\begin{proof}
  As shown above for $H_{1,2}$, but taking all of
  \eqref{Eq:rule1}--\eqref{Eq:rule5} into account, so that one
  instead considers $S_{5,3} \cup S_{4,4}$ and thus gets all terms of
  total degree $\leqslant 8$.
\end{proof}


\begin{acknowledgement}
This work was partially supported by The Swedish Foundation for International Cooperation
in Research and Higher Education (STINT), The Swedish Research Council,
The Royal Swedish Academy of Sciences, The Royal Physiographic Society in Lund, The Crafoord Foundation and The Letterstedtska
F{\"o}reningen.
\end{acknowledgement}

\mathversion{normal}

\message{c@mv@normal = \number\csname c@mv@normal\endcsname}
\message{c@mv@normal2 = \number\csname c@mv@normal2\endcsname}

\tracingstats=1

\end{document}